\documentclass[11pt]{amsart}
\usepackage{geometry}
\usepackage[dvips]{graphicx}
\usepackage{hyperref}

\usepackage{amscd,amsfonts,amssymb,amsmath,amsthm,latexsym}
\usepackage{mathrsfs}
\usepackage{lscape}
\usepackage{xcolor}
\usepackage{color}

\usepackage{geometry}
\usepackage{mathtools}
\usepackage{enumerate}
\usepackage{float}
\usepackage{color}
\usepackage{url}
\usepackage{array}
\usepackage{multirow}
\usepackage{bigstrut}
\usepackage{makecell}

\usepackage{graphics}
\usepackage[all]{xy}
\xyoption{curve}
\xyoption{import}
\xyoption{arc}
\xyoption{ps}

\numberwithin{equation}{section}

\newtheoremstyle{neu}
{20pt}       
{20pt}      
{\itshape}  
{}          
{\bfseries} 
{.}         
{.5em}      
{}          
\theoremstyle{neu}

    \newtheorem{thm}{Theorem}[section]
    \newtheorem{lemma}[thm]{Lemma}
    \newtheorem{coro}[thm]{Corollary}
    \newtheorem{prop}[thm]{Proposition}

    \newtheorem{question}[thm]{Question}
\theoremstyle{definition}
    \newtheorem{defi}[thm]{Definition}
    \newtheorem{ex}[thm]{Example}
    \newtheorem{remark}[thm]{Remark}
\theoremstyle{remark}

    \newtheorem{condition}{Condition}

\newcommand{\Z}{\mathbb{Z}}
\newcommand{\Q}{\mathbb{Q}}

\newcommand{\g}{\mathbf{g}}
\newcommand{\e}{\mathbf{e}}

\newcommand{\suchthat}{\ | \ }
\newcommand{\surf}{(\Sigma,\mathbb{M})}

\newcommand{\marked}{\mathbb{M}}
\newcommand{\punct}{\mathbf{P}}

\newcommand{\C}{\mathbb{C}}

\newcommand{\Gr}{\operatorname{Gr}}
\newcommand{\Hom}{\operatorname{Hom}}

\newcommand{\Image}{\operatorname{Im}}
\newcommand{\Ker}{\operatorname{Ker}}
\newcommand{\soc}{\operatorname{soc}}

\newcommand{\ZZ}{\mathbb{Z}}

\newcommand{\cR}{\mathcal{R}}

\newcommand{\alp}{\alpha}
\newcommand{\bet}{\beta}
\newcommand{\gam}{\gamma}

\newcommand{\df}{\colon}

\newcommand{\abs}[1]{\lvert #1\rvert}
\newcommand{\ebrace}[1]{\langle #1\rangle}

\newcommand{\rank}{\operatorname{rank}}
\newcommand{\maxid}{\mathfrak{m}}

\newcommand{\RA}[1]{\C\langle\hspace{-0.05cm}\langle #1\rangle\hspace{-0.05cm}\rangle}
\newcommand{\usualRA}[1]{\C\langle #1\rangle}
\newcommand{\jacobalg}[2]{\mathcal{P}(#1,#2)}
\newcommand{\modcat}{\operatorname{mod}}
\newcommand{\rep}{\operatorname{Rep}}
\newcommand{\decrep}{\operatorname{decRep}}
\newcommand{\decirr}{\operatorname{decIrr}}
\newcommand{\decirrsr}{\operatorname{decIrr^{\tau^{-}}}}
\newcommand{\Ext}{\operatorname{Ext}}
\newcommand{\vecspan}{\operatorname{span}}
\newcommand{\upper}{\mathcal{U}}
\newcommand{\calB}{\mathcal{B}}
\newcommand{\prin}{\operatorname{prin}}
\newcommand{\myid}{1\hspace{-0.125cm}1}

\newcommand{\bfd}{\mathbf{d}}
\newcommand{\bfv}{\mathbf{v}}
\newcommand{\GL}{\operatorname{GL}}
\newcommand{\calCC}{\mathcal{CC}}
\newcommand{\calF}{\mathcal{F}}
\newcommand{\nullity}{\operatorname{null}}
\newcommand{\End}{\operatorname{End}}

\begin{document}

\title[Generic CC-functions with coefficients \& surface cluster algebras]{Generic Caldero-Chapoton functions with coefficients and applications to surface cluster algebras}

\author{Christof Gei{\ss}}
\address{Christof Gei{\ss}\newline
Instituto de Matem\'aticas\newline
Universidad Nacional Aut{\'o}noma de M{\'e}xico\newline
Ciudad Universitaria\newline
04510 M{\'e}xico D.F.\newline
M{\'e}xico}
\email{christof.geiss@im.unam.mx}

\author{Daniel Labardini-Fragoso}
\address{Daniel Labardini-Fragoso\newline
Instituto de Matem\'aticas\newline
Universidad Nacional Aut{\'o}noma de M{\'e}xico\newline
Ciudad Universitaria\newline
04510 M{\'e}xico D.F.\newline
M{\'e}xico}
\email{labardini@im.unam.mx}

\author{Jan Schr\"oer}
\address{Jan Schr\"oer\newline
Mathematisches Institut\newline
Universit\"at Bonn\newline
Endenicher Allee 60\newline
53115 Bonn\newline
Germany}
\email{schroer@math.uni-bonn.de}

\date{\today}
\maketitle

\begin{abstract} We realize Derksen-Weyman-Zelevinsky's mutations of representations as densely-defined regular maps on representation spaces, and study the generic values of Caldero-Chapoton functions with coefficients, giving, for instance, a sufficient combinatorial condition for their linear independence. 

For a quiver with potential $(Q,S)$, we show that if $k$ is a vertex not incident to any oriented 2-cycle, then every generically $\tau$-reduced irreducible component $Z$ of any affine variety of (decorated) representations has a dense open subset $U$ on which Derksen-Weyman-Zelevinsky's mutation of representations $\mu_k$ can be defined consistently as a regular map to an affine variety of (decorated) representations of the Jacobian algebra of the mutated QP $\mu_k(Q,S)$.

Our techniques involve only basic linear algebra and elementary algebraic geometry, and do not require to assume Jacobi-finiteness. Thus, the paper yields a new and more general proof of the mutation invariance of generic Caldero-Chapoton functions, generalizing and providing a new natural geometric perspective on results of Derksen-Weyman-Zelevinsky and Plamondon. (For Jacobi-finite non-degenerate quivers with potential, this invariance was shown by Plamondon using the machinery of Ginzburg dg-algebras and $\Hom$-finite generalized cluster categories.)

We apply our results, together with results of Mills, Muller and Qin, to prove that for any choice of geometric coefficient systems, not necessarily of full rank, the cluster algebra associated to a possibly punctured surface with  at least two marked points on the boundary has the generic Caldero-Chapoton functions as a basis over the Laurent polynomial ring of coefficients. 
\end{abstract}
\setcounter{tocdepth}{1}
\tableofcontents
%
%
\section{Introduction and main results}\label{sec:intro}

{

The mutation theory of quivers with potential developed between 12 and 15 years ago by Derksen-Weyman-Zelevinsky \cite{DWZ1,DWZ2} has had a very deep impact in cluster algebras, representation theory, algebraic geometry and other areas, including quantum field theories in theoretical physics. After showing how quivers with potential and their representations can be mutated, they prove that every cluster monomial in a skew-symmetric cluster algebra can be written as the Caldero-Chapoton function of a quiver representation naturally associated to the given cluster monomial. In this paper we show that Derksen-Weyman-Zelevinsky's mutations of representations can be described consistently by regular maps defined on dense open subsets of the affine varieties of (decorated) representations. As an application, we give a sufficient combinatorial condition on skew-symmetric matrices $B$ that, without requiring $B$ to have full rank, guarantees the linear independence of the set of generic Caldero-Chapoton functions for arbitrary choices of geometric coefficients. We show that this condition is fulfilled in the case of skew-symmetric matrices arising from triangulations of surfaces with marked points and non-empty boundary. Combining this with results of Mills \cite{Mills1}, Muller \cite{Muller1,Muller2} and Qin \cite{Qin}, we prove that for (possibly punctured) surfaces with at least two marked points on the boundary, the corresponding cluster algebra with arbitrary geometric coefficients (not necessarily of full rank), the set of generic Caldero-Chapoton functions forms a basis, which we then call \emph{generic basis}.

A prominent role in the work of Derksen-Weyman-Zelevinsky is played by various integer-valued functions defined on decorated representations of Jacobian algebras of quivers with potential. We prove the upper semicontinuity of these functions on the affine varieties of (decorated) representations of arbitrary associative algebras. Namely, we show that the $\g$-vector, the $E$-invariant and the dimensions of the vector spaces $\Hom(-,\tau(\bullet))$ and $\Hom(\tau^-(-),\bullet)$ define upper semicontinuous functions.

In the case of a Jacobian algebra of a quiver with potential $(Q,S)$, we show that if $k$ is a vertex not incident to any oriented 2-cycle, then every generically $\tau$-reduced irreducible component $Z$ of any affine variety of (decorated) representations has a dense open subset $U$ on which Derksen-Weyman-Zelevinsky's mutation of representations $\mu_k$ can be defined consistently as a regular map to an affine variety of (decorated) representations of the Jacobian algebra of the mutated quiver with potential $\mu_k(Q,S)$. Using the aforementioned upper semicontinuities, we show that the Zariski closure of the union of the orbits of the decorated representations $\mu_k(\mathcal{M})$, for $\mathcal{M}\in U$, is a generically $\tau$-reduced component for $\mu_k(Q,S)$, thus proving that there is a mutation rule $Z\mapsto\mu_k(Z)$ at the level of generically $\tau$-reduced irreducible components, and that on any such component, the map that combines Derksen-Weyman-Zelevinsky's mutation of representations with the action of the general linear group is a densely-defined dominant regular map. 

Our techniques involve only basic linear algebra and elementary algebraic geometry of minors, and do not require to assume Jacobi-finiteness. Thus, the paper yields a new and more general proof of the mutation invariance of generic Caldero-Chapoton functions, generalizing and providing a new natural geometric perspective on results of Derksen-Weyman-Zelevinsky and Plamondon. (For Jacobi-finite non-degenerate quivers with potential, this invariance was shown by Plamondon using the machinery of Ginzburg dg-algebras and $\Hom$-finite generalized cluster categories.) 

We apply the results obtained throughout the paper, together with results of Mills \cite{Mills1}, Muller \cite{Muller1,Muller2} and Qin \cite{Qin}, to prove that for any choice of geometric coefficient systems, not necessarily of full rank, the cluster algebra associated to a possibly punctured surface with  at least two marked points on the boundary has the generic Caldero-Chapoton functions as a basis over the Laurent polynomial ring of coefficients.

}

\subsection{Assumptions and notation}\label{subsec:assumptions} All the notation and assumptions that we are about to establish will be kept all throughout the paper.

Let us write $[\ell]:=\{1,\ldots,\ell\}$ for any given positive integer $\ell$. Let $n$ be a positive integer, $m$ a non-negative integer, and $\widetilde{Q}=\{\widetilde{Q}_0,\widetilde{Q}_1,s,t\}$ any quiver with vertex set $\widetilde{Q}_0=[n+m]$. We write $a:s(a)\rightarrow t(a)$ for $a\in\widetilde{Q}_1$. Let $Q:=\widetilde{Q}|_{[n]}$ be the restriction of $\widetilde{Q}$ to the subset $[n]$ of $\widetilde{Q}_0=[n+m]$. Thus, $Q_0:=\widetilde{Q}_0$ and $Q_1:=\{a\in\widetilde{Q}_1\suchthat s(a),t(a)\in [n]\}$  (see \cite[Definition 8.8]{DWZ1}). Throughout the paper, the relation between $\widetilde{Q}$ and $Q$ will always be the one just described.

Let $\widetilde{\mathfrak{m}}_0\subseteq\C\langle\widetilde{Q}\rangle$ (resp. $\mathfrak{m}_0\subseteq\C\langle Q\rangle$, $\widetilde{\mathfrak{m}}\subseteq\RA{\widetilde{Q}}$, $\mathfrak{m}\subseteq\RA{Q} $) be the 2-sided ideal generated inside $\C\langle\widetilde{Q}\rangle$ (resp. $\C\langle Q\rangle$, $\RA{\widetilde{Q}}$, $\RA{Q}$)  by the arrow set of $\widetilde{Q}$ (resp. of $Q$, of $\widetilde{Q}$, of $Q$). The ideal $\widetilde{\mathfrak{m}}$ (resp. $\mathfrak{m}$) induces a topology on $\RA{\widetilde{Q}}$ (resp. $\RA{Q}$) having the non-negative powers of $\widetilde{\mathfrak{m}}$ (resp. $\mathfrak{m}$) as fundamental system of open neighborhoods around zero. Under this topology, $\C\langle\widetilde{Q}\rangle$ (resp. $\C\langle Q\rangle$) is dense in $\RA{\widetilde{Q}}$ (resp. in $\RA{Q}$).

Let $\widehat{\rho}:\RA{\widetilde{Q}}\rightarrow\RA{Q}$ be the continuous $\C$-algebra homomorphism that acts as the identity on the set of idempotents $\{e_j\suchthat j\in\widetilde{Q}_0=Q_0\}$ and is given by the rule
\begin{equation}\label{eq:restriction-homomorphism-rule}
\widehat{\rho}(a)=\begin{cases}
a & \text{if $a\in Q_1$;}\\
0 & \text{otherwise;}
\end{cases}
\end{equation}
for $a\in\widetilde{Q}_1$.
Then $\widehat{\rho}(\widetilde{\mathfrak{m}})=\mathfrak{m}$ and $\widehat{\rho}(\C\langle\widetilde{Q}\rangle)=\C\langle Q\rangle$. The restriction $\C\langle\widetilde{Q}\rangle\rightarrow \C\langle Q\rangle$ of $\widehat{\rho}$ will be denoted as $\rho$. It satisfies
$\rho(\widetilde{\mathfrak{m}}_0)=\mathfrak{m}_0$.

Let $\widetilde{\mathcal{J}}$ be a closed two sided ideal of $\RA{\widetilde{Q}}$ contained in $\widetilde{\mathfrak{m}}^2$, $\mathcal{J}:=\widehat{\rho}(\widetilde{\mathcal{J}})\subseteq\mathfrak{m}^2$, and set 
$$
\widetilde{\Lambda}:=\RA{\widetilde{Q}}/\widetilde{\mathcal{J}} 
\text{\qquad and \qquad}
\Lambda:=\RA{Q}/\mathcal{J}.
$$ 
The results of Sections \ref{sec:g-vectors}, \ref{sec:CC-map-and-algebra}, \ref{subsec:generically-tau--reduced-comps} and \ref{subsec:upper-semicontinuity-of-E-invariant} will be shown to hold without assuming that $\widetilde{\Lambda}$ or $\Lambda$ is finite-dimensional or Jacobian.

\begin{remark} By \cite[Lemma 2.1]{CLS}, the results to be established in this paper are valid even without the assumption that $\widetilde{\mathcal{J}}$ be closed in $\RA{\widetilde{Q}}$.
\end{remark}

The vertex set of the quiver $Q$ is $[n+m]$, but $Q$ does not have arrows incident to any of the vertices $n+1,\ldots,n+m$. By $Q_{\prin}$ we will denote the quiver obtained from $Q$ by first deleting the vertices $n+1,\ldots,n+m$, and then adding $n$ vertices $1',2',\ldots,n'$ and exactly one arrow $j\rightarrow j'$ for each $j\in[n]$. By $\Lambda_{\prin}$ we will denote the quotient $\RA{Q_{\prin}}/\mathcal{J}_{\prin}$, where $\mathcal{J}_{\prin}$ is the 2-sided ideal generated by $\mathcal{J}$ inside $\RA{Q_{\prin}}$. Notice that $\mathcal{J}_{\prin}\subseteq \mathfrak{m}^2_{\prin}$, where $\mathfrak{m}_{\prin}$ is the 2-sided ideal of $\RA{Q_{\prin}}$ generated by the arrows of $Q_{\prin}$.

Let  $\widetilde{S}\in\RA{\widetilde{Q}}$ be a potential on $\widetilde{Q}$. Set $S:=\widehat{\rho}(\widetilde{S})$. Thus, $(Q,S)=(\widetilde{Q}|_{[n]},\widetilde{S}|_{[n]})$ is the restriction of $(\widetilde{Q},\widetilde{S})$ to $[n]=\{1,\ldots,n\}$, cf. \cite[Definition 8.8]{DWZ1}. Notice that, by \cite[Corollary 22]{LF1}, if $(\widetilde{Q},\widetilde{S})$ is non-degenerate, then $(Q,S)$ is non-degenerate too. In Sections \ref{subsec:cluster-variables}, \ref{sec:lin-ind-for-surfaces}, \ref{sec:spanning-the-CC-algebra} and \ref{sec:B-is-a-basis-of-upper-cluster-algebra} we will work under the assumption that $(\widetilde{Q},\widetilde{S})$ is non-degenerate. We will never suppose that $(\widetilde{Q},\widetilde{S})$ is Jacobi-finite, though. In Subsection \ref{subsec:behavior-of-irred-comps-under-mutation} we will not suppose non-degeneracy nor Jacobi-finiteness.

Let $J(\widetilde{S})\subseteq\RA{\widetilde{Q}}$  and $J(S)\subseteq\RA{Q}$ be the Jacobian ideals of $\widetilde{S}$ and $S$, i.e., the topological closures of the 2-sided ideals generated respectively in $\RA{\widetilde{Q}}$ and $\RA{Q}$ by $\{\partial_a(\widetilde{S})\suchthat a\in\widetilde{Q}_1\}$ and $\{\partial_a(S)\suchthat a\in Q_1\}$, and let $\jacobalg{\widetilde{Q}}{\widetilde{S}}=\RA{\widetilde{Q}}/J(\widetilde{S})$ and $\jacobalg{Q}{S}=\RA{Q}/J(S)$ be the Jacobian algebras of $(\widetilde{Q},\widetilde{S})$ and $(Q,S)$, respectively. Lemma \ref{lemma:extend-reps-putting-zeros} states in particular that $J(S)$ coincides with $\widehat{\rho}(J(\widetilde{S}))$, which means that the relation between $\jacobalg{\widetilde{Q}}{\widetilde{S}}$ and $\jacobalg{Q}{S}$ is precisely the relation between $\widetilde{\Lambda}$ and $\Lambda$ above.

Following \cite{CLS}, even when $\widetilde{Q}$ or $Q$ fails to be loop-free or 2-acyclic, we define $\widetilde{B}\in\ZZ^{(n+m)\times n}$ to be the matrix whose $(i,j)^{\operatorname{th}}$ entry is $\widetilde{b}_{ij}:=|\{a\in\widetilde{Q}\suchthat i\overset{a}{\leftarrow}j\}|-|\{a\in\widetilde{Q}\suchthat i\overset{a}{\rightarrow} j\}|$ for $(i,j)\in[n+m]\times[n]$. Notice that if $\widetilde{Q}$ is 2-acyclic, then for $(i,j)\in[n+m]\times[n+m]$ the number of arrows from $j$ to $i$ can be fully recovered from $\widetilde{B}$ if and only if at least one of $i$ and $j$ belongs to $[n]$.

Define $B=(b_{ij})_{i,j}$ to be the top $n\times n$ submatrix of $\widetilde{B}$. Since $n+1,\ldots,n+m$, are isolated vertices in $Q$, the $(n+m)\times n$ signed-adjacency matrix of $Q$ is $\left[\begin{array}{c}B\\ \mathbf{0}\end{array}\right]$. It will never be assumed that $\widetilde{B}$ or $B$ is a full-rank matrix.

Finally, notice that the rectangular $(n+n)\times n$ matrix corresponding to $Q_{\prin}$ is
$B_{\prin}=\left[\begin{array}{c}
B\\
\myid_{n\times n}
\end{array}
\right]$,
where $\myid_{n\times n}\in\ZZ^{n\times n}$ is the identity matrix.

Let us recall the definitions of cluster algebra and upper cluster algebra \cite{FZ1,FZ2,BFZ3,FZ4}. We assume that the reader is familiar with the notion of seeds and with seed mutations.

\begin{defi}\label{def:precise-def-of-cluster-alg-and-upper-cluster-alg} Let $\calF$ be the field of rational functions in $n+m$ indeterminates with coefficients in $\Q$, and $\mathbf{x}=(x_1,\ldots,x_n,x_{n+1},\ldots,x_{n+m})$ a tuple of elements of $\calF$ algebraically independent over $\Q$. As above, let $\widetilde{B}\in\ZZ^{(n+m)\times n}$ be a rectangular matrix whose top $n\times n$ submatrix is skew-symmetric. Take $(\widetilde{B},\mathbf{x})$ as an initial seed, with $x_{n+1},\ldots,x_{n+m}$ frozen, so that only seed mutations with respect to the first $n$ indices $1,\ldots,n$ are allowed, let $\mathcal{S}$ be the set whose elements are all the clusters $\mathbf{u}=(u_1,\ldots,u_n,x_{n+1},\ldots,x_{n+m})$ that belong to seeds that can be obtained from $(\widetilde{B},\mathbf{x})$ by finite sequences of allowed seed mutations, and $\mathcal{X}$ be the set of all rational functions that appear in the clusters belonging to $\mathcal{S}$. 
\begin{enumerate}
\item\label{item:def-cluster-alg-coefficient-ring-R} Declare our ground ring to be the Laurent polynomial ring $R:=\Q[x_{n+1}^{\pm1},\ldots,x_{n+m}^{\pm1}]$.
\item The \emph{cluster algebra of $\widetilde{B}$} is the $R$-subalgebra $\mathcal{A}(\widetilde{B})$ that $\mathcal{X}$ generates inside $\calF$.
\item The \emph{upper cluster algebra of $\widetilde{B}$} is the $R$-subalgebra of $\calF$ defined as
\begin{eqnarray*}
\mathcal{U}(\widetilde{B})&:=& 
\bigcap_{\mathbf{u}\in\mathcal{S}}R[u_{1}^{\pm1},\ldots,u_{n}^{\pm1}].
\end{eqnarray*}
\item We set $R_{\prin}:=\Q[x_{n+1}^{\pm1},\ldots,x_{n+n}^{\pm1}]$.
\end{enumerate}
\end{defi}

\subsection{$\g$-vectors and extended $\g$-vectors} 

Following \cite{DWZ1,DWZ2,CLS}, we will work with decorated representations. A \emph{decorated representation of $\Lambda$} is a pair $\mathcal{M}=(M,\bfv)$ consisting of a finite-dimensional representation $M$ of $\Lambda$, and a non-negative integer vector $\bfv=(v_j)_{j\in\widetilde{Q}_0}\in\ZZ_{\geq 0}^{\widetilde{Q}_0}$. A decorated representation $\mathcal{M}=(M,\bfv)$ of $\Lambda$ is \emph{supported in $[n]$} if $M_{j}=0$ and $v_j=0$ for all $j\in\widetilde{Q}_0\setminus[n]=\{n+1,\ldots,n+m\}$.

For a decorated representation $\mathcal{M}=(M,\bfv)$ of $\Lambda$ which is supported in $[n]$ we define the \emph{$\g$-vector} $\g_{\mathcal{M}}^\Lambda=(g_{i})_{i=1}^n\in\ZZ^{n}$  and the  \emph{extended $\g$-vector} $\widetilde{\g}_{\mathcal{M}}^{\widetilde{\Lambda}}=(\widetilde{g}_{i})_{i=1}^{n+m}\in\ZZ^{n+m}$, according to the following formulas:
\begin{eqnarray*}
g_i &:=& -\dim_{\C}(\Hom_{\Lambda}(S(i),M))+\dim_{\C}(\Ext^1_{\Lambda}(S(i),M))+v_i \ \ \ \ \   \text{for $i\in[n]$}\\
\widetilde{g}_i &:=& -\dim_{\C}(\Hom_{\widetilde{\Lambda}}(S(i),M))+\dim_{\C}(\Ext^1_{\widetilde{\Lambda}}(S(i),M))+v_i \ \ \ \ \   \text{for $i\in[n+m]$.}
\end{eqnarray*}

In Section \ref{sec:g-vectors} we prove:

\begin{prop}\label{prop:restricted-g-vector-vs-extended-g-vector-intro}
The vector formed with the first $n$ entries of $\widetilde{\g}_{\mathcal{M}}^{\widetilde{\Lambda}}$ is precisely $\g_{\mathcal{M}}^\Lambda$.
\end{prop}

This result is stated as Proposition \ref{prop:restricted-g-vector-vs-extended-g-vector} below. To prove it, we first establish for any $A$-module $M$ a very concrete relation between the minimal injective presentations of $M$ in $A$-$\modcat$ and $\widetilde{A}$-$\modcat$ in the particular case where instead of $\widetilde{\Lambda}$ and $\Lambda$ we have quotients $\widetilde{A}=\C\langle\widetilde{Q}\rangle/\widetilde{J}$ and $A=\C\langle Q\rangle/\rho(\widetilde{J})$ with $\widetilde{\mathfrak{m}}_0^\ell\subseteq \widetilde{J}\subseteq\widetilde{\mathfrak{m}}_0^2$ for some $\ell \geq 2$. This is done in Subsection \ref{subsec:admissible-quotients}. 
We then prove Proposition \ref{prop:restricted-g-vector-vs-extended-g-vector-intro} for our more general $\widetilde{\Lambda}$ and $\Lambda$ using truncation techniques inspired in \cite{CLS}, see Subsection \ref{subsec:arbitrary-quotients}. In Subsection \ref{subsec:Jacobian-algebras} we show that the relation between $\jacobalg{\widetilde{Q}}{\widetilde{S}}$ and $\jacobalg{Q}{S}$ is precisely the relation between $\widetilde{\Lambda}$ and $\Lambda$, implying that Proposition \ref{prop:restricted-g-vector-vs-extended-g-vector-intro} can be applied to the case of Jacobian algebras.

\subsection{A partial order on $\ZZ^n$}\label{subsec:partial-order} Inspired by \cite[Proof of Proposition 4.3]{CLS}, in Section \ref{sec:partial-order} below we show that if a skew-symmetric matrix $B\in\ZZ^{n\times n}$ satisfies $\Ker(B)\cap\ZZ^n_{\geq 0}=0$, then the rule
$$
\left[\mathbf{a}\preceq_B\mathbf{b}\Longleftrightarrow\mathbf{a}-\mathbf{b}\in B(\ZZ^n_{\geq 0})\right]
$$
automatically defines a partial order on $\ZZ^n$. We also list a series of sufficient conditions that guarantee that a skew-symmetric matrix $B$ satisfies $\Ker(B)\cap\ZZ^n_{\geq 0}=0$.

\subsection{The Caldero-Chapoton algebra with coefficients}

Let $\calF$ be the field of rational functions in $n+m$ indeterminates with coefficients in $\Q$, and $(x_1,\ldots,x_n,x_{n+1},\ldots,x_{n+m})$ a tuple of elements of $\calF$ algebraically independent over $\Q$. For any decorated representation $\mathcal{M}=(M,\bfv)$ of $\Lambda$ which is supported in $[n]$, the \emph{Caldero-Chapoton function of $\mathcal{M}$ with coefficients extracted from $\widetilde{\Lambda}$}  is 
\begin{eqnarray*}
CC_{\widetilde{\Lambda}}(\mathcal{M})&:=& \sum_{\e\in\ZZ_{\geq 0}^n}\chi(\Gr_{\e}(M))\mathbf{x}^{\widetilde{B}\e+\widetilde{\g}_{\mathcal{M}}^{\widetilde{\Lambda}}},
\end{eqnarray*}
where, as usual, $\mathbf{x}^{\mathbf{f}}:=x_1^{f_1}\cdots x_{n}^{f_n}x_{n+1}^{f_{n+1}}\cdots x_{n+m}^{f_{n+m}}$ for $\mathbf{f}=(f_1,\ldots,f_n,f_{n+1},\ldots,f_{n+m})\in\ZZ^{n+m}$.
For a set $\mathscr{M}$ of decorated representations of $\Lambda$ supported in $[n]$, we write
$$
CC_{\widetilde{\Lambda}}(\mathscr{M}):=\{CC_{\widetilde{\Lambda}}(\mathcal{M})\suchthat \mathcal{M}\in \mathscr{M}\}.
$$
The \emph{Caldero-Chapoton algebra of $\Lambda$ with coefficients extracted from $\widetilde{\Lambda}$} is the $R$-submodule of $R[x_1^{\pm1},\ldots,x_n^{\pm1}]$ spanned by $CC_{\widetilde{\Lambda}}(\decrep(\Lambda))=\{CC_{\widetilde{\Lambda}}(\mathcal{M})\suchthat \mathcal{M}$ is a decorated representation of $\Lambda$ supported in $[n]\}$.

Our first main result is the following non-trivial refinement of \cite[Proposition 4.3]{CLS}. 

\begin{thm}\label{thm:criterion-for-lin-ind-of-CC-functions-intro} Take a set $\mathscr{M}$ of decorated representations of $\Lambda$ supported in $[n]$. Suppose that the following two conditions are satisfied:
\begin{enumerate}
\item $\Ker(B)\cap\mathbb{Z}^{n}_{\geq 0}=0$;
\item for any $\mathcal{M},\mathcal{M}'\in\mathscr{M}$, $\mathcal{M}\neq\mathcal{M}'\Longrightarrow\g_{\mathcal{M}}^{\Lambda}\neq\g_{\mathcal{M}'}^{\Lambda}$.
\end{enumerate}
Then $CC_{\widetilde{\Lambda}}(\mathscr{M})$ is an $R$-linearly independent subset of the Laurent polynomial ring $R[x_1^{\pm1},\ldots,x_n^{\pm1}]$.
\end{thm}

\begin{remark}\label{rem:on-the-criterion-for-lin-ind-of-CC-functions-intro}\begin{itemize}\item In Theorem \ref{thm:criterion-for-lin-ind-of-CC-functions-intro}, stated as Theorem \ref{thm:criterion-for-lin-ind-of-CC-functions} below, none of the matrices $\widetilde{B}$ and $B$ is assumed to have full rank. See Subsection \ref{subsec:assumptions} and Remark \ref{rem:the-5-conditions-on-B}.
\item The main difference between \cite[Proposition~4.3]{CLS} and Theorem \ref{thm:criterion-for-lin-ind-of-CC-functions-intro} is that the former is stated and proved only in the absence of coefficients, that is, when $m=0$, so Theorem~\ref{thm:criterion-for-lin-ind-of-CC-functions-intro} is really more general. The inspiration comes directly from \cite[Proposition~4.3]{CLS}~though.
\item Assuming that $B$ or $\widetilde{B}$ has full rank, Fu-Keller \cite{FK}, Plamondon \cite{Plamondon-generic} and Qin \cite{Qin} have proved similar linear independence statements, for wide classes of Jacobian algebras $\jacobalg{Q}{S}$, $\jacobalg{\widetilde{Q}}{\widetilde{S}}$, of non-degenerate quivers with potential satisfying further conditions (e.g. Jacobi-finiteness, or injective reachability).
\end{itemize}
\end{remark}

Our second main result, stated as Theorem \ref{thm:spanning-set-in-princ-coeffs=>spanning-set-in-arbitrary-coeffs} in the body of the paper, is:

\begin{thm}\label{thm:spanning-set-in-princ-coeffs=>spanning-set-in-arbitrary-coeffs-intro} 
Take any set $\mathscr{M}$ of decorated representations of $\Lambda$ supported in $[n]$, and any decorated representation $\mathcal{N}$ of $\Lambda$ supported in $[n]$. If $CC_{\Lambda_{\prin}}(\mathcal{N})$ belongs to the $R_{\prin}$-submodule of $R_{\prin}[x_{1}^{\pm1},\ldots,x_{n}^{\pm1}]$ spanned by 
$
CC_{\Lambda_{\prin}}(\mathscr{M})
$
(see Definition \ref{def:precise-def-of-cluster-alg-and-upper-cluster-alg}),
then $CC_{\widetilde{\Lambda}}(\mathcal{N})$ belongs to the $R$-submodule of $R[x_{1}^{\pm1},\ldots,x_{n}^{\pm1}]$ spanned by 
$
CC_{\widetilde{\Lambda}}(\mathscr{M}).
$
Consequently, if 
$CC_{\Lambda_{\prin}}(\mathscr{M})$
spans the Caldero-Chapoton algebra $\calCC_{\Lambda_{\prin}}(\Lambda)$ over the ground ring $R_{\prin}$, then 
$CC_{\widetilde{\Lambda}}(\mathscr{M})$
spans the Caldero-Chapoton algebra $\calCC_{\widetilde{\Lambda}}(\Lambda)$ over the ground ring $R$.
\end{thm}

\begin{remark}\label{rem:Qin-correction-technique} Fan Qin has pointed out to us that, provided the binary relation $\preceq_B$ from Subsection \ref{subsec:partial-order} above and Section \ref{sec:partial-order} below is a partial order for $\ZZ^n$, Theorem \ref{thm:spanning-set-in-princ-coeffs=>spanning-set-in-arbitrary-coeffs-intro} can be deduced from his \emph{correction technique} theorem, stated in \cite[Theorem 9.2]{Qin-qcharacters} and \cite[Theorem 4.2.1]{Qin-triangular} under a full-rank assumption.
\end{remark}

\subsection{Irreducible components of representation spaces and their mutations}\label{subsec:irred-comps}

Following \cite{CLS}, for each $(\bfd,\bfv)\in\ZZ^{n}\times\ZZ^{n}$ we define the \emph{decorated representation space}
$$
\decrep(\Lambda,\bfd,\bfv):=\rep(\Lambda,\bfd)\times\{\bfv\},
$$
where
$$
\rep(\Lambda,\bfd):=\{M\in\prod_{a\in Q_1}\Hom_{\C}(\C^{d_{s(a)}},\C^{d_{t(a)}})\suchthat \text{$M$ is annihilated by every element of $\mathcal{J}$}\}
$$
is the usual affine variety of representations of $\Lambda$ with dimension vector $\bfd$. It is clear that $\decrep(\Lambda,\bfd,\bfv)$ and $\rep(\Lambda,\bfd)$ are canonically isomorphic as affine varieties. Following Derksen-Weyman-Zelevinsky \cite[Equations (7.4) and (7.6)]{DWZ2}, for decorated representations $\mathcal{M}=(M,\bfv)\in \decrep(\Lambda,\bfd,\bfv)$ and $\mathcal{N}=(N,\bfv')\in\decrep(\Lambda,\bfd',\bfv')$, we define the \emph{$E$-invariant} as
\begin{eqnarray*}
E_{\Lambda}(\mathcal{M},\mathcal{N}) &:=& \dim_{\C}(\Hom_{\Lambda}(M,N))+\langle\bfd,\g_{\mathcal{N}}^{\Lambda}\rangle\\
E_{\Lambda}(\mathcal{M}) &:=& E_{\Lambda}(\mathcal{M},\mathcal{M}).
\end{eqnarray*}

Recall 
that the assignments 
$$
\mathcal{M}\mapsto\g_{\mathcal{M}}^{\Lambda},
\quad 
\mathcal{M} \mapsto \widetilde{\g}_{\mathcal{M}}^{\widetilde{\Lambda}},
\quad
\mathcal{M}\mapsto E_{\Lambda}(\mathcal{M})
\text{\quad and \quad}
\mathcal{M}\mapsto CC_{\widetilde{\Lambda}}(\mathcal{M})
$$ 
are constructible functions on each $\decrep(\Lambda,\bfd,\bfv)$, and that $\mathcal{M}\mapsto-\dim(\GL_{\bfd}\cdot\mathcal{M})$ defines an upper semicontinuous function on $\decrep(\Lambda,\bfd,\bfv)$, see for example \cite{CLS}. As is well known, these facts imply that every irreducible component $Z$ of $\decrep(\Lambda,\bfd,\bfv)$ has an open dense subset $U\subseteq Z$ such that each of the functions 
$\g_{-}^{\Lambda}$, $\widetilde{\g}_{-}^{\widetilde{\Lambda}}$, $E_{\Lambda}(-)$ and $CC_{\widetilde{\Lambda}}(-)$ is constant on $U$, and every $\mathcal{N}\in U$ satisfies
$$
\dim(Z)-\dim(\GL_{\bfd}\cdot\mathcal{N})=\min\{\dim(Z)-\dim(\GL_{\bfd}\cdot\mathcal{M})\suchthat\mathcal{M}\in Z\}.
$$
These constant values are denoted $\g_{Z}^{\Lambda}$, $\widetilde{\g}_{Z}^{\widetilde{\Lambda}}$, $E_{\Lambda}(Z)$, $CC_{\widetilde{\Lambda}}(Z)$ and $c_{\Lambda}(Z)$, and called the \emph{generic $\g$-vector of $Z$}, the \emph{generic extended $\g$-vector of $Z$}, the \emph{generic $E$-invariant of $Z$}, the \emph{generic Caldero-Chapoton function of $Z$} and the \emph{generic codimension of orbits contained in $Z$}. 

According to \cite[Section 8.1]{GLS3} and \cite[Lemma 5.2]{CLS},  for every irreducible component $Z\in \decirr(\Lambda)$ one always has $$
c_{\Lambda}(Z)\leq E_{\Lambda}(Z).
$$ 
Following Geiss-Leclerc-Schr\"{o}er \cite{GLS3}, we pay special attention to those irreducible components that satisfy $c_{\Lambda}(Z)= E_{\Lambda}(Z)$, calling them \emph{generically $\tau^-$-reduced irreducible components}\footnote{Geiss-Leclerc-Schr\"{o}er called them \emph{strongly reduced irreducible components}.}. Accordingly, setting
\begin{eqnarray*}
\decirrsr(\Lambda)&:=&\{Z\suchthat Z \ \text{is a generically $\tau^-$-reduced irreducible component}\\
&& \phantom{ \{Z\suchthat }\text{of $\decrep(\Lambda,\bfd,\bfv)$ for some $(\bfd,\bfv)\in\ZZ^{n}_{\geq 0}\times\ZZ^{n}_{\geq 0}\}$},\\
\text{we call} \hspace{1.25cm} \mathcal{B}_{\widetilde{\Lambda}}(\Lambda)&:=&\{CC_{\widetilde{\Lambda}}(Z)\suchthat Z\in \decirrsr(\Lambda)\}
\end{eqnarray*}
the \emph{set of generically $\tau^-$-reduced Caldero-Chapoton functions}, or simply, \emph{the set of generic Caldero-Chapoton functions}.

Our use of the term \emph{generically $\tau^-$-reduced} comes from the fact that, according to \cite[Proposition 3.5]{CLS}, if $\mathcal{M}=(M,\bfv)$ and $\mathcal{N}=(N,\mathbf{w})$ are decorated representations of $\Lambda$ supported in $[n]$, and $p$ is an integer greater than $\dim_{\C}(M)$ and $\dim_{\C}(N)$, then we have the equality 
$$
E_{\Lambda}(\mathcal{M},\mathcal{N})=\dim_{\C}\Hom_{\Lambda_p}(\tau^-_{\Lambda_p}(N),M)+\langle\underline{\dim}(M),\mathbf{w}\rangle, 
$$
where $\Lambda_p$ is the $p^{\operatorname{th}}$ truncation of $\Lambda$ (this homological interpretation of the $E$-invariant was first given by Derksen-Weyman-Zelevinsky for certain subclass of the class of finite-dimensional Jacobian algebras, see \cite[Section 10]{DWZ2} and, more specifically, \cite[Equation (10.3), Theorem 10.5 and Corollary 10.9]{DWZ2}). Plamondon has proved in \cite[end of Section 2.1]{Plamondon-generic} that if $\Lambda$ is a finite-dimensional Jacobian algebra, then 
$$
\dim_{\C}\Hom_{\Lambda}(\tau^-_{\Lambda}(M),M)=\dim_{\C}\Hom_{\Lambda}(M,\tau_{\Lambda}(M)).
$$ 
Consequently, if, following \cite{GLFS-schemes}, one decides to call $Z\in\decirr(\Lambda)$ \emph{generically $\tau$-reduced} whenever one has the equality 
$$
c_{\Lambda}(Z)=\dim_{\C}\Hom_{\Lambda}(M,\tau_{\Lambda}(M))+\langle\underline{\dim}(M),\mathbf{v}\rangle
$$ 
for all $\mathcal{M}=(M,\bfv)$ in an open dense subset of $Z$, then the two concepts, \emph{generic $\tau$-reducedness} and \emph{generic $\tau^-$-reducedness}, coincide, provided $\Lambda$ is finite-dimensional and Jacobian. To the best of our knowledge, it is not known whether the two notions coincide or not for arbitrary associative algebras.

\begin{remark}\label{rem:some-history-of-generic-CC-functions}
\begin{enumerate}

\item 
Dupont \cite{Dupont} was the first to consider generic values of the CC-map. He did so for acyclic cluster algebras 
{
using only acyclic seeds.
In this case, the corresponding Jacobian algebras are path algebras
of quivers without relations, and all representation spaces
are just affine spaces, hence irreducible.
Dupont conjectured that the generic CC-functions form a basis for
acyclic cluster algebras, and he essentially proved this for some tame quivers.
}

\item\label{item:rem-some-some-history-of-generic-CC-functions-GLS}
{
Afterwards, Geiss-Leclerc-Schr\"{o}er \cite{GLS3} considered
generically $\tau^-$-reduced CC-functions for quivers
with Jacobi-finite non-degenerate potential arising in a Lie-theoretic context.
Their definition of these functions makes sense for arbitrary
quivers with
Jacobi-finite potential.
The main idea was to discard most irreducible components and
only look at the generic CC-functions associated to generically $\tau^-$-reduced components.
This allows to extend Dupont's conjecture to arbitrary 
seeds for arbitrary quivers with Jacobi-finite non-degenerate potential.
For a large class of algebras the generically $\tau^-$-reduced CC-functions form a basis of the associated cluster algebras, see \cite{GLS3} for details.
}

\item\label{item:rem-some-some-history-of-generic-CC-functions-Plamondon}
{
Plamondon \cite{Plamondon-generic} defined a
seemingly different set of CC-functions.
He worked with arbitrary quivers with  Jacobi-finite non-degenerate potential.
Taking any pair $(P_1,P_0)$ of projective modules
over the Jacobian algebra, he looked at the CC-function of the 
generic cokernel of homomorphisms $P_1 \to P_0$.
He proved that these CC-functions
coincide with the set of generically $\tau^-$-reduced CC-functions defined by Geiss-Leclerc-Schr\"oer.
Roughly speaking,
this yields a natural parametrization of these functions by 
$\g$-vectors, which are elements in $\Z^n$ and correspond to
the pairs $(P_1,P_0)$ of projectives.
Plamondon also proved the mutation invariance of 
the CC-functions arising from generic cokernels.
}

\item
{
For quivers with non-degenerate potential $(\widetilde{Q},\widetilde{S})$ whose mutable part is \emph{injective-reachable}
and whose exchange matrix $\widetilde{B}$ has full rank,
Fan Qin \cite{Qin} showed that the set of generically $\tau^-$-reduced 
CC-functions forms a basis of the corresponding upper cluster algebra. 
}

\item
Coefficient-free generic CC-functions for arbitrary quotients of complete path algebras with finitely many vertices were later studied by Cerulli-Labardini-Schr\"{o}er in \cite{CLS}; Labardini-Velasco \cite{LF-V} showed that with this point of view, one can obtain a basis of the generalized cluster algebras associated by Chekhov-Shapiro \cite{Chekhov-Shapiro} to polygons with one orbifold point of order 3. 

\end{enumerate}
\end{remark}

{
It is well known that the functions
$$
M \mapsto \dim\Ext_\Lambda^i(M,M) 
\text{\qquad and \qquad}
(M,N) \mapsto \dim \Ext_\Lambda^i(M,N)
$$
are upper semicontinuous for $i=0,1$, 
where $\Ext^0_A(-,\bullet) := \Hom_A(-,\bullet)$.
}
In Subsection \ref{subsec:upper-semicontinuity-of-E-invariant} we prove that each component of the (extended) $\g$-vector defines an upper semicontinuous function $\decrep(\Lambda,\bfd,\bfv)\rightarrow\ZZ$,
see Theorem \ref{thm:E-invariant-is-upper-semicontinuous}. 
From this we deduce the following theorem which says that 
the $E$-invariant is also upper semicontinuous.

{
\begin{thm}\label{thm:E-invariant-is-upper-semicontinuous-intro}
For all decorated dimension vectors $(\bfd,\bfv), (\bfd',\bfv')\in\ZZ_{\geq0}^{n}\times\ZZ_{\geq0}^n$,
the functions
\begin{align*}
\decrep(\Lambda,\bfd,\bfv) &\to \Z 
&\text{and}&&
\decrep(\Lambda,\bfd,\bfv) \times \decrep(\Lambda,\bfd',\bfv')&\to \Z 
\\
\mathcal{M} &\mapsto E_{\Lambda}(\mathcal{M},\mathcal{M})
&&&
(\mathcal{M},\mathcal{N}) &\mapsto E_{\Lambda}(\mathcal{M},\mathcal{N})
\end{align*}
are upper semicontinuous.
In particular, if $A$ is a finite-dimensional $\C$-algebra, then for all dimension vectors $\bfd,\bfd'\in\ZZ_{\geq0}^{n}$ the functions
\begin{align*}
\rep(A,\bfd) &\to \Z 
&\text{and}&&
\rep(A,\bfd) \times \rep(A,\bfd')&\to \Z 
\\
M &\mapsto \dim\Hom_A(\tau^-_A(M),M) 
&&&
(M,N) &\mapsto  \dim \Hom_A(\tau^-_A(N),M)
\end{align*}
as well as the functions
\begin{align*}
\rep(A,\bfd) &\to \Z 
&\text{and}&&
\rep(A,\bfd) \times \rep(A,\bfd')&\to \Z 
\\
M &\mapsto \dim\Hom_A(M,\tau_A(M)) 
&&&
(M,N) &\mapsto  \dim \Hom_A(N,\tau_A(M))
\end{align*}
are upper semicontinuous.
\end{thm}
}

We expect that these upper semicontinuity results can be useful in many different contexts.
They play an essential role in Subsection~\ref{subsec:behavior-of-irred-comps-under-mutation},
where we study the behavior of generically $\tau^-$-reduced irreducible components under Derksen-Weyman-Zelevinsky's mutation of representations.
This leads to our fourth main result:

{
\begin{thm}\label{thm:premut-&-mut-of-reps-are-locally-regular-intro} 
Let $\mathcal{P} =  \jacobalg{Q}{S}$.
Assume that $k \in [n]$ is a vertex such that $Q$ does not have $2$-cycles incident to $k$,
and set $\mathcal{P}' := \jacobalg{\mu_k(Q}{S)}$.
Then the following hold:
\begin{itemize}

\item[(i)]
For every irreducible component $Z \subseteq \decrep(\mathcal{P},\bfd,\bfv)$ 
there exist a dense open subset $U \subseteq Z$ and a
regular function
$$
\nu_U:U\rightarrow\decrep(\mathcal{P}',\bfd',\bfv')
$$ 
such that for each $(M,\bfv)\in U$, the decorated representation $\nu_U(M,\bfv)$ is isomorphic
to Derksen-Weyman-Zelevinsky's mutation $\mu_k(M,\bfv)$, where $(\bfd',\bfv')\in\ZZ^{n}_{\geq 0}\times \ZZ^{n}_{\geq 0}$ is defined by \eqref{eq:mutated-decorated-dim-vector}.

\item[(ii)]
Let 
$$
\nu_U'\df \GL_{\bfd'} \times U \to \decrep(\mathcal{P}',\bfd',\bfv')
$$
be the regular function 
defined by $(g,u) \mapsto g \cdot \nu_U(u)$.
If $Z$ is generically $\tau^-$-reduced, then
$$
\mu_k(Z) := \overline{\Image(\nu_U')}
$$
is
a generically $\tau^-$-reduced component of
$\decrep(\mathcal{P}',\bfd',\bfv')$.

\item[(iii)]
The map
$$
\decirrsr(\mathcal{P}) \to \decirrsr(\mathcal{P}') 
$$
defined by $Z \mapsto \mu_k(Z)$
is bijective, and we have $\mu_k(\mu_k(Z)) = Z$ for all 
$Z \in \decirrsr(\mathcal{P})$.

\end{itemize}
\end{thm}

{
Thus Theorem~\ref{thm:premut-&-mut-of-reps-are-locally-regular-intro} provides a geometrization of Derksen-Weyman-Zelevinsky's mutation of decorated representations of
Jacobian algebras.
This yields the following result which 
essentially says that generic CC-functions are invariant under 
mutation.}

{
\begin{coro}\label{cor:1.8b} 
Assume that $(\widetilde{Q},\widetilde{S})$ is non-degenerate,
and let $\widetilde{\mathcal{P}} := \jacobalg{\widetilde{Q}}{\widetilde{S}}$.
Then for 
each generically $\tau^-$-reduced irreducible component 
$
Z\subseteq\decrep(\widetilde{\mathcal{P}},\bfd,\bfv)
$
there exists a dense open subset $U\subseteq Z$ such that for any $\mathcal{M} \in U$ and
any sequence $(k_1,\ldots,k_\ell)$ of elements of $[n]$, we have
$$
CC_{\widetilde{\mathcal{P}}'}(Z')
= 
CC_{\widetilde{\mathcal{P}}'}(\mathcal{M}')
$$
where
\begin{align*}
\widetilde{\mathcal{P}}' &:= \jacobalg{\widetilde{Q}'}{\widetilde{S}'}
\text{ with }
(\widetilde{Q}',\widetilde{S}') := \mu_{k_\ell} \cdots \mu_{k_1}(\widetilde{Q},\widetilde{S}),
\\
Z' &:= \mu_{k_\ell} \cdots \mu_{k_1}(Z)
\text{\quad and \quad}
\mathcal{M}' :=  \mu_{k_\ell} \cdots \mu_{k_1}(\mathcal{M}).
\end{align*}
\end{coro}
}

{
As a corollary we get the following mutation invariance of the 
set of generically $\tau^-$-reduced generic CC-functions.
}

\begin{coro}\label{cor:generic-stays-generic-after-DWZ-mutation-intro} 
If $(\widetilde{Q},\widetilde{S})$ is non-degenerate, then for any sequence $(k_1,\ldots,k_\ell)$ of elements of $[n]$, we have 
$$
\mathcal{B}_{(\widetilde{Q},\widetilde{S})}(Q,S)=\mathcal{B}_{\mu_{k_\ell}\cdots\mu_{k_1}(\widetilde{Q},\widetilde{S})}(\mu_{k_\ell}\cdots\mu_{k_1}(Q,S)).
$$
\end{coro}

\begin{remark}\label{rem:comparison-to-previous-approaches-intro} \begin{enumerate}

\item 
Corollary~\ref{cor:generic-stays-generic-after-DWZ-mutation-intro} states in particular that $\mathcal{B}_{(\widetilde{Q},\widetilde{S})}(Q,S)$ is contained in the upper cluster algebra of $\widetilde{B}$. Results of Dersken-Weyman-Zelevinsky \cite{DWZ2} imply that all cluster monomials belong to $\mathcal{B}_{(\widetilde{Q},\widetilde{S})}(Q,S)$.

\item\label{item:rem-comparison-to-previous-approaches-intro-GLS-and-Plamodon} 
Geiss-Leclerc-Schr\"{o}er proved Corollary~\ref{cor:generic-stays-generic-after-DWZ-mutation-intro} for certain Jacobi-finite non-degenerate ice QPs $(\widetilde{Q},\widetilde{S})$ arising in a specific Lie-theoretic framework \cite{GLS3}. 
Plamondon generalized their result to the situation where the ice quiver with potential $(\widetilde{Q},\widetilde{S})$ is assumed to be Jacobi-finite and non-degenerate \cite{Plamondon-generic}. 

\item\label{item:rem-comparison-to-previous-approaches-intro-Linear-Alg-vs-cluster-categories}
{
Our approach to model Derksen-Weyman-Zelevinsky's mutation
of decorated representations via regular functions 
between certain open subsets of  representation varieties
is rather elementary.
It uses only Linear Algebra and some basics on regular maps between
affine varieties.
In contrast to \cite{Plamondon-generic},
we do not need to assume the Jacobi-finiteness of the potential, and we also do not need generalized cluster categories or 
Ginzburg dg-algebras.
}

\end{enumerate}
\end{remark}

A remarkable theorem of Plamondon \cite{Plamondon-generic} states that if $\Lambda$ is a finite-dimensional algebra over an algebraically closed field, then the $\g$-vector function 
$$
\g^{\Lambda}_-:\decirrsr(\Lambda)\rightarrow\Z^{n}
$$ 
is bijective. 
A result of Cerulli-Labardini-Schr\"{o}er \cite{CLS} shows that this map is still injective for arbitrary quotients of complete path algebras of finite quivers. Combining this with Theorem~\ref{thm:criterion-for-lin-ind-of-CC-functions-intro} and Corollary~\ref{cor:generic-stays-generic-after-DWZ-mutation-intro}, we obtain the following:

\begin{coro}\label{coro:lin-ind-of-generic-basis-in-general-intro} If $(\widetilde{Q},\widetilde{S})$ is non-degenerate and at least one matrix $B'$ in the mutation-equivalence class of $B$ satisfies $\Ker(B')\cap\ZZ_{\geq 0}^n$, then $\mathcal{B}_{(\widetilde{Q},\widetilde{S})}(Q,S)$ is an $R$-linearly independent subset of 
the upper cluster algebra with geometric coefficients $\mathcal{A}(\widetilde{B})$. 
\end{coro}

Subsection \ref{subsec:cluster-variables} has two aims. One is to illustrate the necessity to assume that $(\widetilde{Q},\widetilde{S})$, and not only $(Q,S)$, be non-degenerate in order to be able to express cluster variables as Caldero-Chapoton functions. The second aim is to recall how, by the work of Derksen-Weyman-Zelevinsky \cite{DWZ2}, cluster monomials can be written as generic Caldero-Chapoton functions of generically $\tau^-$-reduced components of representation spaces of quivers with potential.

\subsection{Applications to surface cluster algebras}\label{subsec:applications-to-surfaces}

{The last 15 years have seen a lot of research aimed at finding and understanding
various types of natural bases for cluster algebras, ``natural'' meaning in particular that cluster monomials must be contained in the bases proposed (cluster monomials in skew-symmetric cluster algebras were shown to be linearly independent in \cite{CKLP}). 
This arguably culminated in the work \cite{GHKK} of Gross-Hacking-Keel-Kontsevich on theta-bases.
In the case of skew-symmetric cluster algebras, another remarkable candidate for basis is the set of generically $\tau^{-1}$-reduced Caldero-Chapoton-functions proposed by \cite{GLS3} and inspired in Lusztig's dual semicanonical basis, see 
\cite{Lu} and also \cite{GLS1,GLS2}.

Recently, Qin showed in \cite{Qin} that the generic CC-functions
form  indeed a basis for the upper cluster algebra if the mutable part of the corresponding
QP has the injective reachable property and the extended exchange matrix has
full rank. These hypothesis are for example fulfilled for surface cluster
algebras corresponding to surfaces with non empty boundary and principal
coefficients, since they admit maximal green sequences~\cite{ACCERV,Mills1}.  By a result
of Muller \cite{Muller1,Muller2}, for surface cluster algebras corresponding to surfaces with
at least two marked points in the boundary, coincide for any choice of
geomtric coefficients with their upper cluster algebra.  

Using these observations, we show here that the generic CC-functions are
actually a basis of the above mentioned cluster algebras for any choice of
geometric coefficients. Our strategy consists roughly speaking of two
steps:
\begin{enumerate}
\item\label{item:intro-linear-ind} We can show, with the help of an adequate triangulation,  that
the generic CC-functions are linearly independet for each choice of
coefficients.
\item\label{item:intro:spanning} The spanning
property of the generic CC-functions is preserved under coefficient
specialization because the CC-algebra sits between the cluster algebra
and its upper cluster algebra.
\end{enumerate}

The main technical difficulty in dealing with \eqref{item:intro:spanning} was that generic CC-functions
previously were only considered when the quiver for the extended exchange matrix
admits a Jacobi-finite non-degenerate potential. This Jacobi-finiteness can
not be expected for arbitrary choices of coefficients. For this reason we are not able to simply rely on the results of \cite{Plamondon-generic}, obtained via cluster categories of Jacobi-finite non-degenerate quivers with potential; in order to encompass all possible choices of geometric coefficients, we need the mutation invariance of the set of generic Caldero-Chapoton functions even in the infinite-dimensional case.}
}

\subsubsection{Bypasses of a gentle quiver with potential and matrix column identities}

In Subsection \ref{sec:formula-for-special-bands} we show that if $(Q,W)$ is a 2-acyclic Jacobi-finite QP with gentle Jacobian algebra, then the \emph{bypasses} and \emph{almost bypasses} of $Q$ give rise to certain linear identities satisfied by the columns of the signed adjacency matrix of $Q$. It turns out that in an unpunctured surface with at most two marked points on each boundary component, around every boundary component runs a bypass or almost bypass of the QP of any of its triangulations.

\subsubsection{Linear independence of the set of generically $\tau^-$-reduced CC-functions}

Using the matrix column identities proved in Subsection \ref{sec:formula-for-special-bands}, in Subsection \ref{sec:lin-ind-for-surfaces} we show the fifth main result of this paper, namely:

\begin{thm}\label{thm:existence-triang-with-nice-matrix-intro} Every surface $\surf$ such that $\partial\Sigma\neq\varnothing$ admits a triangulation $\sigma$ whose associated skew-symmetric matrix $B(\sigma)$ satisfies $\Ker(B(\sigma))\cap\ZZ^n_{\geq 0}$.
\end{thm}

From Theorem \ref{thm:existence-triang-with-nice-matrix-intro} and Corollary \ref{coro:lin-ind-of-generic-basis-in-general-intro} we immediately deduce:

\begin{coro}\label{coro:lin-ind-for-surfs-arbitrary-coeffs-intro} If $Q=Q(\tau)$ for some tagged triangulation $\tau$ of a surface $\surf$ with non-empty boundary, and $(\widetilde{Q},\widetilde{S})$ is non-degenerate, then $\calB_{(\widetilde{Q},\widetilde{S})}(Q,S)$ is linearly independent over $R=\Q[x_{n+1}^{\pm1},\ldots,x_{n+m}^{\pm1}]$.
\end{coro}

\subsubsection{Spanning the Caldero-Chapoton algebra with coefficients}

Recall that Qin \cite{Qin} has recently proved that if the rectangular matrix $\widetilde{B}\in\ZZ^{(n+m)\times n}$ has full rank, and the mutable part $Q$ of $\widetilde{Q}$ admits a reddening mutation sequence, then for any non-degenerate potential $\widetilde{S}\in\RA{\widetilde{Q}}$, the set $\calB_{(\widetilde{Q},\widetilde{S})}(Q,S)$ is a basis of the upper cluster algebra $\upper(\widetilde{B})$ over the ground ring $R$. Recall also that Mills \cite{Mills1} has shown that for any triangulation $\tau$ of a surface with non-empty boundary $\surf$, the quiver $Q(\tau)$ admits a reddening mutation sequence. Using these two important results, in Section \ref{sec:spanning-the-CC-algebra} we prove the following:

\begin{coro}\label{coro:gen-basis-is-basis-of-CC-alg-for-surfaces-intro}  If $Q=Q(\tau)$ for some tagged triangulation $\tau$ of a surface $\surf$ with non-empty boundary, then there exists a non-degenerate potential $\widetilde{S}\in\RA{\widetilde{Q}}$ such that $\calB_{(\widetilde{Q},\widetilde{S})}(Q,S)$ is a basis of the Caldero-Chapoton algebra $\calCC_{(\widetilde{Q},\widetilde{S})}(Q,S)$ over the ground ring $R$.
\end{coro}

\subsubsection{$\calB_{(\widetilde{Q},\widetilde{S})}(Q,S)$ is a basis of the upper cluster algebra with arbitrary geometric coefficients}

Muller has shown that if $B=B(\tau)$ for some tagged triangulation of a surface $\surf$ with at least two marked points on the boundary, then $\mathcal{A}(\widetilde{B})=\upper(\widetilde{B})$ (see Definition \ref{def:precise-def-of-cluster-alg-and-upper-cluster-alg}). Combining this with Corollary \ref{coro:gen-basis-is-basis-of-CC-alg-for-surfaces-intro}, we arrive at the sixth main result of the present article:

\begin{thm} If $B=B(\tau)$ for some tagged triangulation $\tau$ of a surface with non-empty boundary $\surf$ with at least two marked points on the boundary, then $\calB_{(\widetilde{Q},\widetilde{S})}(Q,S)$ is a basis of the (upper) cluster algebra with coefficients $\mathcal{A}(\widetilde{B})=\upper(\widetilde{B})$ over the ground ring $R=\Q[x_{n+1}^{\pm1},\ldots,x_{n+m}^{\pm1}]$.
\end{thm}

\section{$\g$-vectors and extended $\g$-vectors}\label{sec:g-vectors}

\subsection{Preliminaries involving admissible quotients of path algebras}\label{subsec:admissible-quotients}

Suppose that $\widetilde{J}$ is a 2-sided ideal of $\C\langle\widetilde{Q}\rangle$ for which there exists an integer $p\gg 2$ such that $\widetilde{\maxid}_0^p\subseteq\widetilde{J}\subseteq \widetilde{\maxid}_0^2$, and set $J:=\rho(\widetilde{J})$. Then we obviously have $\maxid_0^p\subseteq J\subseteq \maxid_0^2$, and both
\begin{equation}\label{eq:fin-dim-algebras-tildeA-and-A}
\widetilde{A}:=\C\langle\widetilde{Q}\rangle/\widetilde{J} \ \ \  \text{and} \ \ \ A:=\C\langle Q\rangle/J
\end{equation}
are finite-dimensional $\C$-algebras. Furthermore, $\rho$ induces a surjective $\C$-algebra homomorphism
$\overline{\rho}:\widetilde{A}\rightarrow A$ that makes the diagram
$$
\xymatrix{
\usualRA{\widetilde{Q}} \ar[d]_{\pi_{\widetilde{A}}} \ar[r]^{\rho} & \usualRA{Q} \ar[d]^{\pi_A} \\
\widetilde{A} \ar[r]_{\overline{\rho}} & A
}
$$
commute, where the vertical arrows are the canonical projections. We can thus see the category $A$-$\modcat$ of finite-dimensional left $A$-modules as a full-subcategory of the category $\widetilde{A}$-$\modcat$ of finite-dimensional left $\widetilde{A}$-modules. Note that although projective (resp. injective) $A$-modules may no longer be projective (resp. injective) in $\widetilde{A}$-$\modcat$, the simple $A$-modules remain simple in $\widetilde{A}$-$\modcat$, and for every $M\in A$-$\modcat$ we have $\soc_A(M)=\soc_{\widetilde{A}}(M)$. Furthermore, $\overline{\rho}$ is a surjective homomorphism of right $\widetilde{A}$-modules that for each $k\in\widetilde{Q}_0=Q_0$ restricts to a surjective homomorphism of right $\widetilde{A}$-modules $e_k\widetilde{A}\rightarrow e_kA$ and hence induces a monomorphism of left $\widetilde{A}$-modules
$$
I(k):=\Hom_{\mathbb{C}}(e_kA,\C)\rightarrow \Hom_{\mathbb{C}}(e_k\widetilde{A},\C)=:\widetilde{I}(k).
$$

\begin{lemma}\label{lemma:compatible-bases-for-restriction-of-quivers-and-algs} 
There exist sets of paths $\mathcal{B}\subseteq\usualRA{Q}$ and $\mathcal{C}\subseteq \usualRA{\widetilde{Q}}\cap\Ker(\rho)$ such that
$\pi_{\widetilde{A}}(\mathcal{B}\cup\mathcal{C})$
is a $\mathbb{C}$-basis of
$\widetilde{A}$ and
$\pi_{A}(\mathcal{B})=\overline{\rho}(\pi_{\widetilde{A}}(\mathcal{B}))$
is a $\mathbb{C}$-basis of $A$. Furthermore, for every tuple $(b_k)_{k\in\widetilde{Q}_0}$ of non-negative integers, there is a commutative diagram of left $\widetilde{A}$-module monomorphisms
$$
\xymatrix{
\bigoplus_{k\in\widetilde{Q}_0}\bigoplus_{\ell=1}^{b_k}I(k) \ar[r] \ar[d]_{\cong} &
\bigoplus_{k\in\widetilde{Q}_0}\bigoplus_{\ell=1}^{b_k}\widetilde{I}(k) \ar[d]^{\cong} \\
\Hom_{\C}(\bigoplus_{k\in\widetilde{Q}_0}\bigoplus_{\ell=1}^{b_k}e_kA,\C) \ar[r] &
\Hom_{\C}(\bigoplus_{k\in\widetilde{Q}_0}\bigoplus_{\ell=1}^{b_k}e_k\widetilde{A},\C)
}
$$
the image of whose bottom row consists precisely of the $\C$-linear maps $\varphi:\bigoplus_{k\in\widetilde{Q}_0}\bigoplus_{\ell=1}^{b_k}e_k\widetilde{A}\rightarrow\C$ such that $\varphi(\bigoplus_{k\in\widetilde{Q}_0}\bigoplus_{\ell=1}^{b_k}\vecspan_{\C}(e_k\pi_{\widetilde{A}}(\mathcal{C})))=0$.
\end{lemma}

\begin{proof} Let $\mathcal{C}$ be a maximal set with the following properties:
\begin{enumerate}
\item $\mathcal{C}$ consists entirely of paths that pass through at least one vertex outside $V$;
\item  $\pi_{\widetilde{A}}(\mathcal{C})$ is a $\C$-linearly independent subset of $\widetilde{A}$.
\end{enumerate}
Such a set $\mathcal{C}$ exists because $\dim_{\C}(\widetilde{A})<\infty$. Let $\mathcal{B}$ be any set of paths on $\widetilde{Q}$ such that $\pi_{\widetilde{A}}(\mathcal{C})\cup\pi_{\widetilde{A}}(\mathcal{B})$ is a $\C$-basis of $\widetilde{A}$. Because of the way $\mathcal{C}$ was chosen, $\mathcal{B}$ consists entirely of paths that never pass through any vertex outside $[n]$.

It is clear that $\pi_{A}(\mathcal{B})=\overline{\rho}(\pi_{\widetilde{A}}(\mathcal{B}))$ spans $A$ as a $\C$-vector space. Suppose that $p_1,\ldots,p_s\in\mathcal{B}$ and $\alpha_1,\ldots,\alpha_s\in\C$ satisfy $\alpha_1\pi_A(p_1)+\ldots\alpha_s\pi_A(p_s)=0$ in $A$.
Then $\rho(\alpha_1p_1+\ldots+\alpha_sp_s)=\alpha_1p_1+\ldots+\alpha_sp_s\in J=\rho(\widetilde{J})$, so there exist $x\in \widetilde{J}$ and $y\in\Ker(\rho)$ such that $\alpha_1p_1+\ldots+\alpha_sp_s=x+y$. Because of the way $\mathcal{C}$ was chosen, for every path $q$ appearing in the expression of $y$ as $\C$-linear combination of paths on $\widetilde{Q}$, the element $\pi_{\widetilde{A}}(q)$ belongs to the $\C$-vector space spanned by $\pi_{\widetilde{A}}(\mathcal{C})$. Hence, $\alpha_1\pi_{\widetilde{A}}(p_1)+\ldots+\alpha_s\pi_{\widetilde{A}}(p_s)=\pi_{\widetilde{A}}(y)$ is simultaneously a $\C$-linear combination of elements of $\pi_{\widetilde{A}}(\mathcal{B})$ and a $\C$-linear combination of elements of $\pi_{\widetilde{A}}(\mathcal{C})$. Therefore, $\alpha_1\pi_{\widetilde{A}}(p_1)+\ldots\alpha_s\pi_{\widetilde{A}}(p_s)=0$ in $\widetilde{A}$. The way $\mathcal{B}$ was chosen implies that $\alpha_1=\ldots=\alpha_s=0$.

We see that, as $\C$-vector spaces,
$$
e_k\widetilde{A}=\operatorname{span}_{\C}(e_k\pi_{\widetilde{A}}(\mathcal{C}))\oplus\operatorname{span}_{\C}(e_k\pi_{\widetilde{A}}(\mathcal{B})) \ \text{and} \ e_kA=\operatorname{span}_{\C}(e_k\pi_{A}(\mathcal{B})),
$$
and that $\overline{\rho}$ maps $\operatorname{span}_{\C}(e_k\pi_{\widetilde{A}}(\mathcal{C}))$ to zero and restricts to a $\C$-vector space isomorphism
$$
\operatorname{span}_{\C}(e_k\pi_{\widetilde{A}}(\mathcal{B}))\rightarrow\operatorname{span}_{\C}(e_k\pi_{A}(\mathcal{B})).
$$
Since the monomorphism of left $\widetilde{A}$-modules
$$
\Hom_{\C}(\bigoplus_{k\in\widetilde{Q}_0}\bigoplus_{\ell=1}^{b_k}e_kA,\C) \hookrightarrow
\Hom_{\C}(\bigoplus_{k\in\widetilde{Q}_0}\bigoplus_{\ell=1}^{b_k}e_k\widetilde{A},\C)
$$
is induced by the right $\widetilde{A}$-module epimorphism
$$
\bigoplus_{k\in\widetilde{Q}_0}\bigoplus_{\ell=1}^{b_k}e_k\widetilde{A}\rightarrow \bigoplus_{k\in\widetilde{Q}_0}\bigoplus_{\ell=1}^{b_k}e_kA
$$
that is defined diagonally in terms of $\overline{\rho}$, the last assertion of the lemma follows.
\end{proof}

\begin{lemma}\label{lemma:restricted-vs-extended-injective-presentations} 
For any $M\in A$-$\modcat$, if $M$ is supported in $[n]$ and
$$
0\rightarrow M\rightarrow\bigoplus_{i=1}^n I(i)^{a_i}\rightarrow \bigoplus_{i=1}^n I(i)^{b_i}
$$
is a minimal injective presentation of $M$ in $A$-$\modcat$, then there exist non-negative integers $c_j$, for $j\in\widetilde{Q}_0\setminus [n]=[n+1,\ldots,n+m]$, such that
$$
0\rightarrow M\rightarrow\bigoplus_{i=1}^n \widetilde{I}(i)^{a_i}\rightarrow \left(\bigoplus_{i=1}^n \widetilde{I}(i)^{b_i}\right)\oplus\left(\bigoplus_{j=n+1}^{n+m}\widetilde{I}(j)^{c_j}\right)
$$
is a minimal injective presentation of $M$ in $\widetilde{A}$-$\modcat$. The corresponding statement for minimal projective presentations is true as well.
\end{lemma}

\begin{proof}
Let $f:M\hookrightarrow I$ be an injective envelope of $M$ in $A$-$\modcat$, and let $g:I\hookrightarrow\widetilde{I}$ be an injective envelope of $I$ in $\widetilde{A}$-$\modcat$. Then $\soc_{\widetilde{A}}(\widetilde{I})=\soc_{\widetilde{A}}(I)=\soc_{A}(I)=\soc_A(M)=\soc_{\widetilde{A}}(M)$, and hence, $h:=g\circ f:M\hookrightarrow\widetilde{I}$ is an injective envelope of $M$ in $\widetilde{A}$-$\modcat$. Moreover, if we write $I=\Hom_{\C}(\bigoplus_{k\in\widetilde{Q}_0}\bigoplus_{\ell=1}^{a_k}e_kA,\C)$, then $\widetilde{I}=\Hom_{\C}(\bigoplus_{k\in\widetilde{Q}_0}\bigoplus_{\ell=1}^{a_k}e_k\widetilde{A},\C)$ and the inclusion $g:I\hookrightarrow\widetilde{I}$ is the one described in the last assertion of Lemma \ref{lemma:compatible-bases-for-restriction-of-quivers-and-algs}.

Since $I/M$ is an $\widetilde{A}$-submodule of $\widetilde{I}/M$, $\soc_{\widetilde{A}}(I/M)$ is an $\widetilde{A}$-submodule of $\soc_{\widetilde{A}}(\widetilde{I}/M)$. Hence, for all $k\in [n]$ the multiplicity of $S_k$ as a direct summand of $\soc_A(I/M)$ is less than or equal to the multiplicity of $S_k$ as a direct summand of $\soc_{\widetilde{A}}(\widetilde{I}/M)$. The lemma will follow once we establish the equality of these multiplicities.

Fix $j\in [n]$ and take any element $w\in\soc_{\widetilde{A}}(\widetilde{I}/M)$ such that $w=e_j\cdot w$. Then $w\in e_j\cdot(\widetilde{I}/M)=(e_j\cdot\widetilde{I})/(e_j\cdot M)$, so $w$ can be written as $w=u+e_j\cdot M$ for some $u\in e_j\cdot \widetilde{I}$. 

Let $\mathcal{C}$ be as in the conclusion of Lemma \ref{lemma:compatible-bases-for-restriction-of-quivers-and-algs}, and
take an element
$$
v=\sum_{k\in\widetilde{Q}_0}\sum_{\ell=1}^{a_k}\sum_{\alpha_1\cdots\alpha_s}\lambda_{k,\ell,\alpha_1\cdots\alpha_{s}} \alpha_1\cdots\alpha_{s}\in
\bigoplus_{k\in\widetilde{Q}_0}\bigoplus_{\ell=1}^{a_k}\vecspan_{\C}(e_k\pi_{\widetilde{A}}(\mathcal{C}))
$$
with each $\alpha_1\cdots\alpha_s\in e_k\pi_{\widetilde{A}}(\mathcal{C})$. Then
\begin{eqnarray*}
u(v) &=&
\sum_{k\in\widetilde{Q}_0}\sum_{\ell=1}^{a_k}\sum_{\alpha_1\cdots\alpha_s}\lambda_{k,\ell,\alpha_1\cdots\alpha_{s}} u(\alpha_1\cdots\alpha_{s})\\
&=&
\sum_{k\in\widetilde{Q}_0}\sum_{\ell=1}^{a_k}\sum_{\alpha_1\cdots\alpha_s}\lambda_{k,\ell,\alpha_1\cdots\alpha_{s}} \widetilde{I}_{a_s}(u)(\alpha_1\cdots\alpha_{s-1}).
\end{eqnarray*}
Since $w\in\soc_{\widetilde{A}}(\widetilde{I}/M)$,  we have $\widetilde{I}_{\alpha_s}(u)\in M\subseteq I$, and hence, $(\widetilde{I}_{\alpha_s}(u))(\alpha_1\cdots \alpha_{s-1})=0$. Hence, $u(v)=0$. Therefore, $u\in I$, which means that $w\in\soc_{\widetilde{A}}(I/M)$.

We conclude that the part of $\soc_{\widetilde{A}}(\widetilde{I}/M)$ which is supported in $[n]$ coincides with $\soc_{\widetilde{A}}(I/M)$, and this finishes the proof of Lemma \ref{lemma:restricted-vs-extended-injective-presentations} for minimal injective presentations.

The proof of the corresponding statement for minimal projective presentations is similar.
\end{proof}

\begin{remark}\label{rem:APR-idempotent-ideals} With some effort, it is possible to extract Lemma \ref{lemma:restricted-vs-extended-injective-presentations} from \cite{APT}, which works in much more abstract~terms.
\end{remark}

\subsection{The case of arbitrary quotients of complete path algebras}\label{subsec:arbitrary-quotients}

\begin{defi}\label{def:decorated-representation} \begin{enumerate} \item A \emph{decorated representation of $\Lambda$} is a pair $\mathcal{M}=(M,\bfv)$ consisting of a finite-dimensional left $\Lambda$-module $M$ and a non-negative integer vector $\bfv=(v_j)_{j\in\widetilde{Q}_0}\in\ZZ_{\geq 0}^{\widetilde{Q}_0}$;
\item a decorated representation $\mathcal{M}=(M,\bfv)$ of $\Lambda$ is \emph{supported in $[n]$} if $M_{j}=0$ and $v_j=0$ for all $j\in\widetilde{Q}_0\setminus[n]=\{n+1,\ldots,n+m\}$.
\end{enumerate}
\end{defi}

\begin{defi}\label{def:restricted-and-extended-g-vectors} For a decorated representation $\mathcal{M}=(M,\bfv)$ of $\Lambda$ which is supported in $[n]$ we define the \emph{$\g$-vector} $\g_{\mathcal{M}}^\Lambda=(g_{i})_{i=1}^n\in\ZZ^{n}$  and the  \emph{extended $\g$-vector} $\widetilde{\g}_{\mathcal{M}}^{\widetilde{\Lambda}}=(\widetilde{g}_{i})_{i=1}^{n+m}\in\ZZ^{n+m}$ according to the following formulas:
\begin{eqnarray*}
g_i &:=& -\dim_{\C}(\Hom_{\Lambda}(S(i),M))+\dim_{\C}(\Ext^1_{\Lambda}(S(i),M))+v_i \ \ \ \ \   \text{for $i\in[n]$}\\
\widetilde{g}_i &:=& -\dim_{\C}(\Hom_{\widetilde{\Lambda}}(S(i),M))+\dim_{\C}(\Ext^1_{\widetilde{\Lambda}}(S(i),M))+v_i \ \ \ \ \   \text{for $i\in[n+m]$.}
\end{eqnarray*}
\end{defi}

The following result will be essential in the proof of the linear independence of the generic basis for arbitrary geometric coefficient systems.

\begin{prop}\label{prop:restricted-g-vector-vs-extended-g-vector} 
For every decorated representation $\mathcal{M}=(M,\bfv)$ of $\Lambda$ which is supported in $[n]$, the vector obtained from $\widetilde{\g}_{\mathcal{M}}^{\widetilde{\Lambda}}$ by deleting its last $m$ entries is precisely~$\g_{\mathcal{M}}^{\Lambda}$.
\end{prop}

\begin{proof} Take $p\in\mathbb{Z}$ greater than $\dim_{\C}(M)$. Let $\widetilde{\Lambda}_p:=\RA{\widetilde{Q}}/(\widetilde{\mathcal{J}}+\widetilde{\mathfrak{m}}^p)$ and $\Lambda_p:=\RA{Q}/(\mathcal{J}+\mathfrak{m}^p)$ be the $p^{\operatorname{th}}$ \emph{truncations} of $\widetilde{\Lambda}$ and $\Lambda$. Let $\widetilde{J}^{(< p)}\subseteq\C\langle \widetilde{Q}\rangle$ (resp. $J^{(<p)}\subseteq\C\langle Q\rangle$) be the 2-sided ideal whose elements are precisely those obtained from the elements of $\widetilde{\mathcal{J}}$ (resp. $\mathcal{J}$) by deleting all the paths of length greater than $p-1$ from their expressions as possibly infinite $\C$-linear combinations of paths on $\widetilde{Q}$ (resp. $Q$), and, as before, let $\widetilde{\mathfrak{m}}_0\subseteq\C\langle\widetilde{Q}\rangle$ (resp. $\mathfrak{m}_0\subseteq\C\langle Q\rangle$) be the 2-sided ideal generated in $\C\langle\widetilde{Q}\rangle$ (resp. $\C\langle Q\rangle$) by the arrow set $\widetilde{Q}_1$ (resp. $Q_1$). Then the inclusions $\C\langle\widetilde{Q}\rangle\hookrightarrow\RA{\widetilde{Q}}$ and $\C\langle Q\rangle\hookrightarrow\RA{Q}$ induce algebra isomorphisms
\begin{eqnarray*}
&&
\widetilde{A}:=\C\langle\widetilde{Q}\rangle/(\widetilde{J}^{(< p)}+\widetilde{\mathfrak{m}}_0^p)\overset{\cong}{\longrightarrow}\RA{\widetilde{Q}}/(\widetilde{\mathcal{J}}+\widetilde{\mathfrak{m}}^p)=\widetilde{\Lambda}_p\\
&\text{and}&
A:=\C\langle Q\rangle/(J^{(< p)}+\mathfrak{m}_0^p)\overset{\cong}{\longrightarrow}\RA{Q}/(\mathcal{J}+\mathfrak{m}^p)=\Lambda_p.
\end{eqnarray*}
Since Lemma \ref{lemma:restricted-vs-extended-injective-presentations} can be applied to $\widetilde{A}$ and $A$, this implies that it can also be applied to $\widetilde{\Lambda}_p$ and~$\Lambda_p$.

Let
$$
0\rightarrow M\rightarrow\bigoplus_{i\in [n]}I(i)^{a_i}\rightarrow \bigoplus_{i\in [n]}I(i)^{b_i}
$$
be a minimal injective presentation of $M$ in $\Lambda_p$-$\modcat$. By Lemma \ref{lemma:restricted-vs-extended-injective-presentations}, there exist positive integers $c_i$, for $i\in[n+m]\setminus [n]$, such that
$$
0\rightarrow M\rightarrow\bigoplus_{i\in [n]}\widetilde{I}(i)^{a_i}\rightarrow \left(\bigoplus_{i\in [n]}\widetilde{I}(i)^{b_i}\right)\oplus\left(\bigoplus_{j\in[n+m]\setminus [n]}\widetilde{I}(j)^{c_j}\right)
$$
is a minimal injective presentation of $M$ in $\widetilde{\Lambda}_p$-$\modcat$. The theorem then follows by applying \cite[Lemma~3.4]{CLS}.
\end{proof}

From the proofs of Lemma \ref{lemma:restricted-vs-extended-injective-presentations} and Proposition \ref{prop:restricted-g-vector-vs-extended-g-vector} we can deduce the following result of independent interest.

\begin{coro} Let $\widetilde{Q}'$ be the subquiver of $\widetilde{Q}$ obtained by deleting all the arrows that connect frozen vertices to frozen vertices, and let $\widetilde{\Lambda}'$ be the $\C$-subalgebra of $\widetilde{\Lambda}$ defined as the image of $\RA{\widetilde{Q}'}$ in $\widetilde{\Lambda}$ under the projection $\RA{\widetilde{Q}}\rightarrow\widetilde{\Lambda}$. Then, for every decorated representation $\mathcal{M}$ of $\Lambda$ supported in $[n]$ we have $\widetilde{\g}_{\mathcal{M}}^{\widetilde{\Lambda}'}=\widetilde{\g}_{\mathcal{M}}^{\widetilde{\Lambda}}$. In other words, the removal of arrows not incident to mutable vertices does not affect the extended $\g$-vector.
\end{coro}

\subsection{The case of arbitrary Jacobian algebras}\label{subsec:Jacobian-algebras}

Now, we verify that the results from the previous subsection can be applied to quivers with potential. 
The following lemma, whose obvious proof we leave in the reader's hands, states in particular that $J(S)$ coincides with $\widehat{\rho}(J(\widetilde{S}))$.

\begin{lemma}\label{lemma:extend-reps-putting-zeros} The continuous $\C$-algebra homomorphism $\widehat{\rho}:\RA{\widetilde{Q}}\rightarrow\RA{Q}$ satisfies $\widehat{\rho}(\widetilde{\mathfrak{m}})=\mathfrak{m}$. Furthermore, for $a\in\widetilde{Q}_1$ we have
$$
\widehat{\rho}(\partial_a(\widetilde{S}))=\begin{cases}
\partial_a(S) & \text{if $a\in Q_1$;}\\
0 & \text{otherwise.}
\end{cases}
$$
Hence, $\widehat{\rho}(J(\widetilde{S}))=J(S)$ and $\widehat{\rho}$ induces a surjective $\C$-algebra homomorphism $\jacobalg{\widetilde{Q}}{\widetilde{S}}\rightarrow\jacobalg{Q}{S}$ that makes the diagram
$$
\xymatrix{
\RA{\widetilde{Q}} \ar[r]^{\widehat{\rho}} \ar[d] & \RA{Q} \ar[d] \\
\jacobalg{\widetilde{Q}}{\widetilde{S}} \ar[r] & \jacobalg{Q}{S}
}
$$
commute. Therefore, the category of finite-dimensional left $\jacobalg{Q}{S}$-modules is a full subcategory of the category of finite-dimensional left $\jacobalg{\widetilde{Q}}{\widetilde{S}}$-modules.
Consequently, if $M$ is a nilpotent representation of $Q$ which is annihilated by all the elements of $\{\partial_a(S)\suchthat a\in Q_1\}$, then $M$ is nilpotent as a representation of $\widetilde{Q}$ and all the elements of $\{\partial_a(\widetilde{S})\suchthat a\in \widetilde{Q}_1\}$ annihilate $M$.
\end{lemma}

Lemma \ref{lemma:extend-reps-putting-zeros} implies that Definitions \ref{def:decorated-representation} and \ref{def:restricted-and-extended-g-vectors}, as well as Proposition \ref{prop:restricted-g-vector-vs-extended-g-vector} can be fully applied to the Jacobian algebras $\widetilde{\Lambda}:=\jacobalg{\widetilde{Q}}{\widetilde{S}}$ and $\Lambda:=\jacobalg{Q}{S}$, for which we will write $\widetilde{\mathbf{g}}_{\mathcal{M}}^{(\widetilde{Q},\widetilde{S})}:=\widetilde{\mathbf{g}}_{\mathcal{M}}^{\jacobalg{\widetilde{Q}}{\widetilde{S}}}$ and $\mathbf{g}_{\mathcal{M}}^{(Q,S)}:=\mathbf{g}_{\mathcal{M}}^{\jacobalg{Q}{S}}$.

\begin{ex}\label{ex:two-triangles} Let $(\widetilde{Q},\widetilde{S})$ be the following (non-degenerate) quiver with potential
$$
\xymatrix{
& & 3 \ar[dr]^{\beta} & &\\
\widetilde{Q}=&1 \ar[ur]^{\gamma} \ar[dr]_{\varepsilon} & & 2 \ar[ll]^{\alpha} & \widetilde{S}=\alpha\beta\gamma-\alpha\delta\varepsilon. \\
 & & 4 \ar[ur]_{\delta} & &}
$$
Taking $n=3$, $m=1$, $n+m=4$, we have
$$
\xymatrix{
& & 3 \ar[dr]^{\beta} & &\\
Q=&1 \ar[ur]^{\gamma}  & & 2 \ar[ll]^{\alpha} & S=\alpha\beta\gamma. \\
 & & 4  & &}
$$
The minimal injective presentation of $S(2)$ in $\jacobalg{Q}{S}$-$\modcat$ is $0\rightarrow S(2)\rightarrow I(2)\rightarrow I(3)$, whereas the minimal injective presentation of $S(2)$ in $\jacobalg{\widetilde{Q}}{\widetilde{S}}$-$\modcat$ is $0\rightarrow S(2)\rightarrow \widetilde{I}(2)\rightarrow \widetilde{I}(3)\oplus\widetilde{I}(4)$. Hence, letting $\mathcal{M}=(S(2),0)$, we have
$\g_{\mathcal{M}}=(0,-1,1)$ and $\widetilde{\g}_{\mathcal{M}}=(0,-1,1,1)$.
\end{ex}

\section{A partial order on $\ZZ^n$}\label{sec:partial-order}

\begin{lemma}\label{lemma:partial-order-on-Z^n} 
If $\Ker(B)\cap\ZZ^{n}_{\geq 0}=0$, then the rule
$$
\mathbf{a}\preceq_B\mathbf{b} \overset{\operatorname{def}}{\Longleftrightarrow} \mathbf{a}-\mathbf{b}\in B(\ZZ^{n}_{\geq 0})
$$
defines  a partial order on $\ZZ^{n}$.
\end{lemma}

\begin{proof} It is clear that $\preceq_B$ is reflexive and transitive. For the anti-symmetry, suppose that $\mathbf{a},\mathbf{b}\in\ZZ^n$ satisfy $\mathbf{a}\preceq_B\mathbf{b}$ and $\mathbf{b}\preceq_B\mathbf{a}$, so that there exist $\mathbf{u},\mathbf{v}\in\ZZ^{n}_{\geq 0}$ such that $\mathbf{a}-\mathbf{b}=B\mathbf{u}$ and $\mathbf{b}-\mathbf{a}=B\mathbf{v}$, which implies that $\mathbf{u}+\mathbf{v}\in \Ker(B)\cap \ZZ^{n}_{\geq 0}$. The equality $\Ker(B)\cap\ZZ^{n}_{\geq 0}=0$ allows us to deduce that $\mathbf{u}+\mathbf{v}=0$, and since $\mathbf{u},\mathbf{v}\in\ZZ^{n}_{\geq 0}$, this implies $\mathbf{u}=0=\mathbf{v}$ and hence, $\mathbf{a}=\mathbf{b}$.
\end{proof}

\begin{remark}
\label{rem:the-5-conditions-on-B}\begin{itemize}\item 
Consider the following five conditions:
\begin{enumerate}
\item\label{item:rank(B)=n} $\rank(B)=n$;
\item\label{item:B-satisfies-columns-condition} no column of $B$ is identically zero, and $B$ has $\operatorname{corank}(B)$ columns that can be written as non-negative linear combinations of the remaining $\operatorname{rank}(B)$ columns of $B$ with non-negative rational coefficients;
\item\label{item:Im(B)capQ>0n-is-not-empty} $\operatorname{Im}(B)\cap\Q_{>0}^n\neq \varnothing$;
\item\label{item:ker(B)capQ>=0n-is-trivial} $\Ker(B)\cap\Q_{\geq 0}^n=0$;
\item\label{item:ker(B)capZ>=0n-is-trivial} $\Ker(B)\cap\ZZ_{\geq 0}^n=0$.
\end{enumerate}
In \cite[Lemma 4.4]{CLS} it was proved that the first four conditions are related by the implications \eqref{item:rank(B)=n}~$\Longrightarrow$~\eqref{item:B-satisfies-columns-condition}~$\Longrightarrow$~\eqref{item:Im(B)capQ>0n-is-not-empty}~$\Longrightarrow$~\eqref{item:ker(B)capQ>=0n-is-trivial}, and it is obvious that \eqref{item:ker(B)capQ>=0n-is-trivial} and \eqref{item:ker(B)capZ>=0n-is-trivial} are equivalent.
\item In the case of skew-symmetric matrices with $\rank(B)=n$, the corresponding partial order $\preceq_B$ has been considered by Fu-Keller \cite{FK}, Plamondon \cite{Plamondon-generic} and Qin \cite{Qin}. 
In the less restrictive situation $\Ker(B)\cap\Q_{\geq 0}^n=0$, it has been considered in \cite{CLS}.
\item It is straightforward to check that no triangulation $\tau$ of the once-punctured torus with empty boundary satisfies $\Ker(B(\tau))\cap\ZZ_{\geq 0}^3=0$. So, it is not true that every mutation class of skew-symmetric matrices has a representative $B$ satisfying $\Ker(B)\cap\ZZ_{\geq 0}^n=0$.
\item Denote by $\preceq_{B}^{\ZZ}$ the binary relation defined in Lemma \ref{lemma:partial-order-on-Z^n} and by $\preceq_B^{\Q}$ the binary relation defined by the rule $[\mathbf{a}\preceq_B^{\Q}\mathbf{b} \Longleftrightarrow \mathbf{a}-\mathbf{b}\in B(\Q^{n}_{\geq 0})]$. Then for all $\mathbf{a},\mathbf{b}\in\ZZ^n$ we obviously have $[\mathbf{a}\preceq_B^{\ZZ}\mathbf{b}\Longrightarrow \mathbf{a}\preceq_B^{\Q}\mathbf{b}]$. However, it is very easy to give examples of skew-symmetric matrices $B$ satisfying the equivalent conditions \eqref{item:ker(B)capQ>=0n-is-trivial} and \eqref{item:ker(B)capZ>=0n-is-trivial} are satisfied, for which there exist $\mathbf{a},\mathbf{b}\in\ZZ^n$ such that $\mathbf{a}\not\preceq_B^{\ZZ}\mathbf{b}$ and $\mathbf{a}\preceq_B^{\Q}\mathbf{b}]$. Informally speaking, this means that it is harder for two vectors in $\ZZ^n$ to be $\preceq_B^{\ZZ}$-comparable than to be $\preceq_B^{\Q}$-comparable.
\end{itemize}
\end{remark}

\section{The Caldero-Chapoton algebra with coefficients}\label{sec:CC-map-and-algebra}

Recall that the \emph{$F$-polynomial} of a representation $M$ of $Q$ supported in $[n]$ is
$$
F_M:=\sum_{\e\in\ZZ_{\geq 0}^n}\chi(\Gr_{\e}(M))\mathbf{X}^{\e}\in\ZZ[X_1,\ldots,X_n].
$$

\begin{defi}\label{def:CC-function} 
\begin{enumerate}\item For any decorated representation $\mathcal{M}=(M,\bfv)$ of $\Lambda$ which is supported in $[n]$, the \emph{Caldero-Chapoton function of $\mathcal{M}$ with coefficients extracted from $\widetilde{\Lambda}$}  is defined~to~be
\begin{eqnarray*}
CC_{\widetilde{\Lambda}}(\mathcal{M})&:=&\mathbf{x}^{\widetilde{\g}_{\mathcal{M}}^{\widetilde{\Lambda}}}F_M(\widehat{y}_1,\ldots,\widehat{y}_n)
\ = \ \sum_{\e\in\ZZ_{\geq 0}^n}\chi(\Gr_{\e}(M))\mathbf{x}^{\widetilde{B}\e+\widetilde{\g}_{\mathcal{M}}^{\widetilde{\Lambda}}},
\ \ \ \ \ \text{where}\\
\widehat{y}_j&:=&\prod_{i=1}^{n+m}x_i^{b_{ij}} \ \ \ \text{for each $j\in\{1,\ldots,n\}$}, \ \ \ \ \ \text{and}\\
\mathbf{x}^{\mathbf{f}}&:=&x_1^{f_1}\cdots x_{n}^{f_n}x_{n+1}^{f_{n+1}}\cdots x_{n+m}^{f_{n+m}}\ \ \ \ \  \text{for $\mathbf{f}=(f_1,\ldots,f_n,f_{n+1},\ldots,f_{n+m})\in\ZZ^{n+m}$}.
\end{eqnarray*}
\item For a set $\mathscr{M}$ of decorated representations of $\Lambda$ supported in $[n]$, we write
\begin{eqnarray*}CC_{\widetilde{\Lambda}}(\mathscr{M})&:=&\{CC_{\widetilde{\Lambda}}(\mathcal{M})\suchthat \mathcal{M}\in \mathscr{M}\}.
\end{eqnarray*}
\end{enumerate}
\end{defi}

In the case of Jacobian algebras, we shall write $CC_{(\widetilde{Q},\widetilde{S})}(\mathcal{M}):=CC_{\jacobalg{\widetilde{Q}}{\widetilde{S}}}(\mathcal{M})$ and $CC_{(\widetilde{Q},\widetilde{S})}(\mathscr{M}):=CC_{\jacobalg{\widetilde{Q}}{\widetilde{S}}}(\mathscr{M})$.

\begin{ex}
Let $(\widetilde{Q},\widetilde{S})$, $(Q,S)$ and $\mathcal{M}=(S(2),0)$ be as in Example \ref{ex:two-triangles}. Then
\begin{eqnarray*}
CC_{(\widetilde{Q},\widetilde{S})}(\mathcal{M}) &=& x_2^{-1}x_3x_4(1+x_1x_3^{-1}x_4^{-1}) \ \ \ \text{and}\\
CC_{(Q,S)}(\mathcal{M}) &=& x_2^{-1}x_3(1+x_1x_3^{-1}).
\end{eqnarray*}
\end{ex}

The following definition is inspired by \cite{CLS}, in turn inspired by the works of Caldero-Chapoton \cite{Caldero-Chapoton} and Derksen-Weyman-Zelevinsky \cite{DWZ1,DWZ2}. Recall from Definition \ref{def:precise-def-of-cluster-alg-and-upper-cluster-alg} that we have set $R:=\Q[x_{n+1}^{\pm1},\ldots,x_{n+m}^{\pm1}]$.

\begin{defi} 
The \emph{Caldero-Chapoton algebra of $\Lambda$ with coefficients extracted from $\widetilde{\Lambda}$} is the $R$-submodule $\calCC_{\widetilde{\Lambda}}(\Lambda)$ of $R[x_{1}^{\pm1},\ldots,x_{n}^{\pm1}]$ spanned by
$$
CC_{\widetilde{\Lambda}}(\decrep(\Lambda))=\{CC_{\widetilde{\Lambda}}(\mathcal{M})\suchthat\mathcal{M} \ \text{is a decorated representation of $\Lambda$ supported in} \ [n]\}
$$.
\end{defi}

As before, in the case of Jacobian algebras we will write $\calCC_{(\widetilde{Q},\widetilde{S})}(Q,S):=\calCC_{\jacobalg{\widetilde{Q}}{\widetilde{S})}}(\jacobalg{Q}{S})$. 

\begin{remark} Despite being defined as a submodule, the Caldero-Chapoton algebra $\calCC_{\widetilde{\Lambda}}(\Lambda)$ is actually an $R$-subalgebra of  $R[x_{1}^{\pm1},\ldots,x_{n}^{\pm1}]$. This is because $F_{\mathcal{M}}F_{\mathcal{N}}=F_{\mathcal{M}\oplus\mathcal{N}}$ (cf. \cite[Proposition~3.2]{DWZ2}) and $\widetilde{\g}_{\mathcal{M}}^{\widetilde{\Lambda}}+\widetilde{\g}_{\mathcal{N}}^{\widetilde{\Lambda}}=\widetilde{\g}_{\mathcal{M}\oplus\mathcal{N}}^{\widetilde{\Lambda}}$.
\end{remark}

The following result and its proof are a refinement of \cite[Proposition 4.3]{CLS} and its proof. 

\begin{thm}\label{thm:criterion-for-lin-ind-of-CC-functions} 
Take a set $\mathscr{M}$ of decorated representations of $\Lambda$ supported in $[n]$. Suppose that the following two conditions are satisfied:
\begin{enumerate}
\item $\Ker(B)\cap\mathbb{Z}^{n}_{\geq 0}=\{0\}$;
\item for any two different elements $\mathcal{M},\mathcal{M}'\in\mathscr{M}$, the $\g$-vectors $\g_{\mathcal{M}}^{\Lambda}$ and $\g_{\mathcal{M}'}^{\Lambda}$ are different.
\end{enumerate}
Then $CC_{\widetilde{\Lambda}}(\mathscr{M})$ is an $R$-linearly independent subset of the Laurent polynomial ring $R[x_1^{\pm1},\ldots,x_n^{\pm1}]$. 
\end{thm}

\begin{proof} Take finitely many elements of $\mathscr{M}$, say $\mathcal{M}_{1},\ldots,\mathcal{M}_{t}$. View each $CC_{\widetilde{\Lambda}}(\mathcal{M}_{k})$ as a Laurent polynomial in $x_1,\ldots,x_n$ with coefficients in $\Q[x_{n+1},\ldots,x_{n+m}]\subseteq R$. By
Proposition~\ref{prop:restricted-g-vector-vs-extended-g-vector} above, the $\g$-vector $\g_{\mathcal{M}_{k}}^{\Lambda}$ gives precisely the first $n$ entries of the extended $\g$-vector $\widetilde{\g}_{\mathcal{M}_{k}}^{\widetilde{\Lambda}}$. Hence, the Laurent monomial $\mathbf{x}^{\g_{\mathcal{M}_{k}}^{\Lambda}}$ appears in $CC_{\widetilde{\Lambda}}(\mathcal{M}_{k})$ accompanied with the non-zero coefficient $x_{n+1}^{\widetilde{g}_{\mathcal{M}_{k},n+1}}\cdots x_{n+m}^{\widetilde{g}_{\mathcal{M}_{k},n+m}}$. This is because for $\mathbf{e}=0$ we have $\chi(\Gr_{\mathbf{e}}(M_{k}))\mathbf{x}^{\widetilde{B}\mathbf{e}+\widetilde{g}_{\mathcal{M}_{k}}^{\Lambda}}=\mathbf{x}^{\widetilde{g}_{\mathcal{M}_{k}}^{\widetilde{\Lambda}}}$, and because, due to the equality $\Ker(B)\cap\ZZ^n_{\geq 0}=\{0\}$, for $\mathbf{e}\neq 0$ we have $\mathbf{x}^{B\mathbf{e}+\g_{\mathcal{M}_{k}}^{\Lambda}}\neq \mathbf{x}^{\g_{\mathcal{M}_{k}}^{\Lambda}}$. We deduce, in particular, that the Caldero-Chapoton functions $CC_{\widetilde{\Lambda}}(\mathcal{M}_{1}), \ldots, CC_{\widetilde{\Lambda}}(\mathcal{M}_{t})$ are all pairwise distinct.

Suppose that $\lambda_1\ldots,\lambda_t\in R=\Q[x_{n+1}^{\pm1},\ldots,x_{n+m}^{\pm1}]$ make the equality
$$
\lambda_1CC_{\widetilde{\Lambda}}(\mathcal{M}_1)+\ldots+\lambda_tCC_{\widetilde{\Lambda}}(\mathcal{M}_t)=0
$$
hold true in $R[x_1^{\pm1},\ldots,x_n^{\pm1}]$.
Let $s\in\{1,\ldots,t\}$ be an index such that ${\g_{\mathcal{M}_s}^{\Lambda}}$ is a $\preceq$-maximal element of the set
$\{{\g_{\mathcal{M}_k}^{\Lambda}}\suchthat 1\leq k\leq t\}$. Then $\mathbf{x}^{\g_{\mathcal{M}_s}^{\Lambda}}$ does not appear in the Laurent expansion of 
$CC_{\widetilde{\Lambda}}(\mathcal{M}_k)$ if $k\neq s$, and this forces $\lambda_s$ to be zero.
The theorem now follows from a straightforward inductive argument on $t$.
\end{proof}

For the next result, recall that $\Lambda_{\prin}$ and $R_{\prin}$ have been defined in Subsection \ref{subsec:assumptions}. See also Remark \ref{rem:Qin-correction-technique} as the theorem can be deduced from Qin's \emph{correction technique} when the binary relation $\preceq_B$ is a partial order.

\begin{thm}\label{thm:spanning-set-in-princ-coeffs=>spanning-set-in-arbitrary-coeffs} 
Take any set $\mathscr{M}$ of decorated representations of $\Lambda$ supported in $[n]$, and any decorated representation $\mathcal{N}$ of $\Lambda$ supported in $[n]$. If $CC_{\Lambda_{\prin}}(\mathcal{N})$ belongs to the $R_{\prin}$-submodule of $R_{\prin}[x_{1}^{\pm1},\ldots,x_{n}^{\pm1}]$ spanned by the set
$
CC_{\Lambda_{\prin}}(\mathscr{M}),
$
then $CC_{\widetilde{\Lambda}}(\mathcal{N})$ belongs to the $R$-submodule of $R[x_{1}^{\pm1},\ldots,x_{n}^{\pm1}]$ spanned by the set
$
CC_{\Lambda}(\mathscr{M}).
$
Consequently, if the set
$CC_{\Lambda_{\prin}}(\mathscr{M})$
spans the Caldero-Chapoton algebra $\calCC_{\Lambda_{\prin}}(\Lambda)$ over the ground ring $R_{\prin}$, then the set
$CC_{\Lambda}(\mathscr{M})$
spans the Caldero-Chapoton algebra $\calCC_{\widetilde{\Lambda}}(\Lambda)$ over the ground ring $R$.
\end{thm}

\begin{proof} Let $\varphi:\Q[x_1^{\pm1},\ldots,x_{n}^{\pm1},x_{n+1}^{\pm1},\ldots,x_{n+n}^{\pm1}]\rightarrow\Q[x_1^{\pm1},\ldots,x_{n}^{\pm1},x_{n+1}^{\pm1},\ldots,x_{n+m}^{\pm1}]$ be the $\Q$-algebra homomorphism defined by the rule
$$
\varphi(x_j)=\begin{cases}
x_j & \text{if $1\leq j\leq n$;}\\
\prod_{i=n+1}^{n+m}x_i^{\widetilde{b}_{ij}} & \text{if $n+1\leq j\leq n+n$}.
\end{cases}
$$

Suppose that the elements $\mathcal{M}_1,\ldots,\mathcal{M}_t\in\mathscr{M}$, and the Laurent polynomials $\lambda_1,\ldots,\lambda_t\in R_{\prin}$ satisfy
\begin{equation}\label{eq:CC_prin(M)=lin-comb-in-prin-coeffs}
CC_{\Lambda_{\prin}}(\mathcal{N})=\sum_{k=1}^t\lambda_tCC_{\Lambda_{\prin}}(\mathcal{M}_k).
\end{equation}
When we apply $\varphi$ to \eqref{eq:CC_prin(M)=lin-comb-in-prin-coeffs} we obtain
$$
\frac{\mathbf{x}^{\g_{\mathcal{N}}^{\Lambda}}}{\mathbf{x}^{\widetilde{\g}_{\mathcal{N}}^{\widetilde{\Lambda}}}}CC_{\widetilde{\Lambda}}(\mathcal{N})=\varphi(CC_{\Lambda_{\prin}}(\mathcal{N}))=\sum_{k=1}^t\varphi(\lambda_t)\varphi(CC_{\Lambda_{\prin}}(\mathcal{M}_k))=
\sum_{k=1}^t\varphi(\lambda_t)\frac{\mathbf{x}^{\g_{\mathcal{M}_k}^{\Lambda}}}{\mathbf{x}^{\widetilde{\g}_{\mathcal{M}_k}^{\widetilde{\Lambda}}}}CC_{\widetilde{\Lambda}}(\mathcal{M}_k)
$$
and hence,
$$
CC_{\widetilde{\Lambda}}(\mathcal{N})
=
\sum_{k=1}^t\varphi(\lambda_t)
\frac{\mathbf{x}^{\widetilde{\g}_{\mathcal{N}}^{\widetilde{\Lambda}}}}{\mathbf{x}^{\g_{\mathcal{N}}^{\Lambda}}}
\frac{\mathbf{x}^{\g_{\mathcal{M}_k}^{\Lambda}}}{\mathbf{x}^{\widetilde{\g}_{\mathcal{M}_k}^{\widetilde{\Lambda}}}}
CC_{\widetilde{\Lambda}}(\mathcal{M}_k)
$$
with each coefficient
\begin{equation}\label{eq:Laurent-monomial-coefficient}
\varphi(\lambda_t)
\frac{\mathbf{x}^{\widetilde{\g}_{\mathcal{N}}^{\widetilde{\Lambda}}}}{\mathbf{x}^{\g_{\mathcal{N}}^{\Lambda}}}
\frac{\mathbf{x}^{\g_{\mathcal{M}_k}^{\Lambda}}}{\mathbf{x}^{\widetilde{\g}_{\mathcal{M}_k}^{\widetilde{\Lambda}}}}\in R
\end{equation}
Theorem \ref{thm:spanning-set-in-princ-coeffs=>spanning-set-in-arbitrary-coeffs} follows.
\end{proof}

\begin{question}\label{question:non-negative-exponents?} Is it possible to give necessary and/or sufficient conditions that guarantee that all of the Laurent polynomials $\varphi(\lambda_t)
\frac{\mathbf{x}^{\widetilde{\g}_{\mathcal{M}}}}{\mathbf{x}^{\g_{\mathcal{M}}}}
\frac{\mathbf{x}^{\g_{Z_k}}}{\mathbf{x}^{\widetilde{\g}_{Z_k}}}$ in \eqref{eq:Laurent-monomial-coefficient} belong to the polynomial ring $\Q[x_{n+1},\ldots,x_{n+m}]$?
\end{question}

A positive answer to Question \ref{question:non-negative-exponents?} would help to find necessary and/or sufficient conditions in order for the spanning of $\calCC_{\widetilde{\Lambda}}$ by $\mathscr{M}$ to take place not only over $R=\Q[x_{n+1}^{\pm1},\ldots,x_{n+m}^{\pm1}]$, but actually over $\Q[x_{n+1},\ldots,x_{n+m}]$, which is stronger and harder to prove.

\begin{ex} Let $(\widetilde{Q},\widetilde{S})$ and $(Q,S)$ be as in Example \ref{ex:two-triangles}, $I(2)$ the indecomposable injective representation of $(Q,S)$ at the vertex $2$, and $S(1)$ the simple  representation of $(Q,S)$ at the vertex $1$. Both $S(1)$ and $I(2)$ are supported in $[3]$, hence so is the decorated representation $\mathcal{N}=(S(1)\oplus I(2))$. Direct independent computations show that
\begin{eqnarray}\nonumber
CC_{(Q_{\prin},S)}(\mathcal{N}) &=& x_1^{-1}+x_1^{-1}x_2^{-1}x_3y_1+x_2^{-1}y_1y_2+x_3^{-1}y_2+x_1^{-1}x_2x_3^{-1}y_2y_3+x_1^{-1}y_1y_2y_3\\
 \label{eq:example-specializing-princ-princ}
 &=& \left(x_1^{-1}+x_1^{-1}x_2^{-1}x_3y_1+x_2^{-1}y_1y_2\right)+y_2\left(x_3^{-1}+x_1^{-1}x_2x_3^{-1}y_3+x_1^{-1}y_1y_3\right)\\
\nonumber
 &=& CC_{(Q_{\prin},S)}(I(1),0) + y_2CC_{(Q_{\prin},S)}(I(3),0)\\
 \nonumber
 \text{and} \ \ \ \ \ 
CC_{(\widetilde{Q},\widetilde{S})}(\mathcal{N}) & = & x_1^{-1}x_4 +x_1^{-1}x_2^{-1}x_3x_4^{2} +x_2^{-1}x_4+x_3^{-1}+x_1^{-1}x_2x_3^{-1}+x_1^{-1}x_4 \\
\label{eq:example-specializing-princ-specialized}
&=& x_4\left(x_1^{-1}+x_1^{-1}x_2^{-1}x_3x_4+x_2^{-1}\right)+\left(x_3^{-1}+x_1^{-1}x_2x_3^{-1}+x_1^{-1}x_4\right)\\
\nonumber
&=& x_4CC_{(\widetilde{Q},\widetilde{S})}(I(1),0)+CC_{(\widetilde{Q},\widetilde{S})}(I(3),0),
\end{eqnarray}
where we have written $y_1:=x_{3+1}$, $y_2:=x_{3+2}$ and $y_3:={x_3+3}$.
The $\Q$-algebra homomorphism $\varphi:\Q[x_1^{\pm1},x_2^{\pm1},x_{3}^{\pm1},y_{1}^{\pm1},y_2^{\pm1},y_{3}^{\pm1}]\rightarrow\Q[x_1^{\pm1},x_2^{\pm1},x_{3}^{\pm1},x_{4}^{\pm1}]$ from the proof of Theorem \ref{thm:spanning-set-in-princ-coeffs=>spanning-set-in-arbitrary-coeffs} is given by
$$
x_1\mapsto x_1, \ \ \  x_2\mapsto x_2, \ \ \  x_3\mapsto x_3, \ \ \  y_1\mapsto x_4, \ \ \  y_2\mapsto x_4^{-1}, \ \ \  y_3\mapsto 1.
$$
It is obvious that $\varphi(CC_{(Q_{\prin},S)}(I(1),0))=CC_{(\widetilde{Q},\widetilde{S})}(I(1),0)$,  $\varphi(CC_{(Q_{\prin},S)}(I(3),0))=CC_{(\widetilde{Q},\widetilde{S})}(I(3),0)$, and $\varphi(CC_{(Q_{\prin},S)}(\mathcal{N}))=x_4^{-1}CC_{(\widetilde{Q},\widetilde{S})}(\mathcal{N})$. The relation between the expressions \eqref{eq:example-specializing-princ-princ} and \eqref{eq:example-specializing-princ-specialized} used in the proof of Theorem \ref{thm:spanning-set-in-princ-coeffs=>spanning-set-in-arbitrary-coeffs} is
\begin{eqnarray*}
CC_{(\widetilde{Q},\widetilde{S})}(\mathcal{N}) &=&x_4CC_{(\widetilde{Q},\widetilde{S})}(I(1),0)+CC_{(\widetilde{Q},\widetilde{S})}(I(3),0)
\\ &=&
x_4CC_{(\widetilde{Q},\widetilde{S})}(I(1),0)+x_4\varphi(y_2)CC_{(\widetilde{Q},\widetilde{S})}(I(3),0)\\
&=&
x_4\varphi(CC_{(Q_{\prin},S)}(I(1),0))+x_4\varphi(y_2)\varphi(CC_{(Q_{\prin},S)}(I(3),0))\\
&=&
x_4\varphi(CC_{(Q_{\prin},S)}(I(1),0)+y_2CC_{(Q_{\prin},S)}(I(3),0))\\
&=&
x_4\varphi(CC_{(Q_{\prin},S)}(\mathcal{N}))\\
&=&
\frac{\widetilde{\g}_{\mathcal{N}}}{\g_{\mathcal{N}}}\varphi(CC_{(Q_{\prin},S)}(\mathcal{N})).
\end{eqnarray*}

Note that, although injective $\jacobalg{Q}{S}$-modules are not always injective in $\jacobalg{\widetilde{Q}}{\widetilde{S}}$-$\modcat$ (e.g., in the current example, $I(2)$ is a proper $\jacobalg{\widetilde{Q}}{\widetilde{S}}$-submodule of the indecomposable injective $\jacobalg{\widetilde{Q}}{\widetilde{S}}$-module $\widetilde{I}(2)$), in this example $I(1)$ and $I(3)$ remain injective in $\jacobalg{\widetilde{Q}}{\widetilde{S}}$-$\modcat$, and hence
$$
\frac{\mathbf{x}^{\g_{I(1)}}}{\mathbf{x}^{\widetilde{\g}_{I(1)}}}=1=\frac{\mathbf{x}^{\g_{I(3)}}}{\mathbf{x}^{\widetilde{\g}_{I(3)}}},
$$
which is why it seems that only $\frac{\widetilde{\g}_{\mathcal{N}}}{\g_{\mathcal{N}}}$ appears in the relation between the expansions of $CC_{(Q_{\prin},S)}(\mathcal{N})$ and $CC_{(\widetilde{Q},\widetilde{S})}(\mathcal{N})$ in terms of the corresponding CC-functions of $(I(1),0)$ and $(I(3),0)$.
\end{ex}

\section{Irreducible components of representation spaces and their mutations}\label{sec:irreducible-components}

\subsection{Generic $E$-invariant, generic codimension and generically $\tau^-$-reduced irreducible components}\label{subsec:generically-tau--reduced-comps}

\begin{defi} For a vector $(\bfd,\bfv)\in\ZZ_{\geq 0}^n\times\ZZ_{\geq 0}^n$, define
\begin{eqnarray*}
\rep(\Lambda,\bfd) &:=& \{M\in\prod_{a\in Q_1}\Hom_{\C}(\C^{d_{s(a)}},\C^{d_{t(a)}})\suchthat \text{$M$ is annihilated by every element of $\mathcal{J}$}\}\\
\decrep(\Lambda,\bfd,\bfv) &:=& \rep(\Lambda,\bfd)\times\{\bfv\}\\
\decirr(\Lambda,\bfd,\bfv) &:=& \{Z\suchthat \text{$Z$ is an irreducible component of $\decrep(\Lambda,\bfd,\bfv)$}\}\\
\text{Furthermore,} && \\
\decirr(\Lambda) &:=& \bigcup_{(\bfd,\bfv)\in\ZZ_{\geq 0}^n\times\ZZ_{\geq 0}^n}\decirr(\Lambda,\bfd,\bfv).
\end{eqnarray*}
\end{defi}

\begin{defi}\cite[Equations (7.4) and (7.6)]{DWZ2} For decorated representations $\mathcal{M}=(M,\bfv)$ and $\mathcal{N}=(N,\mathbf{w})$ of $\Lambda$ supported in $[n]$, the \emph{$E$-invariant} is defined as
\begin{eqnarray*}
E_{\Lambda}(\mathcal{M},\mathcal{N}) &:=& \dim_{\C}(\Hom_{\Lambda}(M,N))+\langle\underline{\dim}(M),\g_{\mathcal{N}}^{\Lambda}\rangle\\
E_{\Lambda}(\mathcal{M}) &:=& E_{\Lambda}(\mathcal{M},\mathcal{M}).
\end{eqnarray*}
\end{defi}

\begin{defi}\cite[Section 8.1]{GLS3} (See also \cite[Section 5.1]{CLS}) For $Z\in \decirr(\Lambda,\bfd,\bfv)$ we define
\begin{eqnarray*}
c_{\Lambda}(Z) &:=& \min\{\dim(Z)-\dim(\GL_{\bfd}\cdot\mathcal{M})\suchthat\mathcal{M}\in Z\}.
\end{eqnarray*}
\end{defi}

The following proposition is well known. See, for instance, \cite{CLS,GLS3,Plamondon-generic}.

\begin{prop}\label{prop:every-irr-Z-has-subset-where-g-vects-CC-map-E-inv-and-codim-are-constant} For every $(\bfd,\bfv)\in\ZZ_{\geq 0}^n\times\ZZ_{\geq 0}^n$, the functions
\begin{center}
\begin{tabular}{cccc}
$\decrep(\Lambda,\bfd,\bfv) \rightarrow \ZZ^{n}$, &
$\decrep(\Lambda,\bfd,\bfv) \rightarrow \ZZ^{n+m}$, &
$\decrep(\Lambda,\bfd,\bfv) \rightarrow \ZZ$, &
$\decrep(\Lambda,\bfd,\bfv) \rightarrow \calF$\\
$\mathcal{M}\mapsto\g_{\mathcal{M}}^{\Lambda}$, & 
$\mathcal{M}\mapsto\widetilde{\g}_{\mathcal{M}}^{\widetilde{\Lambda}}$, & 
$\mathcal{M}\mapsto E_{\Lambda}(\mathcal{M})$, & 
$\mathcal{M}\mapsto CC_{\widetilde{\Lambda}}(\mathcal{M})$,
\end{tabular}
\end{center}
are constructible, and the function
\begin{center}
\begin{tabular}{cc}
$\decrep(\Lambda,\bfd,\bfv) \rightarrow \ZZ$, & $\mathcal{M}\mapsto -\dim(\GL_{\bfd}\cdot\mathcal{M})$,
\end{tabular}
\end{center}
is upper semicontinuous. Consequently, each
$Z\in \decirr(\Lambda,\bfd,\bfv)$ has a dense open subset $U\subseteq Z$ such that
\begin{enumerate}
\item\label{item:every-irr-Z-has-subset-where-g-vects-CC-map-and-E-inv-are-constant} the values of $\g_{-}^{\Lambda}$ $\widetilde{\g}_{-}^{\widetilde{\Lambda}}$, $E_{\Lambda}(-)$ and $CC_{\widetilde{\Lambda}}(-)$ are constant on $U$;
\item $c_{\Lambda}(Z)=\dim(Z)-\dim(\GL_{\bfd}\cdot\mathcal{M})$ for every $\mathcal{M}\in U$.
\end{enumerate}
\end{prop}

In Subsection \ref{subsec:upper-semicontinuity-of-E-invariant} we will strengthen 
Proposition \ref{prop:every-irr-Z-has-subset-where-g-vects-CC-map-E-inv-and-codim-are-constant} by showing that the entries of $\g_{-}^{\Lambda}$ are upper semicontinuous functions, and that, consequently, $E_{\Lambda}(-):\decrep(\Lambda,\bfd,\bfv)\rightarrow \ZZ$ is upper semicontinuous, not only constructible; see Lemma \ref{lemma:entries-of-g-vector-are-upper-semicontinuous} and Theorem \ref{thm:E-invariant-is-upper-semicontinuous}. This will play an essential role in the culmination of Section \ref{subsec:behavior-of-irred-comps-under-mutation}.

\begin{defi} Let $Z\in \decirr(\Lambda)$ be an irreducible component.
\begin{enumerate}
\item The \emph{generic $\g$-vector} $\g_Z$, the \emph{generic extended $\g$-vector} $\widetilde{\g}_Z$, the \emph{generic Caldero-Chapoton function with coefficients $CC_{\widetilde{\Lambda}}(Z)$}, and the \emph{generic $E$-invariant $E_{\Lambda}(Z)$} are the values referred to in the previous proposition;
\item $Z$ is \emph{strongly reduced}, or \emph{generically $\tau^-$-reduced}, if $c_{\Lambda}(Z)=E_{\Lambda}(Z)$.
\end{enumerate}
\end{defi}

\begin{remark}\label{rem:inequality-c-leq-E} According to \cite[Section 8.1]{GLS3} and \cite[Lemma 5.2]{CLS}, for every irreducible component $Z\in \decirr(\Lambda)$ one always has $c_{\Lambda}(Z)\leq E_{\Lambda}(Z)$.
\end{remark}

\begin{thm}\label{thm:different-s.r.i.c.-have-different-g-vectors}\cite{CLS,Plamondon-generic} Take $Z_1,Z_2\in\decirrsr(\Lambda)$. If $Z_1\neq Z_2$, then the generic $\g$-vectors $\g_{Z_1}^{\Lambda}$ and $\g_{Z_2}^{\Lambda}$ are different.
\end{thm}

\begin{defi} The \emph{set of generically $\tau^{-}$-reduced Caldero-Chapoton functions}, or simply \emph{set of generic Caldero-Chapoton functions}, of $\Lambda$ with respect to $\widetilde{\Lambda}$,~is
\begin{eqnarray*}
\mathcal{B}_{\widetilde{\Lambda}}(\Lambda)&:=&\{CC_{\widetilde{\Lambda}}(Z)\suchthat Z\in \decirrsr(\Lambda)\},\\
\text{where} \ \ \ \decirrsr(\Lambda)&:=&\{Z\suchthat Z\in \decirr(\Lambda,\bfd,\bfv) \ \text{for some} \ (\bfd,\bfv)\in\ZZ_{\geq0}^n\times\ZZ_{\geq0}^n\\
&& \phantom{\{Z\suchthat} \text{and $Z$ is generically $\tau^-$-reduced}\}.
\end{eqnarray*}
\end{defi}

\subsection{Upper semicontinuity of the $\mathbf{g}$-vector and of the $E$-invariant}
\label{subsec:upper-semicontinuity-of-E-invariant}

This subsection is devoted to strengthening Proposition \ref{prop:every-irr-Z-has-subset-where-g-vects-CC-map-E-inv-and-codim-are-constant}.

\begin{lemma}\label{lemma:entries-of-g-vector-are-upper-semicontinuous} For each decorated dimension vector $(\bfd,\bfv)\in\ZZ_{\geq0}^{n}\times\ZZ_{\geq0}^n$ and each vertex $k$ of $Q$, the function $\widetilde{g}_{-,k}^{\widetilde{\Lambda}}:\decrep(\Lambda,\bfd,\bfv)\rightarrow \ZZ$, $\mathcal{M}\mapsto \widetilde{g}_{\mathcal{M},k}$, that assigns to each decorated representation $\mathcal{M}=(M,\bfv)$ the $k^{\operatorname{th}}$ entry of its extended $\g$-vector $\widetilde{\g}_{\mathcal{M}}^{\widetilde{\Lambda}}$, is upper semicontinuous.
\end{lemma}

\begin{proof} Let $p$ be an integer greater than the sum of the entries of $\bfd$, and let $\widetilde{\Lambda}_p$ be the $p^{\operatorname{th}}$ truncation of $\widetilde{\Lambda}$. Then $\dim_{\C}(\widetilde{\Lambda}_p)<\infty$ and $\decrep(\Lambda,\bfd,\bfv)=\decrep(\Lambda_p,\bfd,\bfv)$ by \cite[Lemma 2.2]{CLS}. Fix a vertex $k$ of $Q$, and let $\varepsilon:P_0\rightarrow S(k)$ be a projective cover in $\widetilde{\Lambda}_p$-$\modcat$ of the $k^{\operatorname{th}}$ simple $S(k)$. Thus, $P_0$ is precisely the $k^{\operatorname{th}}$ indecomposable projective $\widetilde{\Lambda}_p$-module. Complete $\varepsilon$ to a projective resolution of $S(k)$ in $\widetilde{\Lambda}_p$-$\modcat$:
$$
\xymatrix{
P_{\bullet}: &\cdots \ar[r]^{\partial_3} & P_2 \ar[r]^{\partial_2} & P_1\ar[r]^{\partial_1} & P_0 \ar[r]^{\varepsilon} & S(k) \ar[r]& 0.
}
$$

Using the cochain complex
$$
\xymatrix{
0 \ar[r] & \Hom_{\widetilde{\Lambda}_p}(P_0,M) \ar[r]^{\partial_1^M} & \Hom_{\widetilde{\Lambda}_p}(P_1,M) \ar[r]^{\partial_2^M} & \Hom_{\widetilde{\Lambda}_p}(P_2,M) \ar[r]^{\partial_3^M} & \cdots\phantom{\Hom(P_2,M)}   
}
$$
to compute the groups $\Ext^i_{\widetilde{\Lambda}_p}(S(k),M)$, where $\partial_i^M(f):=f\circ \partial_i$ for $f\in \Hom_{\widetilde{\Lambda}_p}(P_{i-1},M)$, and recalling that $\dim_{\C}(\Ext^1_{\widetilde{\Lambda}}(S(k),M))=\dim_{\C}(\Ext^1_{\widetilde{\Lambda}_p}(S(k),M))$ by \cite[Lemma 2.2]{CLS}, for $\mathcal{M}=(M,\bfv)\in\decrep(\Lambda_p,\bfd,\bfv)$ we see that
\begin{eqnarray*}
\widetilde{g}_{\mathcal{M},k}^{\widetilde{\Lambda}}&=& -\dim_{\C}(\Hom_{\widetilde{\Lambda}}(S(k),M))+\dim_{\C}(\Ext^1_{\widetilde{\Lambda}}(S(k),M))+v_k\\
&=& -\dim_{\C}(\Hom_{\widetilde{\Lambda}_p}(S(k),M))+\dim_{\C}(\Ext^1_{\widetilde{\Lambda}_p}(S(k),M))+v_k\\
&=& -\dim_{\C}(\Ker(\partial_1^M)) + \dim_{\C}(\Ker(\partial_2^M)) - \dim_{\C}(\Image(\partial_1^M)) + v_k\\
&=& -\dim_{\C}(\Hom_{\widetilde{\Lambda}_p}(P_0,M)) + \dim_{\C}(\Ker(\partial_2^M)) + v_k \\
&=& -d_k+v_k + \dim_{\C}(\Ker(\partial_2^M)).
\end{eqnarray*}
The lemma will thus follow once the upper semicontinuity of the function $\decrep(\Lambda_p,\bfd,\bfv)\rightarrow \ZZ$ given by $(M,\bfv)\mapsto \dim_{\C}(\Ker(\partial_2^M))$ is established.

Write $\underline{\dim}(P_1)=(\delta_j^{P_1})_{j\in \widetilde{Q}_0}$ and $\underline{\dim}(P_2)=(\delta_j^{P_2})_{j\in \widetilde{Q}_0}$, so that $\Hom_{\widetilde{\Lambda}_p}(P_1,M)\subseteq\prod_{j\in \widetilde{Q}_0}\C^{d_j\times\delta_j^{P_1}}$, $\Hom_{\widetilde{\Lambda}_p}(P_2,M)\subseteq\prod_{j\in \widetilde{Q}_0}\C^{d_j\times\delta_j^{P_2}}$ and $\partial_2=(\partial_{2,j})_{j\in\widetilde{Q}_0}\in\prod_{j\in \widetilde{Q}_0}\C^{\delta_j^{P_1}\times \delta_j^{P_2}}$. We see that for all $(M,\bfv)\in\decrep(\Lambda_p,\bfd,\bfv)$ and all $f=(f_j)_{j\in\widetilde{Q}_0}\in \Hom_{\widetilde{\Lambda}_p}(P_{1},M)$, the value $\partial_2^M(f):=f\circ \partial_2=(f_j\partial_{2,j})_{j\in\widetilde{Q}_0}$, where $f_j\partial_{2,j}$ is the plain product of matrices. It is fairly easy to verify that
\begin{eqnarray*}
K &:=& \{((M,\bfv),f)\in \decrep(\Lambda_p,\bfd,\bfv)\times \prod_{j\in \widetilde{Q}_0}\C^{d_j\times\delta_j^{P_1}}\suchthat f\in \Hom_{\widetilde{\Lambda}_p}(P_{1},M) 
\ \text{and}\\
&&\phantom{\{((M,\bfv),f)\in \decrep(\Lambda_p,\bfd,\bfv)\times \prod_{j\in \widetilde{Q}_0}\C^{d_j\times\delta_j^{P_1}}\suchthat}
f_j\partial_{2,j}=0 \ \text{for all} \ j\in\widetilde{Q}_0\}
\end{eqnarray*}
is a closed subset of $\decrep(\Lambda_p,\bfd,\bfv)\times \prod_{j\in \widetilde{Q}_0}\C^{d_j\times\delta_j^{P_1}}$. For each $(M,\bfv)\in \decrep(\Lambda_p,\bfd,\bfv)$, we have $\Ker(\partial_2^M))=\{f\in \prod_{j\in \widetilde{Q}_0}\C^{d_j\times\delta_j^{P_1}}\suchthat ((M,\bfv),f)\in K\}$. Applying \cite[Lemma 4.2]{CB-S} we obtain the upper semicontinuity of the function $\decrep(\Lambda_p,\bfd,\bfv)\rightarrow \ZZ$, $(M,\bfv)\mapsto \dim_{\C}(\Ker(\partial_2^M))$. Lemma \ref{lemma:entries-of-g-vector-are-upper-semicontinuous} is proved.
\end{proof}

{
\begin{remark}\label{rem:projective-g-vectors-upper-semicontinuity} The notion of (extended) $\g$-vectors we are using in this paper is related to minimal injective presentations. There is another notion of (extended) $\g$-vectors, related to minimal projective presentations. The entries of the (extended) \emph{projective $\g$-vector} are upper semicontinuous functions too. The proof can be given directly or as a consequence of Lemma \ref{lemma:entries-of-g-vector-are-upper-semicontinuous}, using the (not completely trivial) fact that the projective $\g$-vector of $\mathcal{M}=(M,\bfv)\in\decrep(\Lambda,\bfd,\bfv)$ coincides with the $\g$-vector $g_{\mathcal{M}^{\star}}^{\Lambda^{\operatorname{op}}}$ of the dual $\mathcal{M}^{\star}:=(\Hom_{\C}(M,\C),\bfv)\in\decrep(\Lambda^{\operatorname{op}},\bfd,\bfv)$, and the fact that matrix transposition defines an isomorphism of affine varieties $\decrep(\Lambda,\bfd,\bfv)\rightarrow\decrep(\Lambda^{\operatorname{op}},\bfd,\bfv)$. 
\end{remark}
}

We now prove that the $E$-invariant gives upper semicontinuous functions.

\begin{thm}\label{thm:E-invariant-is-upper-semicontinuous}
For all decorated dimension vectors $(\bfd,\bfv), (\bfd',\bfv')\in\ZZ_{\geq0}^{n}\times\ZZ_{\geq0}^n$,
the functions
\begin{align*}
\decrep(\Lambda,\bfd,\bfv) &\to \Z 
&\text{and}&&
\decrep(\Lambda,\bfd,\bfv) \times \decrep(\Lambda,\bfd',\bfv')&\to \Z 
\\
\mathcal{M} &\mapsto E_{\Lambda}(\mathcal{M},\mathcal{M})
&&&
(\mathcal{M},\mathcal{N}) &\mapsto E_{\Lambda}(\mathcal{M},\mathcal{N})
\end{align*}
are upper semicontinuous.
In particular, if $A$ is a finite-dimensional $\C$-algebra, then for all dimension vectors $\bfd,\bfd'\in\ZZ_{\geq0}^{n}$ the functions
\begin{align*}
\rep(A,\bfd) &\to \Z 
&\text{and}&&
\rep(A,\bfd) \times \rep(A,\bfd')&\to \Z 
\\
M &\mapsto \dim\Hom_A(\tau^-_A(M),M) 
&&&
(M,N) &\mapsto  \dim \Hom_A(\tau^-_A(N),M)
\end{align*}
as well as the functions
\begin{align*}
\rep(A,\bfd) &\to \Z 
&\text{and}&&
\rep(A,\bfd) \times \rep(A,\bfd')&\to \Z 
\\
M &\mapsto \dim\Hom_A(M,\tau_A(M)) 
&&&
(M,N) &\mapsto  \dim \Hom_A(N,\tau_A(M))
\end{align*}
are upper semicontinuous.
\end{thm}

\begin{proof} The functions
\begin{center}
\begin{tabular}{rl}
$\decrep(\Lambda,\bfd,\bfv)\rightarrow\ZZ$, & $(M,\bfv)\mapsto\dim_{\C}(\Hom_{\jacobalg{Q}{S}}(M,M))$,\\
 $\decrep(\Lambda,\bfd,\bfv) \times \decrep(\Lambda,\bfd',\bfv')\rightarrow\ZZ$, & $((M,\bfv),(N,\bfv'))\mapsto\dim_{\C}(\Hom_{\jacobalg{Q}{S}}(M,N))$,
 \end{tabular}
\end{center}
are upper semicontinuous by \cite[Lemma 4.3]{CB-S}. Since $\bfd$ is a tuple of non-negative integers, this and Lemma \ref{lemma:entries-of-g-vector-are-upper-semicontinuous} imply that the functions
$$
E_{\Lambda}(-)=\dim_{\C}(\Hom_{\Lambda}(-,-))+\langle\bfd,\g_{-}^{\Lambda}\rangle:\decrep(\Lambda,\bfd,\bfv)\rightarrow \ZZ
$$
 and 
$$
E_{\Lambda}(-,\bullet)=\dim_{\C}(\Hom_{\Lambda}(-,\bullet))+\langle\bfd,\g_{\bullet}^{\Lambda}\rangle:\decrep(\Lambda,\bfd,\bfv)\times\decrep(\Lambda,\bfd',\bfv')\rightarrow \ZZ,
$$
being sums of upper semicontinuous functions, are themselves upper semicontinuous.

Suppose $A$ is a finite-dimensional $\C$-algebra. The upper semicontinuity of the functions 
$M \mapsto \dim_{\C}(\Hom_A(\tau^-_A(M),M) )$, and
$(M,N) \mapsto  \dim_{\C} (\Hom_A(\tau^-_A(N),M))$, follows by combining the already established upper semicontinuity of $E_{\Lambda}(-)$ and $E_{\Lambda}(-,\bullet)$ with \cite[Proposition 3.5]{CLS}.

Following Derksen-Weyman-Zelevinsky \cite[first paragraph of Section 10]{DWZ2},
define
$$
E_{A}^{\operatorname{proj}}((M,\bfv),(N,\bfv')):=E_{A^{\operatorname{op}}}((N^{\star},\bfv'),(M^{\star},\bfv)),
$$
where $M^{\star}:=\Hom_{\C}(M,\C)=(M_a^{\operatorname{t}})_{a\in Q_1}\in \rep(A^{\operatorname{op}},\bfd)\subseteq\prod_{a\in Q_1}\C^{d_{s(a)}\times d_{t(a)}}$ (we use the notation $M_a^{\operatorname{t}}$ for the transpose of the matrix $M_a$). Just like \cite[Equation (10.1)]{DWZ2}, one easily sees that
\begin{equation}\label{eq:E^proj-easy-equality}
E_{A}^{\operatorname{proj}}((M,\bfv),(N,\bfv'))=\dim_{\C}(\Hom_{A}(M,N))+\langle g_{\mathcal{M}^{\star}}^{A^{\operatorname{op}}},\underline{\dim}N\rangle,
\end{equation}
where $\mathcal{M}^{\star}=(M^{\star},\bfv)$. Using the results from \cite[Section 10]{DWZ2} (when $A$ is Jacobian) and dualizing the results from \cite[Section 3]{CLS} (for $A$ not necessarily Jacobian), one proves that
\begin{equation}\label{eq:E^proj-hard-equality}
E_{A}^{\operatorname{proj}}((M,\bfv),(N,\bfv'))=\dim_{\C}(\Hom_A(N,\tau_A(M)))+\langle\bfv,\bfd'\rangle.
\end{equation}
Using Lemma \ref{lemma:entries-of-g-vector-are-upper-semicontinuous} (see also Remark \ref{rem:projective-g-vectors-upper-semicontinuity}) and Equations \eqref{eq:E^proj-easy-equality} and \eqref{eq:E^proj-hard-equality}, we see that the functions $\dim_{\C}(\Hom_A(-,\tau_A(-))):\rep(A,\bfd)\rightarrow \Z$ and $\dim_{\C}(\Hom_A(\bullet,\tau_A(-))):\rep(A,\bfd) \times \rep(A,\bfd')\rightarrow \Z$ are sums of upper semicontinuous functions. This finishes the proof of Theorem \ref{thm:E-invariant-is-upper-semicontinuous}.
\end{proof}

\begin{remark}\label{rem:E(Z)<=E(M)-for-M-in-Z} 
Given a topological space $X$ and an upper semicontinuous function $f:X\rightarrow \ZZ_{\geq 0}$, for every non-empty irreducible subset $Y\subseteq X$ one has the equality
$$
\min\{f(y)\suchthat y\in \overline{Y}\}=\min\{f(y)\suchthat y\in Y\},
$$
where $\overline{Y}$ is the topological closure of $Y$ in $X$. It is through this property of upper semicontinuous functions that Theorem \ref{thm:E-invariant-is-upper-semicontinuous} will play an essential role in Subsection \ref{subsec:behavior-of-irred-comps-under-mutation}, where we will analyze the behavior of the generically $\tau^{-}$-reduced components and their Caldero-Chapoton functions under Derksen-Weyman-Zelevinsky's mutation of decorated representations.
\end{remark}

\subsection{Behavior under mutation}
\label{subsec:behavior-of-irred-comps-under-mutation}

In this subsection we will show that for every loop-free quiver with potential $(\widetilde{Q},\widetilde{S})$, not necessarily non-degenerate, and every mutable vertex $k$ of $\widetilde{Q}$ not incident to any 2-cycle of $\widetilde{Q}$, one has the equality $\mathcal{B}_{(\widetilde{Q},\widetilde{S})}(Q,S)=\mathcal{B}_{\mu_k(\widetilde{Q},\widetilde{S})}(Q,S)$. We will not suppose finite-dimensionality of any of the Jacobian algebras involved. See Remark \ref{rem:some-history-of-generic-CC-functions}:\eqref{item:rem-some-some-history-of-generic-CC-functions-GLS}-\eqref{item:rem-some-some-history-of-generic-CC-functions-Plamondon} and Remark \ref{rem:comparison-to-previous-approaches-intro}:\eqref{item:rem-comparison-to-previous-approaches-intro-GLS-and-Plamodon}-\eqref{item:rem-comparison-to-previous-approaches-intro-Linear-Alg-vs-cluster-categories}.
 
Our techniques will be rather elementary: we will use only basic linear algebra and elementary algebraic geometry of minors.
 
The next two lemmas are well-known. The matrix spaces $\C^{m\times n}$ will always be taken with the Zariski topology.
 
 \begin{lemma}\label{lemma:lower-semicontinuity-of-rank} For any topological space $Z$ and any continuous map $\rho:Z\rightarrow \C^{m\times n}$, the function $Z\rightarrow\Z_{\geq0}$ given by $x\mapsto\dim_{\C}(\Image \rho(x))$ is lower semicontinuous. Consequently, if $Z$ is irreducible, then there exists a dense open subset $U=U_\rho$ of $Z$ with the property that the number $\rank(\rho(Z)):=\dim_{\C}(\Image \rho(x))$ is constant as we let $x$ vary inside $U$. Furthermore, for an arbitrary element $y\in Z$ we have $\rank(\rho(Z))\geq\dim_{\C}(\Image \rho(y))$.
 \end{lemma}
 
Recall that, given $m,n\in\Z_{\geq 0}$, $r\in\{1,\ldots,\min(m,n)\}$, and subsets $I\subseteq[m]$ and $J\subseteq[n]$, both of size $r$, the $r\times r$ \emph{minor} $\Delta_{I,J}:\C^{m\times n}\rightarrow \C$ is the function given by $\Delta(A):=\det(A_{I,J})$ for $A=(a_{ij})_{i\in[m],j\in[n]}\in\C^{m\times n}$, where $A_{I,J}:=(a_{ij})_{i\in I,j\in J}\in\C^{r\times r}$. Recall also that there exist permutation matrices $P_I=P_{m,I}\in\C^{m\times m}$, $P_J=P_{n,J}\in\C^{n\times n}$, not necessarily unique, such that $A_{I,J}=(P_IAP_J^{-1})_{[r],[r]}$ and hence $\Delta_{I,J}(A)=\Delta_{[r],[r]}(P_IAP_J^{-1})$ for all $A\in\C^{m\times n}$.

 \begin{lemma}\label{lemma:linear-algebra-regular-maps} Take $m,n\in\Z_{\geq 0}$, and let $Z$ be an irreducible affine variety and $\rho:Z\rightarrow \C^{m\times n}$ be a morphism of affine varieties. Set $U=U_\rho$ to be any open dense subset of $Z$ as in Lemma \ref{lemma:lower-semicontinuity-of-rank} and $p_{n-r,n}:=\left[\begin{array}{cc}0_{(n-r)\times r} & \myid_{(n-r)\times (n-r)}\end{array}\right]\in\C^{(n-r)\times n}$, where $r:=\rank(\rho(Z))$. Then, for each $r\times r$ minor $\Delta_{I,J}: \C^{m\times n}\rightarrow\C$: 
 \begin{enumerate}
 \item\label{item:lemmalinalg-openset} The set $U_{\rho,\Delta_{I,J}}:=\{x\in U\suchthat \Delta_{I,J}(\rho(x))\neq 0\}$ is open in $Z$.
 \item\label{item:lemmalinalg-basisofkernel} There is a regular function $\varphi_{\rho,\Delta_{I,J}}:U_{\Delta_{I,J}}\rightarrow \C^{n\times (n-r)}$ with the property that for every $x\in U_{\Delta_{I,J}}$, the columns of $\varphi_{\rho,\Delta_{I,J}}(x)$ form a basis of $\Ker(\rho(x))$, the bottom $(n-r)\times (n-r)$ submatrix of $P_J\cdot\varphi_{\rho,\Delta_{I,J}}(x)$ is the identity matrix, and the diagram
$$
\xymatrix{
\C^{n-r} \ar[rrr]^{\varphi_{\rho,\Delta_{I,J}}(x)} \ar[d]_{T_{\rho,\Delta_{I,J}}(x)} & & &\C^{n} \ar[d]^{p_{n-r,n}P_J}\\
\Ker(\rho(x)) \ar[urrr]_{\iota_\rho(x)=\operatorname{incl}} \ar[rrr]_{T_{\rho,\Delta_{I,J}}(x)^{-1}} & & & \C^{n-r}
}
$$
commutes, where $T_{\rho,\Delta_{I,J}}(x):\C^{n-r}\rightarrow\Ker(\rho(x))$ is the tautological $\C$-vector space isomorphism that takes any given $(n-r)$-tuple of scalars and forms with them the corresponding $\C$-linear combination of the $n-r$ columns of the matrix $\varphi_{\rho,\Delta_{I,J}}(x)$. In particular, $(T_{\rho,\Delta_{I,J}}(x))\circ p_{n-r,n}\circ P_J\circ (\iota_{\rho}(x))=\myid_{\Ker(\rho(x))}$ for every $x\in U_{\Delta_{I,J}}$.
\item\label{item:lemmalinalg-quotientmap} There is a regular function $\psi_{\rho,\Delta_{I,J}}:U_{\Delta_{I,J}}\rightarrow \C^{(m-r)\times m}$ with the
property that for each $x\in U_{\Delta_{I,J}}$, the columns of $(\psi_{\rho,\Delta_{I,J}}(x))^{\operatorname{t}}$
form a basis of $\Ker((\rho(x))^{\operatorname{t}})$, the rightmost $(m-r)\times(m-r)$ submatrix of
$\psi_{\rho,\Delta_{I,J}}(x)\cdot P_I^{-1}$ is the identity matrix, and 
$\Ker(\psi_{\rho,\Delta_{I,J}}(x))=\Image(\rho(x))$.
 \end{enumerate}
Furthermore, $U_\rho=\bigcup_{I,J}U_{\Delta_{I,J}}$ (the union, running over all possible ordered pairs $(I,J)$ consisting of $r$-element subsets $I\subseteq[m]$, $J\subseteq[m]$ --there is a total of $\binom{m}{r}\binom{n}{r}$ different such pairs).
 \end{lemma}

Fix a decorated dimension vector $(\bfd,\bfv)\in\ZZ_{\geq 0}^n\times\ZZ_{\geq 0}^n$ and consider the closed subset
\begin{equation}\label{eq:decrep-iso-to-closed-set-of-matrices}
\decrep(\jacobalg{Q}{S},\bfd,\bfv)\subseteq \left(\prod_{a\in Q_1}\Hom_{\C}(\C^{d_{s(a)}},\C^{d_{t(a)}})\right)\times\{\bfv\}\cong\prod_{a\in Q_1}\C^{d_{t(a)}\times d_{s(a)}}
\end{equation}
Fix also $k\in[n]$ and assume that
\begin{equation}\label{eq:no-2-cycles-incident-to-k}
\text{$\widetilde{Q}$ does not have 2-cycles incident to $k$.}
\end{equation}
For each $(M,\bfv)\in\decrep(\jacobalg{Q}{S},\bfd,\bfv)$ we have Derksen-Weyman-Zelevinsky's \emph{$\alpha$-$\beta$-$\gamma$-triangle}
\begin{equation}\label{eq:DWZ-alphabetagamma-triangle}
\xymatrix{
 & \C^{d_k} \ar[dr]^{\beta_k(M)} & \\
 \underset{{\underset{a}{\rightarrow}k}}{\bigoplus}\C^{d_{s(a)}} \ar[ur]^{\alpha_k(M)} & & \underset{{\underset{b}{\leftarrow}k}}{\bigoplus}\C^{d_{t(b)}} \ar[ll]^{\gamma_k(M)},
 }
\end{equation}
where we think of $\alpha_k(M)$ (resp. $\beta_k(M)$) as the $d_k\times(\sum_{{\overset{a}{\rightarrow}k}}d_{s(a)})$ (resp. $(\sum_{{\overset{b}{\leftarrow}k}}d_{t(b)})\times d_k$) complex matrix obtained by putting the matrices $M_a$ with $t(a)=k$ (resp. $M_b$ with $s(b)=k$) next to each other as columns (resp. rows), and $\gamma_k(M)$ is given by putting together the matrices of the $\C$-linear maps $\partial_{ba}(S)(M):\C^{d_{t(b)}}\rightarrow\C^{d_{s(a)}}$, $v\mapsto \partial_{ba}(S)\cdot v$, for $(a,b)\in Q_1\times Q_1$ with $t(a)=k=s(b)$.

Let $\widetilde{\mu}_k(Q,S)=(\widetilde{\mu}_k(Q),\widetilde{\mu}_k(S))$ be the $k^{\operatorname{th}}$ \emph{$QP$-premutation} of $(Q,S)$, cf. \cite[Equations (5.3), (5.8) and (5.9)]{DWZ1}. Thus, $\widetilde{\mu}_k(Q)$ is the quiver obtained from $Q$ by applying only the first two steps of quiver mutation, and
$$
\widetilde{\mu}_k(S):=[S]+\sum_{\overset{a}{\to}k\overset{b}{\to}}b^*[ba]a^*\in\RA{\widetilde{\mu}_k(Q)}.
$$

\begin{thm}\label{thm:premut-of-reps-is-a-densely-defined-reg-map} Fix a decorated dimension vector $(\bfd,\bfv)\in\ZZ^n_{\geq 0}\times\ZZ^{n}_{\geq 0}$ and a vertex $k\in[n]$ satisfying~\eqref{eq:no-2-cycles-incident-to-k}. For each irreducible component $Z$ of the affine variety $\decrep(\jacobalg{Q}{S},\bfd,\bfv)$ there exist a decorated dimension vector $(\bfd',\bfv')\in\ZZ^{n}_{\geq 0}\times\ZZ^{n}_{\geq 0}$, a dense open subset $U\subseteq Z$ and a regular map $\widetilde{\nu}_U:U\rightarrow \decrep(\mathcal{P}(\widetilde{\mu}_k(Q,S)),\bfd',\bfv')$, such that for every $\mathcal{M}=(M,\bfv)\in U$ the decorated representation $\widetilde{\nu}_U(\mathcal{M})$ is isomorphic to the decorated representation $\widetilde{\mu}_k(\mathcal{M})$ of $\widetilde{\mu}_k(Q,S)$ defined by Derksen-Weyman-Zelevinsky in \cite[Equations (10.6)--(10.10)]{DWZ1}). The decorated dimension vector $(\bfd',\bfv')$ is uniquely determined by $k$ and $Z$, i.e., it is independent of the choice of $U$ and $\widetilde{\nu}_U$.
\end{thm}

\begin{proof}
Fix an irreducible component $Z$ of $\decrep(\jacobalg{Q}{S},\bfd,\bfv)$. By Lemma \ref{lemma:lower-semicontinuity-of-rank} and Proposition \ref{prop:every-irr-Z-has-subset-where-g-vects-CC-map-E-inv-and-codim-are-constant}, there exists a dense open subset $V\subseteq Z$ where each of the values
\begin{equation}\label{eq:preparing-mut-of-components-many-constant-numbers-1}
\begin{array}{lcl}
\dim_{\C}(\Image\alpha_k(M)), & \dim_{\C}(\Image\beta_k(M)), & \dim_{\C}(\Image\gamma_k(M)),\\ 
\dim_{\C}(\Image\beta_k(M)\alpha_k(M)), & E_{\jacobalg{Q}{S}}(M,\bfv), & CC_{(\widetilde{Q},\widetilde{S})}(M,\bfv),
\end{array}
\end{equation}
is constant when we let $\mathcal{M}=(M,\bfv)$ vary inside $V$. The numbers 
\begin{center}
\begin{tabular}{ll}
$\dim_{\C}(\Ker\alpha_k(M))$, & $\dim_{\C}(\Ker\beta_k(M))$,\\ $\dim_{\C}(\Ker\gamma_k(M))$, & $\dim_{\C}(\Ker\beta_k(M)\alpha_k(M))$
\end{tabular}
\end{center}
 are constant on $V$ too. We denote all these constants by
\begin{center}
\begin{tabular}{llll}
$\rank(\alpha_k(Z))$, & $\rank(\beta_k(Z)))$, & $\rank(\gamma_k(Z))$, & $\rank(\beta_k(Z)\alpha_k(Z))$,\\
 $\nullity(\alpha_k(Z))$, & $\nullity(\beta_k(Z))$, & $\nullity(\gamma_k(Z))$, & $\nullity(\beta_k(Z)\alpha_k(Z))$,
 \end{tabular}
 \end{center} 
respectively. 

Define a new decorated dimension vector $(\bfd',\bfv')$ by setting
\begin{eqnarray}\label{eq:mutated-decorated-dim-vector}
d_j'&:=&\begin{cases}
\nullity(\gamma_k(Z))-\rank(\beta_k(Z))+\nullity(\alpha_k(Z))+v_k & \text{if $j=k$;}\\
d_j & \text{if $j\neq k$.}
\end{cases} 
\\
\nonumber
v_j'&:=&\begin{cases}
\nullity(\beta_k(Z))-\nullity(\beta_k(Z)\alpha_k(Z))+\nullity(\alpha_k(Z)) & \text{if $j=k$;}\\
v_j & \text{if $j\neq k$.}
\end{cases}
\end{eqnarray}
For any $(M,\bfv)\in V$, all the decorated representations that are right-equivalent to either $\widetilde{\mu}_k(M,\bfv)$ or $\mu_k(M,\bfv)$ (as defined by Derksen-Weyman-Zelevinsky in \cite[Equations (10.6)--(10.10) and (10.22)]{DWZ1}) have $(\bfd',\bfv')$ as decorated dimension vector\footnote{This may fail to be true for all elements of $Z$.}, see \cite[Equation (10.7)]{DWZ1}. In the next paragraphs we will carefully trace the linear algebra that can be applied to choose specific representatives in the right-equivalence classes of $\widetilde{\mu}_k(M,\bfv)$ and $\mu_k(M,\bfv)$.

After fixing the dense open subset $V\subseteq Z$ on which the values \eqref{eq:preparing-mut-of-components-many-constant-numbers-1} are constant, fix the following regular functions as well:\\
\begin{center}
\begin{tabular}{|ll|}\hline
$\rho_1:Z\rightarrow\C^{(\sum_{{\overset{a}{\rightarrow}k}}d_{s(a)})\times(\sum_{{\overset{b}{\leftarrow}k}}d_{t(b)})}$, & $\rho_1(M,\bfv):=\gamma_k(M)$, \\
$\Delta_{I_1,J_1}:\C^{(\sum_{{\overset{a}{\rightarrow}k}}d_{s(a)})\times(\sum_{{\overset{b}{\leftarrow}k}}d_{t(b)})}\rightarrow\C$,& \text{a $\rank(\gamma_k(Z))\times\rank(\gamma_k(Z))$ minor,}\\
\hline
$\rho_2:Z\rightarrow\C^{\nullity(\gamma_k(Z))\times d_k}$, & $\rho_2(M,\bfv):= p_{\nullity(\gamma_k(Z)),(\sum_{{\overset{b}{\leftarrow}k}}d_{t(b)})}P_{J_1}\beta_k(M)$,\\
$\Delta_{I_2,J_2}:\C^{\nullity(\gamma_k(Z))\times d_k}\rightarrow\C$,&\text{a $\rank(\beta_k(Z))\times\rank(\beta_k(Z))$ minor,}\\
\hline
$\rho_3:Z\rightarrow \C^{d_k\times (\sum_{{\overset{a}{\rightarrow}k}}d_{s(a)})}$, & $\rho_3(M,\bfv):=\alpha_k(M)$\\ 
$\Delta_{I_3,J_3}:\C^{d_k\times (\sum_{{\overset{a}{\rightarrow}k}}d_{s(a)})}\rightarrow\C$, &\text{a $\rank(\alpha_k(Z))\times\rank(\alpha_k(Z))$ minor,}\\
\hline
$\rho_4:Z\rightarrow \C^{\nullity(\alpha_k(Z))\times(\sum_{{\overset{a}{\rightarrow}k}}d_{s(a)})}$ & $\rho_4(M,\bfv):= p_{\nullity(\alpha_k),(\sum_{{\overset{a}{\rightarrow}k}}d_{s(a)})}P_{J_3}\gamma_k(M)$,\\ 
$\Delta_{I_4,J_4}:\C^{\nullity(\alpha_k(Z))\times(\sum_{{\overset{a}{\rightarrow}k}}d_{s(a)})}\rightarrow\C$&\text{a $\rank(\gamma_k(Z))\times\rank(\gamma_k(Z))$ minor.}\\
\hline
\end{tabular}
\end{center}

Let
$U=U_{\Delta_{I_1,J_1},\Delta_{I_2,J_2},\Delta_{I_3,J_3},\Delta_{I_4,J_4}}:=\{(M,\bfv)\in V\suchthat
\prod_{\ell=1}^{4}\Delta_{I_{\ell},J_{\ell}}(\rho_{\ell}(M,\bfv))\neq 0\}$, which is open in $V$ (and in $Z$) by Lemma \ref{lemma:linear-algebra-regular-maps}.
Applying that same lemma to $(\rho_1,\Delta_{I_1,J_1})$, $(\rho_2,\Delta_{I_2,J_2})$, $(\rho_3,\Delta_{I_3,J_3})$ and $(\rho_4,\Delta_{I_4,J_4})$ yields regular functions 
\begin{center}
\begin{tabular}{ll}
$\varphi_{1}:U\rightarrow \C^{(\sum_{{\overset{b}{\leftarrow}k}}d_{t(b)})\times \nullity(\gamma_k(Z))}$, & 
$\psi_2:U\rightarrow \C^{(\nullity(\gamma_k(Z))-\rank(\beta_k(Z)))\times \nullity(\gamma_k(Z))}$,\\
$\varphi_3:U\rightarrow \C^{(\sum_{{\overset{a}{\rightarrow}k}}d_{s(a)})\times \nullity(\alpha_k(Z))}$, &
$\psi_4:U\rightarrow\C^{(\nullity(\alpha_k(Z))-\rank(\gamma_k(Z)))\times\nullity(\alpha_k(Z))}$,
\end{tabular}
\end{center}
such that for all $(M,\bfv)\in U$ the following conditions are simultaneously satisfied:

\begin{condition}
The columns of $\varphi_{1}(M,\bfv)$ form a basis of $\Ker\gamma_k(M)$, the bottom $\nullity(\gamma_k(Z))\times \nullity(\gamma_k(Z))$ submatrix of $P_{J_1}\varphi_{1}(M,\bfv)$ is the identity matrix, and the diagram
$$
\xymatrix{
\C^{\nullity(\gamma_k(Z))} \ar[rr]^{\varphi_{1}(M,\bfv)} \ar[d]_{T_{\rho_1,\Delta_{I_1,J_1}}(M,\bfv)} & &\C^{(\sum_{{\overset{b}{\leftarrow}k}}d_{t(b)})} \ar[d]^{p_{\nullity(\gamma_k(Z)),(\sum_{{\overset{b}{\leftarrow}k}}d_{t(b)})}P_{J_1}}\\
\Ker\gamma_k(M) \ar[urr]_{\iota_{\rho_1}(M,\bfv)} \ar[rr]_{T_{\rho_1,\Delta_{I_1,J_1}}(M)^{-1}} & & \C^{\nullity(\gamma_k(Z))}
}
$$
commutes. Thus, 
writing $T_1(M,\bfv):=T_{\rho_1,\Delta_{I_1,J_1}}(M,\bfv)$ and $p_1:=p_{\nullity(\gamma_k(Z)),(\sum_{{\overset{b}{\leftarrow}k}}d_{t(b)})}$, we have the commutative diagram
$$
\xymatrix{
\Ker(\gamma_k(M)) \ar@/^2pc/[rrrrrr]^{\myid} \ar[rrr]_{\iota_{\rho_1}(M,\bfv)} \ar[d]_{T_1(M,\bfv)^{-1}} &&& \C^{(\sum_{{\overset{b}{\leftarrow}k}}d_{t(b)})} \ar[rrr]_{\text{\tiny{$(T_1(M,\bfv))p_1P_{J_1}$}}} \ar[d]_{\myid} &&&  \Ker(\gamma_k(M)) \ar[d]^{T_1(M,\bfv)^{-1}} \\
\C^{\nullity(\gamma_k(Z))} \ar@/_2pc/[rrrrrr]_{\myid} \ar[rrr]^{\varphi_{1}(M,\bfv)} &&& \C^{(\sum_{{\overset{b}{\leftarrow}k}}d_{t(b)})}
\ar[rrr]^{p_1P_{J_1}} &&& \C^{\nullity(\gamma_k(Z))}
}
$$
\end{condition}

\begin{condition} By \cite[Lemma 10.6]{DWZ1}, $\gamma_k(M)\beta_k(M)=0$. Hence, each column of $\beta_k(M)$ can be written in a unique way as a $\C$-linear combination of the columns of $\varphi_{1}(M,\bfv)$. Thus, for every column $\mathbf{b}$ of $\beta_k(M)$, the equation $P_{J_1}\varphi_{1}(M,\bfv)\mathbf{x}=P_{J_1}\mathbf{b}$ has exactly one solution $\mathbf{x}_{\mathbf{b}}\in \C^{\nullity(\gamma_k(Z))}$.
Actually, since the bottom $\nullity(\gamma_k(Z))\times \nullity(\gamma_k(Z))$ submatrix of $P_{J_1}\varphi_{\Delta_1}(M,\bfv)$ is the identity matrix, we have $P_{J_1}\varphi_{\Delta_1}(M,\bfv)p_{\nullity(\gamma_k(Z)),(\sum_{{\overset{b}{\leftarrow}k}}d_{t(b)})}P_{J_1}\beta_k(M)=P_{J_1}\beta_k(M)$.  This implies, in particular, that the rank of $\rho_2(M,\bfv)=p_{\nullity(\gamma_k(Z)),(\sum_{{\overset{b}{\leftarrow}k}}d_{t(b)})}P_{J_1}\beta_k(M)$ is precisely $\rank(\beta_k(Z))$, which means that Lemma \ref{lemma:linear-algebra-regular-maps} can indeed be applied to $(\rho_2,\Delta_{I_2,J_2})$ in order to produce $\psi_2:U\rightarrow \C^{(\nullity(\gamma_k(Z))-\rank(\beta_k(Z)))\times \nullity(\gamma_k(Z))}$.

Now, the rows of $\psi_{2}(M,\bfv)$ form a basis of $\Ker(\rho_2(M,\bfv))^{\operatorname{t}})$, which means that the kernel of $\psi_{2}(M,\bfv)$ is precisely the image of $\rho_2(M,\bfv)$. Moreover, the rightmost $(\nullity(\gamma_k(Z))-\rank(\beta_k(Z)))\times(\nullity(\gamma_k(Z))-\rank(\beta_k(Z)))$ submatrix of $\psi_{2}(M,\bfv)P_{I_2}^{-1}$ is the identity matrix. Thus, we have a commutative diagram of $\C$-vector spaces
$$
\xymatrix{
0 \ar[r]   & \Image\beta_k(M) \ar[rr]^{\operatorname{incl}} \ar[d]_{T_1(M,\bfv)^{-1}|} & & \Ker\gamma_k(M) \ar[rr]^{\operatorname{proj}}   \ar[d]_{T_1(M,\bfv)^{-1}}  & & \Ker\gamma_k(M)/\Image\beta_k(M) \ar[r]  \ar[d]_{\overline{T_1(M,\bfv)^{-1}}} & 0  \\ 
0 \ar[r] & \Image(\rho_2(M,\bfv)) \ar[rr]_{\operatorname{incl}} & & \C^{\nullity(\gamma_k(Z))} \ar[rr]_{\psi_{2}(M,\bfv)}  & & \C^{\nullity(\gamma_k(Z))-\rank(\beta_k(Z))} \ar[r] & 0
}
$$
whose rows are exact, and whose vertical arrows are isomorphisms.
\end{condition}

\begin{condition} The columns of $\varphi_{3}(M,\bfv)$ form a basis of $\Ker\alpha_k(M)$, the bottom $\nullity(\alpha_k(Z))\times \nullity(\alpha_k(Z))$ submatrix of $P_{J_3}\cdot \varphi_{3}(M,\bfv)$ is the identity matrix, and the diagram
$$
\xymatrix{
\C^{\nullity(\alpha_k(Z))} \ar[rrr]^{\varphi_{3}(M,\bfv)} \ar[d]_{T_{\rho_3,\Delta_{I_3,J_3}}(M,\bfv)} & & &\C^{(\sum_{{\overset{a}{\rightarrow}k}}d_{s(a)})} \ar[d]^{p_{\nullity(\alpha_k(Z)),(\sum_{{\overset{a}{\rightarrow}k}}d_{s(a)})}P_{J_3}} \\
\Ker\alpha_k(M) \ar[urrr]_{\iota_{\rho_3}(M,\bfv)=\operatorname{incl}} \ar[rrr]_{T_{\rho_3,\Delta_{I_3,J_3}}(M,\bfv)^{-1}}  & & & \C^{\nullity(\alpha_k(Z))}
}
$$
commutes. We will write $T_3(M,\bfv):=T_{\rho_3,\Delta_{I_3,J_3}}(M,\bfv)$.
\end{condition}

\begin{condition} By \cite[Lemma 10.6]{DWZ1}, $\alpha_k(M)\gamma_k(M)=0$. Hence, each column of $\gamma_k(M)$ can be written in a unique way as a $\C$-linear combination of the columns of $\varphi_{3}(M,\bfv)$. Thus, for every column $\mathbf{c}$ of $\gamma_k(M)$, the equation $P_{J_3}\varphi_{3}(M,\bfv)\mathbf{x}=P_{J_3}\mathbf{c}$ has exactly one solution $\mathbf{x}_{\mathbf{c}}\in \C^{\nullity(\alpha_k(Z))}$.
Actually, since the bottom $\nullity(\alpha_k(Z))\times \nullity(\alpha_k(Z))$ submatrix of $P_{J_3}\varphi_{\Delta_3}(M,\bfv)$ is the identity matrix, we have $P_{J_3}\varphi_{3}(M,\bfv)p_{\nullity(\alpha_k(Z)),(\sum_{{\overset{a}{\rightarrow}k}}d_{s(a)})}P_{J_3}\gamma_k(M)=P_{J_3}\gamma_k(M)$, that is, $\varphi_{3}(M,\bfv)p_{\nullity(\alpha_k(Z)),(\sum_{{\overset{a}{\rightarrow}k}}d_{s(a)})}P_{J_3}\gamma_k(M)=\gamma_k(M)$. This implies, in particular, that the rank of $\rho_4(M,\bfv)=p_{\nullity(\alpha_k),(\sum_{{\overset{a}{\rightarrow}k}}d_{s(a)})}P_{J_3}\gamma_k(M)$ is precisely $\rank(\gamma_k(Z))$, which means that Lemma \ref{lemma:linear-algebra-regular-maps} can indeed be applied to $(\rho_4,\Delta_{I_4,J_4})$ in order to produce $\psi_4:U\rightarrow\C^{(\nullity(\alpha_k(Z))-\rank(\gamma_k(Z)))\times\nullity(\alpha_k(Z))}$.

Now, the rows of $\psi_{4}(M,\bfv)$
form a basis of $\Ker((p_{\nullity(\alpha_k),(\sum_{{\overset{a}{\rightarrow}k}}d_{s(a)})}P_{J_3}\gamma_k(M))^{\operatorname{t}})$, which means that the kernel of $\psi_{4}(M,\bfv)$ is precisely the image of $p_{\nullity(\alpha_k),(\sum_{{\overset{a}{\rightarrow}k}}d_{s(a)})}P_{J_3}\gamma_k(M)$. Thus, we have a commutative diagram of $\C$-vector spaces
$$
\xymatrix{
0 \ar[r]   & \Image\gamma_k(M) \ar[rr]^{\operatorname{incl}} \ar[d]_{T_{3}(M,\bfv)^{-1}|} & & \Ker\alpha_k(M) \ar[rr]^{\operatorname{proj}}   \ar[d]_{T_{3}(M,\bfv)^{-1}}  & & \Ker\alpha_k(M)/\Image\gamma_k(M) \ar[r]  \ar[d]_{\overline{T_{3}(M,\bfv)^{-1}}} & 0  \\ 
0 \ar[r] & \Image(\rho_4(M,\bfv)) \ar[rr]_{\operatorname{incl}} & & \C^{\nullity(\alpha_k(Z))} \ar[rr]_{\psi_{4}(M,\bfv)}  & & \C^{\nullity(\alpha_k(Z))-\rank(\gamma_k(Z))} \ar[r] & 0
}
$$
whose rows are exact, and whose vertical arrows are isomorphisms. Furthermore, since the rightmost $(\nullity(\alpha_k(Z))-\rank(\gamma_k(Z)))\times(\nullity(\alpha_k(Z))-\rank(\gamma_k(Z)))$ submatrix of
$\psi_{4}(M,\bfv)\cdot P_{I_4}^{-1}$ is the identity matrix, the product $\psi_{4}(M,\bfv)\cdot P_{I_4}^{-1}\cdot (p_{\nullity(\alpha_k(Z))-\rank(\gamma_k(Z)),\nullity(\alpha_k(Z))})^{\operatorname{t}}$ is equal to the identity matrix $\myid_{(\nullity(\alpha_k(Z))-\rank(\gamma_k(Z)))\times (\nullity(\alpha_k(Z))-\rank(\gamma_k(Z)))}$, so we also have the following commutative diagram
$$
\xymatrix{
\Ker\alpha_k(M)/\Image\gamma_k(M) \ar@/^2pc/[rrrrrr]_{\myid} \ar[rrr]^{\sigma} \ar[d]_{\overline{T_{3}(M,\bfv)^{-1}}} &&& \Ker\alpha_k(M) \ar[rrr]^{\operatorname{proj}} \ar[d]_{T_{3}(M,\bfv)^{-1}} &&&  \Ker\alpha_k(M)/\Image\gamma_k(M) \ar[d]_{\overline{T_{3}(M,\bfv)^{-1}}} \\
\C^{\nullity(\alpha_k(Z))-\rank(\gamma_k(Z))} \ar@/_2pc/[rrrrrr]^{\myid} \ar[rrr]_{P_{I_4}^{-1}p_4^{\operatorname{t}}} &&& \C^{\nullity(\alpha_k(Z))}
\ar[rrr]_{\psi_{4}(M,\bfv)} &&& \C^{\nullity(\alpha_k(Z))-\rank(\gamma_k(Z))}
}
$$
where $p_4:=p_{\nullity(\alpha_k(Z))-\rank(\gamma_k(Z)),\nullity(\alpha_k(Z))}$ and $\sigma:=T_{3}(M,\bfv)P_{I_4}^{-1}p_4^{\operatorname{t}}\overline{T_{3}(M,\bfv)^{-1}}$.
\end{condition}

Write $\psi_4(M,\bfv)P_{I_4}^{-1}=\left[\begin{array}{cc}A(M,\bfv) & \myid\end{array}\right]$, with $A(M,\bfv)\in\C^{(\nullity(\alpha_k(Z))-\rank(\gamma_k(Z)))\times\rank(\gamma_k(Z))}$, for each $(M,\bfv)\in U$. Since $\psi_4:U\rightarrow\C^{(\nullity(\alpha_k(Z))-\rank(\gamma_k(Z)))\times\nullity(\alpha_k(Z))}$ is regular, so is the function $\psi_5:U\rightarrow C^{\nullity(\alpha_k(Z))\times\rank(\gamma_k(Z))}$ given by
$$
\psi_5(M,\bfv):= \left[\begin{array}{c}
\myid_{\rank(\gamma_k(Z))\times\rank(\gamma_k(Z))}\\
-A(M,\bfv)
\end{array}\right].
$$

All the previous discussion implies that for every $(M,\bfv)\in U$ we have the commutative diagrams
\begin{equation*}
\xymatrix@+3ex{
\text{{$\Ker{\gamma_k(M)}/\Image\beta_k(M)\oplus \Image\gamma_k(M) \oplus \Ker\alpha_k(M)/\Image\gamma_k(M)\oplus \C^{v_k}$}}
\ar[dd]^{\text{{$
\overline{\beta}_k(M,\bfv):=\left[\begin{matrix} 0 & \operatorname{incl} & \iota_{\rho_3}(M,\bfv)\circ\sigma & 0\end{matrix}\right]$}}
} 
\ar@/_8ex/[dddd]_
{\text{\tiny$\left[\begin{matrix}
\overline{T_{1}(M,\bfv)^{-1}} & 0 & 0 & 0 \\
0 & (p_5P_{I_4})|\circ(T_3(M,\bfv)^{-1}|) & 0 & 0 \\
0 & 0  & \overline{T_{3}(M,\bfv)^{-1}}  & 0 \\
0  & 0  & 0 & \myid
\end{matrix}
\right]$}}
\\
\\
\underset{{\underset{a}{\rightarrow}k}}{\bigoplus}\C^{d_{s(a)}}  
\\
\\
\text{{$\C^{\nullity(\gamma_k(Z))-\rank(\beta_k(Z))}\oplus \C^{\rank(\gamma_k(Z))}\oplus \C^{\nullity(\alpha_k(Z))-\rank(\gamma_k(Z))}\oplus \C^{v_k}$}} 
\ar[uu]_{\text{{$
\left[\begin{matrix} 
0 & \varphi_3(M,\bfv)P_{I_4}^{-1}\psi_5(M,\bfv) & \varphi_3(M,\bfv)P_{I_4}^{-1}p_4^{\operatorname{t}} & 0 
\end{matrix}\right]
$}}
}
}
\end{equation*}
and
\begin{equation*}
\xymatrix@+3ex{
\text{{$\Ker{\gamma_k(M)}/\Image\beta_k(M)\oplus \Image\gamma_k(M) \oplus \Ker\alpha_k(M)/\Image\gamma_k(M)\oplus \C^{v_k}$}}
\ar@/^8ex/[dddd]^{\text{\tiny$\left[\begin{matrix}
\overline{T_{1}(M,\bfv)^{-1}} & 0 & 0 & 0 \\
0 & (p_5P_{I_4})|\circ(T_3(M,\bfv)^{-1}|) & 0 & 0 \\
0 & 0  & \overline{T_{3}(M,\bfv)^{-1}}  & 0 \\
0  & 0  & 0 & \myid
\end{matrix}
\right]$}}
\\
\\
\underset{{\underset{b}{\leftarrow}k}}{\bigoplus}\C^{d_{t(b)}} 
\ar[uu]^{\text{{$\overline{\alpha}_k(M,\bfv):=
\left[\begin{matrix}
-\operatorname{proj}\circ T_1(M,\bfv)p_1P_{J_1}\\
-\gamma_k(M,\bfv)\\
0\\
0\end{matrix}\right]
$}}
} 
\ar[dd]_{
\text{{$
\left[
\begin{matrix}
-\psi_2(M,\bfv)p_1P_{J_1}\\
-p_5P_{I_4}p_3P_{J_3}\gamma_k(M,\bfv)\\
0\\
0\end{matrix}\right]
$}}
} 
\\
\\
\text{{$\C^{\nullity(\gamma_k(Z))-\rank(\beta_k(Z))}\oplus \C^{\rank(\gamma_k(Z))}\oplus \C^{\nullity(\alpha_k(Z))-\rank(\gamma_k(Z))}\oplus \C^{v_k}$}} 
}
\end{equation*}
where 
$$
p_5:=
\left[\begin{matrix}
\myid_{\rank(\gamma_k(Z))\times\rank(\gamma_k(Z))} & 0_{\rank(\gamma_k)\times(\nullity(\alpha_k(Z)) - \rank(\gamma_k(Z))}\end{matrix}\right]
\in \C^{\rank(\gamma_k(Z))\times\nullity(\alpha_k(Z))}.
$$ 

Derksen-Weyman-Zelevinsky show in \cite[Proposition 10.7]{DWZ1} that the representation $N=N(M,\bfv)$ of $\widetilde{\mu}_k(Q)$ determined by the data $\overline{\alpha}_k(M,\bfv)$, $\overline{\beta}_k(M,\bfv)$, together with the original action on $M$ of the arrows of $Q$ not incident to $k$, is actually a module over the Jacobian algebra $\jacobalg{\widetilde{\mu}_k(Q)}{\widetilde{\mu}_k(S)}$. Let us write $\widetilde{\mu}_k(M,\bfv):=(N(M,\bfv),\bfv')\in\decrep(\widetilde{\mu}_k(Q),\widetilde{\mu}_k(S),\bfd',\bfv')$ for $(M,\bfv)\in U$. 
The fact that the two diagrams above
commute tells us that we have computed explicitly the matrices that correspond to the $\C$-linear maps conforming $N(M,\bfv)$ for $(M,\bfv)\in U=U_{\Delta_{I_1,J_1},\Delta_{I_2,J_2},\Delta_{I_3,J_3},\Delta_{I_4,J_4}}$.


The existence of $(\bfd',\bfv')$, $U$ and $\widetilde{\nu}_U$ asserted in the statement of Theorem \ref{thm:premut-of-reps-is-a-densely-defined-reg-map} follows. The uniquenes of $(\bfd',\bfv')$  is obvious.
\end{proof}

\begin{remark} Notice that in the proof of Theorem \ref{thm:premut-of-reps-is-a-densely-defined-reg-map},
\begin{enumerate}
\item the open subset $V$ of $Z$ is fully covered by the open sets $U_{\Delta_{I_1,J_1},\Delta_{I_2,J_2},\Delta_{I_3,J_3},\Delta_{I_4,J_4}}$ as we let the minors $\Delta_{I_1,J_1}$, $\Delta_{I_2,J_2}$, $\Delta_{I_3,J_3}$ and $\Delta_{I_4,J_4}$ vary;
\item for $U$ and $U'$ arising from different choices of minors, the regular functions $\widetilde{\nu}_U$ and $\widetilde{\nu}_{U'}$ may as well differ on the intersection $U\cap U'$. However, by \cite[Proposition 10.9]{DWZ1} we still have $\widetilde{\nu}_U(\mathcal{M})\cong\widetilde{\nu}_{U'}(\mathcal{M})$ as decorated representations for $\mathcal{M}\in U\cap U'$.
\end{enumerate}
\end{remark}

Let $R$ and $\widetilde{A}$ be vertex and arrow spans of $\widetilde{\mu}_k(Q)$. Thus, $R$ (resp. $\widetilde{A}$) is the $\C$-vector subspace of $\RA{\widetilde{\mu}_k(Q)}$ spanned by $\{e_j\suchthat j$ is a vertex of $\widetilde{\mu}_k(Q)\}$ (resp. by the arrow set of $\widetilde{\mu}_k(Q)$). Fix an $R$-$R$-bimodule automorphism $\widetilde{A}\rightarrow\widetilde{A}$ such that the induced $\C$-algebra automorphism $\Psi:\RA{\widetilde{\mu}_k(Q)}\rightarrow\RA{\widetilde{\mu}_k(Q)}$ (called \emph{change of arrows} in Derksen-Weyman-Zelevinsky's nomenclature) satisfies
$$
\Psi(\widetilde{\mu}_k(S))=\sum_{\ell=1}^{t}\delta_\ell \varepsilon_\ell+W
$$
for some $t\in\Z_{\geq 0}$ (determined by $\widetilde{\mu}_k(S)$), $2t$ distinct arrows $\delta_1,\varepsilon_1,\ldots,\delta_t,\varepsilon_t$ of $\widetilde{\mu}_k(Q)$ such that $\delta_1\varepsilon_1,\ldots,\delta_t\varepsilon_t$ are 2-cycles, and a potential $W\in\RA{\widetilde{\mu}_k(Q)}$ involving only cycles of length at least~3, cf. \cite[Proposition 4.4 and Equation (4.6)]{DWZ1}. Let $\Phi:\RA{\widetilde{\mu}_k(Q)}\rightarrow\RA{\widetilde{\mu}_k(Q)}$ be the $\C$-algebra automorphism \cite[Equation (4.7)]{DWZ1} that, starting with $\Psi(\widetilde{\mu}_k(S))$, Derksen-Weyman-Zelevinsky produce in their proof of \cite[Lemmas 4.7 and 4.8]{DWZ1}. Thus, $\Phi$ is defined through a limit process fully determined by $\Psi(\widetilde{\mu}_k(S))$. This automorphism $\Phi$ acts as the identity on every arrow of $\widetilde{\mu}_k(Q)$ not belonging to the set $\{\delta_1,\varepsilon_1,\ldots,\delta_t,\varepsilon_t\}$, and satisfies
$$
\Phi(\Psi(\widetilde{\mu}_k(S)))=\sum_{\ell=1}^{t}\delta_\ell \varepsilon_\ell+\mu_k(S)
$$
for a potential $\mu_k(S)\in\RA{\widetilde{\mu}_k(Q)}$ involving only cycles of length at least 3, and with the additional property that none of the arrows $\delta_1,\varepsilon_1,\ldots,\delta_t,\varepsilon_t$ appears in any term of $\mu_k(S)$. Letting $Q'$ be the quiver obtained from $\widetilde{\mu}_k(Q)$ by deleting the arrows $\delta_1,\varepsilon_1,\ldots,\delta_t,\varepsilon_t$, Derksen-Weyman-Zelevinsky call the QP $\mu_k(Q,S):=(Q',\mu_k(S))$ the \emph{QP-mutation} of $(Q,S)$ with respect to $k$, cf. \cite[Definition 5.5]{DWZ1}. However, noting that it is perfectly possible that $Q'$ have 2-cycles, they avoid referring to $Q'$ as $\mu_k(Q)$, cf. \cite[the paragraph before Definition 7.2]{DWZ1}.

\begin{remark} The algebra automorphism $\Phi\circ\Psi$ is by no means the only one that puts $\widetilde{\mu}_k(S)$ in standard form. This is one of the reasons why on the next result we cannot claim any sort of uniqueness of $\nu_U$. On the other hand, taking the specific $\Psi$ and $\Phi$ defined above allows us to produce a well-defined map $\nu_U$. That is, lack of uniqueness does not mean impossibility to produce at least one well-defined map.
\end{remark}

\begin{thm}\label{thm:premut-&-mut-of-reps-are-locally-regular} Fix a decorated dimension vector $(\bfd,\bfv)\in\ZZ^n_{\geq 0}\times\ZZ^n_{\geq 0}$ and a vertex $k\in[n]$ satisfying \eqref{eq:no-2-cycles-incident-to-k}. For each irreducible component $Z$ of the affine variety $\decrep(\jacobalg{Q}{S},\bfd,\bfv)$ there exist a decorated dimension vector $(\bfd',\bfv')\in\ZZ^{n}_{\geq 0}\times\ZZ^{n}_{\geq 0}$, a dense open subset $U\subseteq Z$ and a regular map $\nu_U:U\rightarrow \decrep(\mathcal{P}(\mu_k(Q,S)),\bfd',\bfv')$, such that for every $\mathcal{M}=(M,\bfv)\in U$ the decorated representation $\nu_U(\mathcal{M})$ is isomorphic to the decorated representation $\mu_k(\mathcal{M})$ of $\mu_k(Q,S)=(Q',\mu_k(S))$ defined by Derksen-Weyman-Zelevinsky in \cite[Equations (10.6)--(10.10) and (10.22)]{DWZ1}. The decorated dimension vector $(\bfd',\bfv')$ is uniquely determined by $k$ and $Z$, i.e., it is independent of the choice of $U$ and $\nu_U$.
\end{thm}

\begin{proof}
Any decorated representation $(N,\bfv')\in \decrep(\widetilde{\mu}_k(Q),\widetilde{\mu}_k(S),\bfd',\bfv')$ gives rise to a decorated representation $(N_{\operatorname{red}},\bfv')$ of $\mu_k(Q,S)=(Q',\mu_k(S))$ by setting $N_{\operatorname{red}}:=N$ as $\C$-vector space and defining the action of each arrow $\alpha$ of $Q'$ on $N_{\operatorname{red}}$ to be the $\C$-linear map given by the action of $\Psi^{-1}(\Phi^{-1}(\alpha))$ on $N$. Derksen-Weyman-Zelevinsky call $(N_{\operatorname{red}},\bfv')$ the \emph{reduced part} of $(N,\bfv')$, cf. \cite[Definition 10.4]{DWZ1}. Since $\Psi^{-1}(\Phi^{-1}(\alpha))=\Psi^{-1}(\alpha)$ for every arrow $\alpha$ of $Q'$, it is obvious that, for our fixed choice of $\Psi$ and $\Phi$, the function $\decrep(\widetilde{\mu}_k(Q),\widetilde{\mu}_k(S),\bfd',\bfv')\rightarrow\decrep(Q',\mu_k(S),\bfd',\bfv')$ given by $(N,\bfv')\mapsto(N_{\operatorname{red}},\bfv')$ is regular. The result is now a consequence of Theorem \ref{thm:premut-of-reps-is-a-densely-defined-reg-map}.
\end{proof}

\begin{lemma}\label{lemma:nu_U-sends-non-isomorphic-to-non-isomorphic} In the situation of Theorem \ref{thm:premut-&-mut-of-reps-are-locally-regular}, if $(M_1,\bfv),(M_2,\bfv)\in U$ are non-isomorphic, then $\nu_U(M_1,\bfv)$ and $\nu_U(M_2,\bfv)$ are non-isomorphic.
\end{lemma}

\begin{proof} This is a direct consequence of \cite[Lemma 3.1]{Velasco} and \cite[Proposition 10.15 and the paragraph that precedes Corollary 10.14]{DWZ1}. 
\end{proof}

\begin{thm}\label{thm:mutation-of-strongly-reduced-components} 
Suppose that $k\in[n]$ is a vertex satisfying \eqref{eq:no-2-cycles-incident-to-k}.
\begin{enumerate}
\item 
Take $Z\in\decirr(\mathcal{P}(Q,S))$, and let $(\bfd',\bfv')$, $U$ and $\nu_U:U\rightarrow \decrep(\mathcal{P}(Q,S),\bfd',\bfv')$ be a dimension vector, a dense open subset of $Z$ and a regular map, obtained by applying Theorem \ref{thm:premut-&-mut-of-reps-are-locally-regular} to $Z$. Let
$$
\nu_U'\df \GL_{\bfd'} \times U \to \decrep(\mathcal{P}',\bfd',\bfv')
$$
be the regular function 
defined by $(g,u) \mapsto g \cdot \nu_U(u)$.
If $Z$ is generically $\tau^-$-reduced, then
$$
\mu_k(Z) := \overline{\Image(\nu_U')}
$$
is
a generically $\tau^-$-reduced component of
$\decrep(\mathcal{P}',\bfd',\bfv')$.

\item
The map
$$
\decirrsr(\mathcal{P}) \to \decirrsr(\mathcal{P}') 
$$
defined by $Z \mapsto \mu_k(Z)$
is bijective, and we have $\mu_k(\mu_k(Z)) = Z$ for all 
$Z \in \decirrsr(\mathcal{P}(Q,S))$.
\end{enumerate}
\end{thm}

\begin{proof}
Suppose that $Z$ is a generically $\tau^-$-reduced irreducible component of $\decrep(Q,S,\bfd,\bfv)$, take a regular map $\nu_U:U\rightarrow \decrep(Q',\mu_k(S),\bfd',\bfv')$ as in Theorem \ref{thm:premut-&-mut-of-reps-are-locally-regular}, and write $\nu_U(M,\bfv)=(N(M,\bfv),\bfv')$ with $N(M,\bfv)\in \rep(Q',\mu_k(S),\bfd')$ for $(M,\bfv)\in U$. By upper semicontinuity, $U$ has a dense open subset $U'$ such that each of the numbers
\begin{eqnarray*}
e:=\dim_{\C}(\End_{\jacobalg{Q}{S}}(M)) & \text{and} & e':=\dim_{\C}(\End_{\jacobalg{Q'}{\mu_k(S)}}(N(M,\bfv)))
\end{eqnarray*}
is constant when we let $(M,\bfv)$ vary inside $U'$. Set
\begin{eqnarray*}
W&:=&\{((M,\bfv),(N,\bfv'))\in U'\times\decrep(Q',\mu_k(S),\bfd',\bfv')\suchthat \nu_U(M,\bfv) \ \text{and} \ (N,\bfv')\\ 
&& \phantom{ \{ } \text{are isomorphic as decorated representations of} \ (Q',\mu_k(S))\}.
\end{eqnarray*}
Being the image of the morphism $U'\times \GL_{\bfd'}(\C)\rightarrow U'\times\decrep(Q',\mu_k(S),\bfd',\bfv')$ given by the rule $((M,\bfv),g)\mapsto ((M,\bfv),g\cdot \nu_U(M,\bfv))$, the set $W$ is constructible in $U'\times\decrep(Q',\mu_k(S),\bfd',\bfv')$ and irreducible. Let
$$
\xymatrix{ & Z\times\decrep(Q',\mu_k(S),\bfd',\bfv') \ar[dl]_{\pi_1} \ar[dr]^{\pi_2} &  \\
Z' & & \decrep(Q',\mu_k(S),\bfd',\bfv')}
$$
be the projections, and set
\begin{eqnarray*}
p_1:=\pi_1|_{W}:W\rightarrow U' &\text{and}&  p_2:=\pi_2|_{W}:W\rightarrow  \decrep(Q',\mu_k(S),\bfd',\bfv').
\end{eqnarray*}

Set $Y:=p_2(W)\subseteq \decrep(Q',\mu_k(S),\bfd',\bfv')$. Being the image of an irreducible set under a continuous function, $Y$ is irreducible. Thus, $Y$ is contained in an irreducible component $Z'$ of $\decrep(Q',\mu_k(S),\bfd',\bfv')$. We claim that $Z'$ is a generically $\tau^-$ reduced irreducible component of $\decrep(Q',\mu_k(S),\bfd',\bfv')$ and that  $Y$ is dense in $Z'$.

For each $(M,\bfv)\in U'$ the fiber $p_1^{-1}(M,\bfv)\subseteq W$ is canonically isomorphic to the $\GL_{\bfd'}(\C)$-orbit of $\nu_U(M,\bfv)$ inside $\decrep(Q',\mu_k(S),\bfd',\bfv')$. Hence, $p_1^{-1}(M,\bfv)$ is irreducible and 
\begin{equation}\label{eq:mut-of-components-dim-fiber-of-p1}
\dim(p_1^{-1}(M,\bfv))=\dim(\GL_{\bfd'}(\C))-e'.
\end{equation}
Thus, all the fibers of $p_1$ are irreducible of the same dimension $\dim(\GL_{\bfd'}(\C))-e'$.
On the other hand, since $\nu_U:U\rightarrow \decrep(Q',\mu_k(S),\bfd',\bfv')$ sends non-isomorphic decorated representations to non-isomorphic ones (Lemma \ref{lemma:nu_U-sends-non-isomorphic-to-non-isomorphic}), for $((M,\bfv),(N,\bfv'))\in W$ we have
\begin{equation}\label{eq:mut-of-components-fiber-of-p2}
p_2^{-1}(N,\bfv')=(\GL_{\bfd}(\C)\cdot(M,\bfv),(N,\bfv')),
\end{equation}
irreducible of dimension $\dim(\GL_{d}(\C))-e$.

Let $\overline{W}$ be the closure of $W$ in $Z\times\decrep(Q',\mu_k(S),\bfd',\bfv')$. We have a commutative diagram
$$
\xymatrix{
W \ar[d]_{p_1} \ar[r] & \overline{W} \ar[d]^{\pi_1|_{\overline{W}}}\\
U' \ar[r] & Z }
$$
whose horizontal arrows are inclusions. Noticing that $\pi_1|_{\overline{W}}:\overline{W}\rightarrow Z$ is a dominant morphism of affine varieties, from \eqref{eq:mut-of-components-dim-fiber-of-p1} and, say, \cite[Proposition 6.6]{Bump-AG-book}, we deduce that 
\begin{eqnarray}\label{eq:mutation-of-components-equality-first-domination}
\dim(Z)+\dim(\GL_{\bfd'}(\C))-e' & = &\dim(\overline{W}).
\end{eqnarray}
Furthermore, we have $\pi_2(\overline{W})\subseteq \overline{Y} \subseteq Z'$, and $\overline{W}$ certainly dominates $\overline{Y}$ through $\pi_2$. Therefore,
\begin{eqnarray}\label{eq:mutation-of-components-equality-second-domination}
\dim(\overline{W}) & = &\dim(\overline{Y})+\dim(\GL_{\bfd}(\C))-e
\end{eqnarray}
by \eqref{eq:mut-of-components-fiber-of-p2} and \cite[Proposition 6.6]{Bump-AG-book}. Combining \eqref{eq:mutation-of-components-equality-first-domination} and \eqref{eq:mutation-of-components-equality-second-domination}, we see that
\begin{eqnarray}\label{eq:generic-codim-preserved}
\dim(Z)-(\dim(\GL_{\bfd}(\C))-e) & = &  \dim(\overline{Y})-(\dim(\GL_{\bfd'}(\C))-e').
\end{eqnarray}

If we apply Voigt's Lemma and Auslander-Reiten's formulas inside $\overline{Y}$, we obtain
\begin{eqnarray}
\nonumber
c_{\mathcal{P}(\mu_k(Q,S))}(\overline{Y}) &:=& \min\{\dim(\overline{Y})-\dim(\GL_{\bfd'}(\C)\cdot(N,\bfv'))\suchthat(N,\bfv')\in\overline{Y}\}\\
\label{eq:c-leq-e-leq-E-for-closed-set}
& \leq & \min\{\dim_{\C}(\Ext^1_{\mathcal{P}(\mu_k(Q,S))}(N,N))\suchthat (N,\bfv')\in\overline{Y}\}\\
\nonumber
& \leq & \min\{E_{\mathcal{P}(\mu_k(Q,S))}(N,\bfv')\suchthat(N,\bfv')\in\overline{Y}\}.
\end{eqnarray}
By the upper semicontinuity of the $E$-invariant (Theorem \ref{thm:E-invariant-is-upper-semicontinuous}, see also Remark \ref{rem:E(Z)<=E(M)-for-M-in-Z}),
\begin{equation}\label{eq:E-does-not-change-under-topological-closure}
\min\{E_{\mathcal{P}(\mu_k(Q,S))}(N,\bfv')\suchthat(N,\bfv')\in\overline{Y}\}=\min\{E_{\mathcal{P}(\mu_k(Q,S))}(N,\bfv')\suchthat(N,\bfv')\in Y\},
\end{equation}
whereas by the upper semicontinuity of $\dim_{\C}(\Hom_{\mathcal{P}(\mu_k(Q,S))}(-,-))$, the way the number $e'$ was chosen, the upper semicontinuity of $\Ext^1_{\mathcal{P}(\mu_k(Q,S))}(-,-))$ and Remark \ref{rem:E(Z)<=E(M)-for-M-in-Z},
\begin{eqnarray}\label{eq:c-and-ext1(-,-)-do-not-change-under-topological-closure}
c_{\mathcal{P}(\mu_k(Q,S))}(\overline{Y})&=&\dim(\overline{Y})-(\dim(\GL_{\bfd'}(\C))-e')\\
\nonumber
\min\{\dim_{\C}(\Ext^1_{\mathcal{P}(\mu_k(Q,S))}(N,N))\suchthat (N,\bfv')\in\overline{Y}\} &=& \min\{\dim_{\C}(\Ext^1_{\mathcal{P}(\mu_k(Q,S))}(N,N))\suchthat (N,\bfv')\in Y\}
\end{eqnarray}

Combining \eqref{eq:generic-codim-preserved}, \eqref{eq:c-leq-e-leq-E-for-closed-set}, \eqref{eq:E-does-not-change-under-topological-closure} and \eqref{eq:c-and-ext1(-,-)-do-not-change-under-topological-closure}, and recalling that $Z$ is generically $\tau$-reduced, that the $E$-invariant is invariant under mutations of decorated representations by \cite[Theorem 7.1]{DWZ2}, that $U'$ is open and dense in $Z$,  and that the $E$-invariant is upper semicontinuous in $Z$, we deduce
\begin{eqnarray*}
E_{\mathcal{P}(Q,S)}(Z) &=& \dim(Z)-(\dim(\GL_{\bfd}(\C))-e)\\
& = &  \dim(\overline{Y})-(\dim(\GL_{\bfd'}(\C))-e')\\
&=& c_{\mathcal{P}(\mu_k(Q,S))}(\overline{Y})\\
& \leq & \min\{\dim_{\C}(\Ext^1_{\mathcal{P}(\mu_k(Q,S))}(N,N))\suchthat(N,\bfv')\in Y\}\\
&\leq & \min\{E_{\mathcal{P}(\mu_k(Q,S))}(N,\bfv')\suchthat(N,\bfv')\in Y\}\\
&=& \min\{E_{\mathcal{P}(Q,S)}(M,\bfv)\suchthat(M,\bfv)\in U'\}\\
&=& E_{\mathcal{P}(Q,S)}(Z),
\end{eqnarray*}
which means that we have equality all along.

Take a decorated representation $(N,\bfv')\in Y$ such that $\dim(\overline{Y})-(\dim(\GL_{\bfd'}(\C))-e')=\dim_{\C}(\Ext^1_{\mathcal{P}(\mu_k(Q,S))}(N,N))$. Combining this equality with an application of Voigt's Lemma to $(N,\bfv')$ inside $Z'$ rather than inside $\overline{Y}$, we obtain $\dim(Z')\leq \dim(\overline{Y})$. Hence, $Z'=\overline{Y}$.
This proves that $Y$ is dense in $Z'$ and that $Z'$ is generically $\tau^-$-reduced.
\end{proof}

Putting Theorems \ref{thm:premut-&-mut-of-reps-are-locally-regular} and \ref{thm:mutation-of-strongly-reduced-components} together we obtain Theorem \ref{thm:premut-&-mut-of-reps-are-locally-regular-intro}. Corollaries \ref{cor:1.8b} and \ref{cor:generic-stays-generic-after-DWZ-mutation-intro}  are easy consequences. In the sake of clarity we provide some details. 

Let $\widetilde{\Lambda}$ and $\Lambda$ be $\C$-algebras as in Subsection \ref{subsec:assumptions}. To define the Caldero-Chapoton functions of the decorated representations of $\Lambda$ (supported in $[n]$), one first fixes the underlying field $\calF$ of rational functions in $n+m$ indeterminates with coefficients in $\Q$, and a tuple $\mathbf{x}=(x_1,\ldots,x_n,x_{n+1},\ldots,x_{n+m})$ of elements of $\calF$ algebraically independent over $\Q$, then defines $CC_{\widetilde{\Lambda}}(\mathcal{M})$ as a Laurent polynomial in $\mathbf{x}$ for all $\mathcal{M}\in\decrep(\Lambda)$. A different choice of initial algebraically independent tuple is of course allowed, and yields an isomorphic realization of the Caldero-Chapoton algebra $\calCC_{\widetilde{\Lambda}}(\Lambda)$.

Suppose now that $\widetilde{\Lambda}=\mathcal{P}(\widetilde{Q},\widetilde{S})$ and $\Lambda=\mathcal{P}(Q,S)$ for quivers with potential $(\widetilde{Q},\widetilde{S})$ and $(Q,S)$ again as in Subsection \ref{subsec:assumptions}. Fix a vertex $k\in[n]$, suppose that $\widetilde{Q}$ does not have 2-cycles incident to $k$, and write $\widetilde{\Lambda}':=\mathcal{P}(\mu_k(\widetilde{Q},\widetilde{S}))$ and $\Lambda':=\mathcal{P}(\mu_k(Q,S))$. For the Caldero-Chapoton algebra $\calCC_{\widetilde{\Lambda}'}(\Lambda')$, instead of taking the same tuple $\mathbf{x}$ we took for $\calCC_{\widetilde{\Lambda}}(\Lambda)$, we take the tuple $\mu_k^{\widetilde{Q}}(\mathbf{x}):=(x_1,\ldots,x_{k-1},u_k,x_{k+1},\ldots,x_n,x_{n+1},\ldots,x_{n+m})$, where
\begin{eqnarray}\label{eq:exchange-relation-1}
u_k&:=&\frac{\prod_{j\to k}x_j+\prod_{k\to i}x_i}{x_k},
\end{eqnarray}
with the products in the numerator running over arrows of $\widetilde{Q}$ (and not only over arrows of $Q$). Thus, while $\calCC_{\widetilde{\Lambda}}(\Lambda)$ is defined as an $R$-subalgebra of $R[\mathbf{x}^{\pm1}]:=R[x_1^{\pm1},\ldots,x_{k-1}^{\pm1},x_k^{\pm1},x_{k+1}^{\pm1},\ldots,x_n^{\pm1}]$, the ring $\calCC_{\widetilde{\Lambda}'}(\Lambda')$ is defined as an $R$-subalgebra of $R[\mu_k^{\widetilde{Q}}(\mathbf{x})^{\pm1}]:=R[x_1^{\pm1},\ldots,x_{k-1}^{\pm1},u_k^{\pm1},x_{k+1}^{\pm1}\ldots,x_n^{\pm1}]$. The relation between $\calCC_{\widetilde{\Lambda}}(\Lambda)$ and $\calCC_{\widetilde{\Lambda}'}(\Lambda')$ is not trivial. In particular, it is not obvious from the definitions whether or not $\calCC_{\widetilde{\Lambda}}(\Lambda)=\calCC_{\widetilde{\Lambda}'}(\Lambda')$ as subsets of $\calF$. The equality is actually a consequence of a very deep theorem of Derksen-Weyman-Zelevinsky that we describe briefly in the next paragraph.

Let $\mathcal{M}$ be a decorated representation of $\Lambda$ supported in $[n]$, and let $\mu_k(\mathcal{M})$ be the mutation defined by Derksen-Weyman-Zelevinsky in \cite[Equations (10.6)--(10.10) and (10.22)]{DWZ1}, which is a decorated representation of $\Lambda'$ also supported in $[n]$. It follows from Derksen-Weyman-Zelevinsky's \cite[Lemma 5.2]{DWZ2} that
\begin{equation}\label{eq:DWZ's-key-lemma}
\text{$CC_{\widetilde{\Lambda}}(\mathcal{M})$ and $CC_{\widetilde{\Lambda}'}(\mu_k(\mathcal{M}))$ are the same element of $\calF$.}
\end{equation}

Corollaries \ref{cor:1.8b} and \ref{cor:generic-stays-generic-after-DWZ-mutation-intro} follow easily from the following more general result.

\begin{coro}\label{coro:generic-stays-generic-after-DWZ-mutation} Let $(\widetilde{Q},\widetilde{S})$ and $(Q,S)$ be quivers with potential as in Subsection \ref{subsec:assumptions}. Fix a vertex $k\in[n]$, suppose that $\widetilde{Q}$ does not have 2-cycles incident to $k$, and write $\widetilde{\Lambda}:=\mathcal{P}(\widetilde{Q},\widetilde{S})$, $\Lambda:=\mathcal{P}(Q,S)$, $\widetilde{\Lambda}':=\mathcal{P}(\mu_k(\widetilde{Q},\widetilde{S}))$ and $\Lambda':=\mathcal{P}(\mu_k(Q,S))$. For any generically $\tau^-$-reduced irreducible component $Z\in\decirrsr(\Lambda)$ and any $\mathcal{M}\in Z$ such that
\begin{eqnarray*}
CC_{\widetilde{\Lambda}}(\mathcal{M})&=&CC_{\widetilde{\Lambda}}(Z)\in R[\mathbf{x}^{\pm1}],\\
\text{we have} \hspace{1cm} CC_{\widetilde{\Lambda}'}(\mu_k(\mathcal{M}))&=&CC_{\widetilde{\Lambda}'}(\mu_k(Z))\in R[\mu_k^{\widetilde{Q}}(\mathbf{x})^{\pm1}].
\end{eqnarray*}
Hence $CC_{\widetilde{\Lambda}}(Z)$ and $CC_{\widetilde{\Lambda}'}(\mu_k(Z))$ are the same element of the rational function field $\calF$. Therefore,
\begin{eqnarray*}
\mathcal{B}_{(\widetilde{Q},\widetilde{S})}(Q,S) &=& \mathcal{B}_{\mu_k(\widetilde{Q},\widetilde{S})}(\mu_k(Q,S)).
\end{eqnarray*}
\end{coro}

\begin{proof}
Take a generically $\tau^-$-reduced irreducible component $Z\in\decirrsr(\Lambda)$ and any $\mathcal{M}\in Z$ such that $CC_{\widetilde{\Lambda}}(\mathcal{M})=CC_{\widetilde{\Lambda}}(Z)$. Let $(\bfd',\bfv')$, $U$, $\nu_U$ and $\nu_U'$ be as in the statement of Theorem \ref{thm:mutation-of-strongly-reduced-components}, so that $\mu_k(Z)=\overline{\Image(\nu_U')}$. Let $V$ be a non-empty open subset of $\mu_k(Z)$ such that $CC_{\widetilde{\Lambda}'}(\mathcal{N})=CC_{\widetilde{\Lambda}'}(\mu_k(Z))$ for every $\mathcal{N}\in V$. Then $V\cap\Image(\nu_U')\neq\varnothing$, so $\nu_U'^{-1}(V)$ is a non-empty open subset of  $\GL_{d'}(\C)\times U$. Hence, there exists $(g,\mathcal{L})\in \nu_U'^{-1}(V)$ such that $CC_{\widetilde{\Lambda}}(\mathcal{L})=CC_{\widetilde{\Lambda}}(Z)$. Then:
\begin{eqnarray*}
CC_{\widetilde{\Lambda}'}(\mu_k(\mathcal{L})) & = & CC_{\widetilde{\Lambda}'}(\nu_k(\mathcal{L}))\\
&=& CC_{\widetilde{\Lambda}'}(g\cdot\nu_k(\mathcal{L}))\\
&=&CC_{\widetilde{\Lambda}'}(\mu_k(Z)).
\end{eqnarray*}
Since $CC_{\widetilde{\Lambda}}(\mathcal{M})=CC_{\widetilde{\Lambda}}(Z)=CC_{\widetilde{\Lambda}}(\mathcal{L})$, the result follows by applying \eqref{eq:DWZ's-key-lemma} to $\mathcal{M}$ and to $\mathcal{L}$.
\end{proof}

As said in Subsection \ref{subsec:irred-comps}, the following corollary can be obtained from a direct combination of Theorems \ref{thm:criterion-for-lin-ind-of-CC-functions}, \ref{thm:different-s.r.i.c.-have-different-g-vectors} and Corollary~\ref{coro:generic-stays-generic-after-DWZ-mutation}.

\begin{coro}\label{coro:lin-ind-of-generic-basis-in-general} If $(\widetilde{Q},\widetilde{S})$ is non-degenerate and at least one matrix $B'$ in the mutation-equivalence class of $B$ satisfies $\Ker(B')\cap\ZZ_{\geq 0}^n$, then
the upper cluster algebra with geometric coefficients $\mathcal{A}(\widetilde{B})$ has $\mathcal{B}_{(\widetilde{Q},\widetilde{S})}(Q,S)$ as an $R$-linearly independent subset. 
\end{coro}

\subsection{Cluster monomials as CC-functions with coefficients of $E$-rigid components}\label{subsec:cluster-variables}

This section has two aims. One is to illustrate the necessity to assume that $(\widetilde{Q},\widetilde{S})$, and not only $(Q,S)$, be non-degenerate in order to be able to express cluster variables as Caldero-Chapoton functions. The second aim is to recall how cluster monomials can be written as Caldero-Chapoton functions of generically $\tau^-$-reduced components of representation spaces of quivers with potential.

Suppose we are given quivers with potential $(\widetilde{Q},\widetilde{S})$ and $(Q,S)$ as in Subsection \ref{subsec:assumptions}, so that $(Q,S)$ is the restriction of $(\widetilde{Q},\widetilde{S})$ and $(Q,S)$ to the vertex subset $[n]$ of $\widetilde{Q}_0=[n+m]$. Thus, in particular, $Q$ is the union of the isolated vertices $n+1,\ldots,n+m$ and the full subquiver of $\widetilde{Q}$ determined by $[n]$. It is perfectly possible that $(\widetilde{Q},\widetilde{S})$ be degenerate and $(Q,S)$ be non-degenerate. Would this affect whether we hit or miss the cluster variables of the cluster algebra with coefficients of $\widetilde{Q}$ when we apply the Caldero-Chapoton map? The following two examples are meant to illustrate what happens.

\begin{ex} Let $\widetilde{Q}$ and $Q$ be the quivers
$$
\xymatrix{1 \ar[dr]_{\gamma} && 2 \ar[ll]_{\alpha}\\
& 3\ar[ur]_{\beta}}
\ \ \ \ \ 
\xymatrix{1  && 2 \ar[ll]_{\alpha}\\
& 3}
$$
That is, we are taking $n=2$, $m=1$, $n+m=3$. Let $\widetilde{S}:=\alpha\beta\gamma$ and $S:=\widehat{\varphi}(\widetilde{S})=0$.  Both $(\widetilde{Q},\widetilde{S})$ and $(Q,S)$ certainly are non-degenerate.
The decorated representation
$$
\xymatrix{\mathbb{C}  && \C \ar[ll]_{\myid}\\
& 0}
\ \ \ \ \ 
\xymatrix{0  && 0\\
& 0}
$$
of $(Q,S)$ is supported in $[2]$. A straightforward computation yields
\begin{eqnarray*}
\widetilde{\g}_{\mathcal{M}} &=& (0,-1,0)\\
F_{\mathcal{M}} &=& 1+X_1+X_1X_2.
\end{eqnarray*}
Since $\widetilde{B}=\left[\begin{array}{rr}0 & 1\\ -1 & 0\\ 1 & -1 \end{array}\right]$, we have
\begin{eqnarray*}
CC_{(\widetilde{Q},\alpha\beta\gamma)}(\mathcal{M}) &=& x_1^{-1}(1+x_2^{-1}x_3+x_1x_2^{-1}),
\end{eqnarray*}
which is a cluster variable in the cluster algebra with coefficients $\mathcal{A}(\widetilde{B})$ that has the initial seed $(\widetilde{B},(x_1,x_2,x_3))$ with $x_3$ frozen.
\end{ex}

\begin{ex} This example is meant to illustrate that in order to hit cluster variables with the CC-formula, the assumption that $(\widetilde{Q},\widetilde{S})$, and not only $(Q,S)$, be a non-degenerate QP, is essential. Let $\widetilde{Q}$ and $Q$ be the exact same quivers from the preceding example, and let $\mathcal{M}$ be the same decorated representation as before. The potential $0$ on $\widetilde{Q}$ certainly has the property that its image under $\widehat{\varphi}$ is a non-degenerate potential on $Q$. A straightforward computation yields
\begin{eqnarray*}
\widetilde{\g}_{\mathcal{M}} &=& (0,-1,1)\\
F_{\mathcal{M}} &=& 1+X_1+X_1X_2.
\end{eqnarray*}
Since $\widetilde{B}=\left[\begin{array}{rr}0 & 1\\ -1 & 0\\ 1 & -1 \end{array}\right]$, we have
\begin{eqnarray*}
CC_{(\widetilde{Q},0)}(\mathcal{M}) &=& x_1^{-1}x_3(1+x_2^{-1}x_3+x_1x_2^{-1}),
\end{eqnarray*}
which is not a cluster variable in the cluster algebra with coefficients $\mathcal{A}(\widetilde{B})$ that has the initial seed $(\widetilde{B},(x_1,x_2,x_3))$ with $x_3$ frozen.
\end{ex}

On the positive side, when $(\widetilde{Q},\widetilde{S})$ is non-degenerate, everything is fine: all cluster variables and cluster monomials can be obtained through the CC-formula. This is a deep theorem of Derksen-Weyman-Zelevinsky \cite[Theorem 5.1]{DWZ2}, whose proof in turn uses another deep theorem, Fomin-Zelevinsky's \emph{separation of additions} \cite[Corollary 6.3]{FZ4}.  

\begin{thm}\cite[Theorem 5.1]{DWZ2}\label{thm:DWZ's-thm-cluster-monomials-are-CC-functions} Let $\mathcal{M}$ be a decorated representation of $(Q,S)$ supported in $[n]$. Under the assumption that $(\widetilde{Q},\widetilde{S})$ be non-degenerate, if $\mathcal{M}$ is mutation-equivalent to a negative representation (under a sequence of QP-mutations involving only vertices from $[n]$), then $CC_{(\widetilde{Q},\widetilde{S})}(\mathcal{M})$ is a cluster monomial in the cluster algebra with coefficients $\mathcal{A}(\widetilde{B})$ that has  $(\widetilde{B},(x_1,\ldots,x_n,x_{n+1},\ldots,x_{n+m}))$ as its initial seed with $x_{n+1},\ldots,x_{n+m}$ frozen. All cluster monomials in $\mathcal{A}(\widetilde{B})$ arise in this way.
\end{thm}

Suppose that $\mathbf{u}=(u_1,\ldots,u_n,x_{n+1},\ldots,x_{n+m})$ is a cluster in the cluster algebra with coefficients $\mathcal{A}(\widetilde{B})$. Let $\mathcal{M}_1=(M_1,\bfv_1),\ldots,\mathcal{M}_n=(M_n,\bfv_n)$ be decorated representations of $(Q,S)$ supported in $[n]$ such that $u_j=CC_{(\widetilde{Q},\widetilde{S})}(\mathcal{M}_j)$ for $j\in[n]$, and set $\bfd_j:=\underline{\dim}(M_j)$ for each $j\in[n]$. The following corollary is a straightforward consequence of Theorem \ref{thm:DWZ's-thm-cluster-monomials-are-CC-functions}.

\begin{coro} Suppose that $(\widetilde{Q},\widetilde{S})$ is a non-degenerate quiver with potential. For any choice of non-negative integers $a_1,\ldots,a_n$,
\begin{enumerate}
\item the Zariski closure $Z_j$ of the $\GL_{\bfd_{j}}(\C)$-orbit of $\mathcal{M}_j$ in $\decrep(\mathcal{P}(Q,S),\bfd_j,\bfv_j)$ is a generically $\tau^-$-reduced irreducible component;
\item the Zariski closure $Z$ of the $\GL_{\sum_{j\in[n]}a_j\bfd_j}(\C)$-orbit of $\mathcal{M}:=(M_1^{a_1}\oplus\cdots\oplus M_n^{a_n},\sum_{j\in[n]}\bfv_j)$ in $\decrep(\mathcal{P}(Q,S),\sum_{j\in[n]}a_j\bfd_j,\sum_{j\in[n]}\bfv_j)$ is a generically $\tau^-$-reduced irreducible component;
\item $Z=\overline{Z_1^{a_1}\oplus\cdots\oplus Z_n^{a_j}}$;
\item $u_1^{a_1}\cdots u_n^{a_n}=CC_{\widetilde{\Lambda}}(Z)$.
\end{enumerate}
\end{coro}

\section{Applications to surface cluster algebras}

\subsection{Bypasses of a gentle quiver with potential and matrix column identities}\label{sec:formula-for-special-bands}\label{sec:formula-for-bands-of-gentle-QPs}

\begin{defi} Let $\cR$ be a finite set of paths on $Q$. We will say that $(Q,\cR)$ is a \emph{gentle pair} provided the following conditions are simultaneously satisfied:
\begin{itemize}
\item Each vertex of $Q$ is the head of at most two arrows;
\item each vertex of $Q$ is the tail of at most two arrows;
\item
$\cR$ consists of paths of length $2$;
\item
if for some arrow $\alp\in Q_1$ there exist two paths
$\bet_1\alp\neq\bet_2\alp$, then precisely one of these paths belongs
to $\cR$;
\item
if for some arrow $\bet\in Q_1$ there exist two different paths
$\bet\alp_1\neq \bet\alp_2$, then precisely one of these paths belongs
to $\cR$;
\item
$\C\langle Q\rangle/\ebrace{\cR}$ is a finite dimensional algebra, where $\ebrace{\cR}$ is the two-sided ideal of $\C\langle Q\rangle$ generated by $\cR$.
\end{itemize}
\end{defi}

\begin{defi}
The quiver with potential $(Q,S)$ is a \emph{gentle QP} if $Q$ has no loops and
$(Q,\partial S)$ is a gentle pair, where $\partial S:=\{\partial_a(S)\suchthat a\in Q_1\}$.
\end{defi}

Note that if $(Q,S)$ is a gentle QP, then $S$ is necessarily a linear combination of
3-cycles. Moreover, if in $S$ two 3-cycles which share an arrow,
appear with non-zero coefficients, they coincide up to rotation.
Finally, up to a change of arrows we may assume that in this situation all coefficients in
$S$ are $1$.

\begin{defi}
Let $(Q,\cR)$ be a gentle pair and $\pi=\bet_{m-1}\cdots\bet_2\bet_1$ a path on $Q$.
\begin{enumerate}
\item Following Crawley-Boevey--Happel--Ringel \cite{CB-H-R}, we will say that $\pi$ is a \emph{bypass} if $\bet_i\bet_{i-1}\not\in\cR$ for $i=2,3,\ldots, m-1$, and there exists an arrow
$\bet_m\in Q_1\setminus\{\bet_1,\bet_{m-1}\}$ with $s(\bet_m)=s(\bet_1)$ and
$t(\bet_m)=t(\bet_{m-1})$, see the left hand side of Figure \ref{Fig:g_vector_band};
\item we will say that $\pi$ is an \emph{almost bypass} if $\bet_i\bet_{i-1}\not\in\cR$ for
$i=2,3,\ldots, m-1$, and there exists a length-$2$ path $\bet_{m+1}\bet_m$ such that $\bet_{m+1},\bet_m\in Q_1\setminus\{\bet_1,\bet_{m-1}\}$, $\bet_{m+1}\bet_m\not\in\cR$, $s(\bet_m)=s(\bet_1)$,
$t(\bet_{m+1})=t(\bet_{m-1})$, and the only arrows of $Q$ that are incident to $t(\bet_m)=s(\bet_{m+1})$ are precisely $\bet_m$ and $\bet_{m+1}$, see the right hand side of Figure \ref{Fig:g_vector_band}.
\end{enumerate}
The arcs $s(\pi):=s(\bet_1)$ and $t(\pi):=t(\bet_{m-1})$ will be called the \emph{source} and the \emph{sink} of $\pi=\bet_{m-1}\cdots\bet_2\bet_1$.
\end{defi}
        \begin{figure}[!h]
                \caption{Left: The arrows and relations directly connected to the vertices of a bypass of a gentle QP. Right: The arrows and relations directly connected to the vertices of an almost bypass of a gentle QP.}\label{Fig:g_vector_band}
                \centering
                \includegraphics[scale=.65]{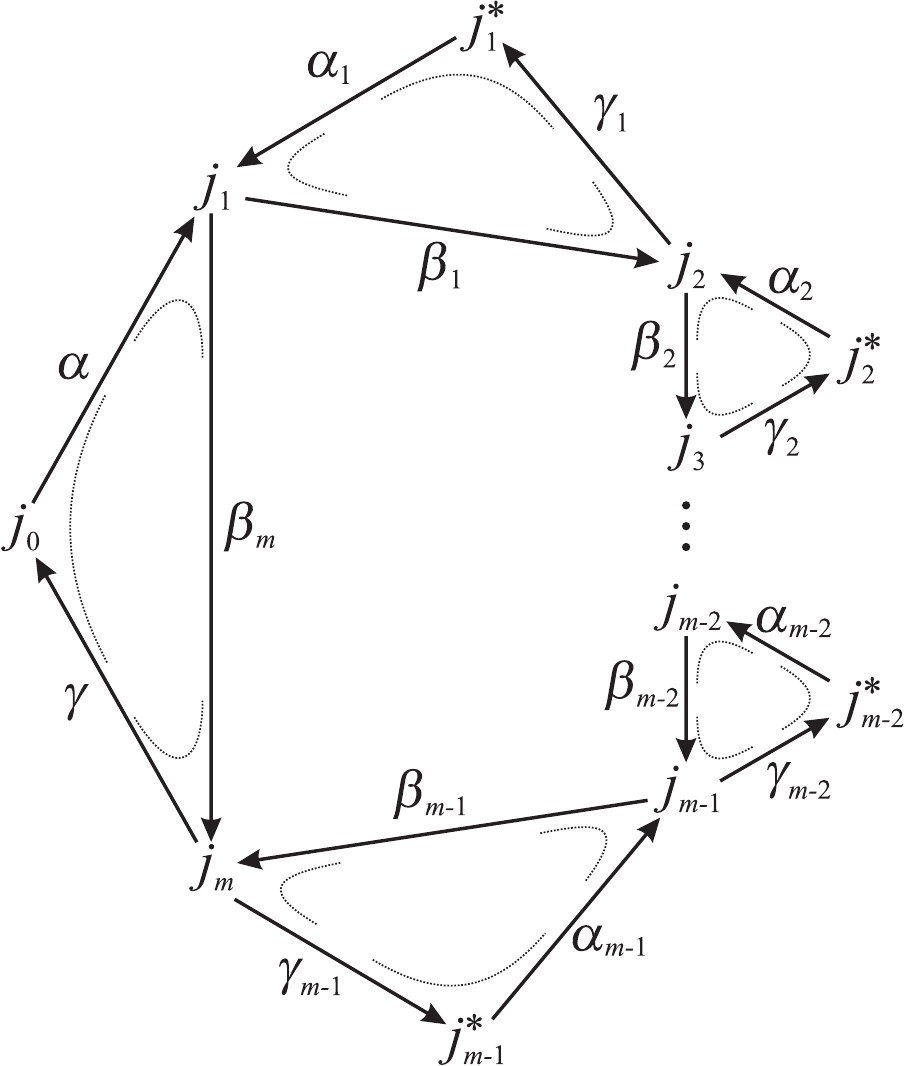}\ \ \ \ \ \ \ \  \includegraphics[scale=.079]{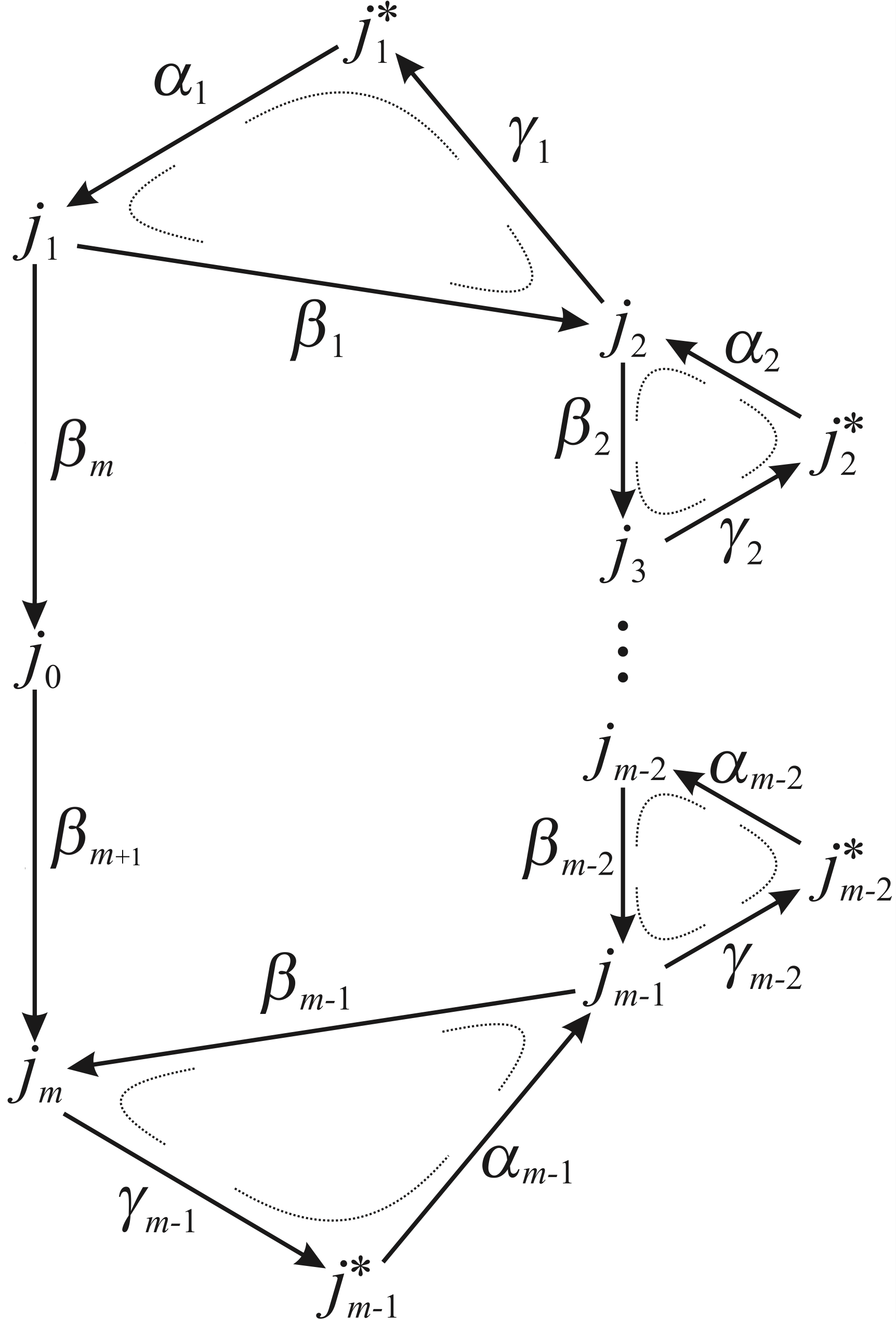}
        \end{figure}

\begin{prop}\label{prop:sum-of-columns}
Suppose $(Q,S)$ is a gentle QP, and $\pi=\bet_{m-1}\cdots\bet_2\bet_1$ either a bypass or an almost bypass.
Set $j_k:=s(\bet_k)$ for $k=1,2,\ldots, m-1$ and $j_m:=t(\bet_{m-1})$, and denote by $\mathbf{b}_{j_k}$ the $j_k^{\operatorname{th}}$ column of the matrix $B=B_{Q}$. Then
\[
\sum_{k=1}^m \mathbf{b}_{j_k} = \begin{cases}2(-\mathbf{e}_{j_1}+\mathbf{e}_{j_m}) & \text{if $\pi$ is special;}\\
(-\mathbf{e}_{j_1}+\mathbf{e}_{j_m}) & \text{if $\pi$ is almost special.}
\end{cases}
\]
where $\mathbf{e}_i$ denotes the $i^{\operatorname{th}}$ standard basis vector.
\end{prop}

\begin{proof} Suppose that $\pi$ is a bypass. We need the following definitions:
\begin{align*}
J_0 &:= \{ s(\alp)\mid \alp\in Q_1 \text{ and } \bet_1\alp\not\in\partial S\}\\
J_k &:= \{j_k\} \text{ for } k\in\{1,2,\ldots,m\} \\
J_{m+1}
    &:= \{ t(\gam) \mid\gam\in Q_1 \text{ and } \gam\bet_{m-1}\not\in\partial S\}
  \end{align*}
Note, that $J_0=J_{m+1}$, and in case $J_0\neq\varnothing$ we have $J_0=\{j_0\}$
for a unique element $j_0\in Q_0$.
Similarly, we define
\begin{align*}
J_k^- &:=\{ s(\alp)\mid \alp\in Q_1\text{ and } \bet_k\alp\in\partial S\}
\text{ for } k\in\{1,2,\ldots,m-1\},\\
J_k^+ &:=\{ t(\gam)\mid \gam\in Q_1\text{ and } \gam\bet_{k-1}\in\partial S\}
\text{ for } k\in\{2,3,\ldots,m\},
\end{align*}
moreover $J_m^-:=\{j_1\}$ and $J_1^+:=\{j_m\}$.
Note, that $J_k^-=J_{k+1}^+$  for $k\in\{1,2,\ldots,m-1\}$ and in case
$J_k^-\neq\varnothing$ we have
$J_k^-=\{j_k^*\}$ for a unique element $j_k^*\in Q_0$.
The left hand side of Figure \ref{Fig:g_vector_band} illustrates these definitions.

By the definition of the signed adjacency matrix $B$ of $Q$, the $(i,j_k)$-entry of $B$ is 
\[
B_{i,j_k}= \abs{J_{k+1}\cap\{i\}}+\abs{J_k^+\cap\{i\}}-
(\abs{J_{k-1}\cap\{i\}}+\abs{J_k^-\cap\{i\}})
\]
for $k=1,2,\ldots,m$.
Thus, the $i^{\operatorname{th}}$ entry of $\sum_{k=1}^m \mathbf{b}_{j_k}$ is
\begin{align*}
\sum_{k=1}^m B_{i,j_k} &=
\sum_{k=1}^m(\abs{J_{k+1} \cap \{ i \} } - \abs{J_{k-1} \cap \{ i \} })+
\sum_{k=1}^m (\abs{J_k^+ \cap \{ i \} })-  \abs{J_k^- \cap \{ i \} })\\
&= (\abs{J_m\cap\{i\}}-\abs{J_1\cap\{i\}}) +
   (\abs{J_1^+\cap\{i\}} -\abs{J_m^-\cap\{i\}})\\
&= 2(\abs{\{j_m\}\cap\{i\}} - \abs{\{j_1\}\cap\{i\}})
\end{align*}
and Proposition \ref{prop:sum-of-columns} follows for bypasses.

The proof of Proposition \ref{prop:sum-of-columns} for almost bypasses is similar.
\end{proof}

The next lemma follows after a moment's thought.

\begin{lemma} If $\surf$ is an unpunctured surface having exactly one marked point per boundary component, and $\tau$ is a triangulation of $\surf$, then around each boundary component runs a bypass of the quiver $Q(\tau)$.
\end{lemma}

\begin{ex} In Figure \ref{Fig:ex_special_bands} we can see a triangulation $\tau$ of an unpunctured surface of genus $0$ with exactly $3$ boundary components and exactly one marked point on each boundary component.
        \begin{figure}[!h]
                \caption{In an unpunctured surface having exactly one marked point on each boundary component, a bypass runs around each boundary component.}\label{Fig:ex_special_bands}
                \centering
                \includegraphics[scale=0.03]{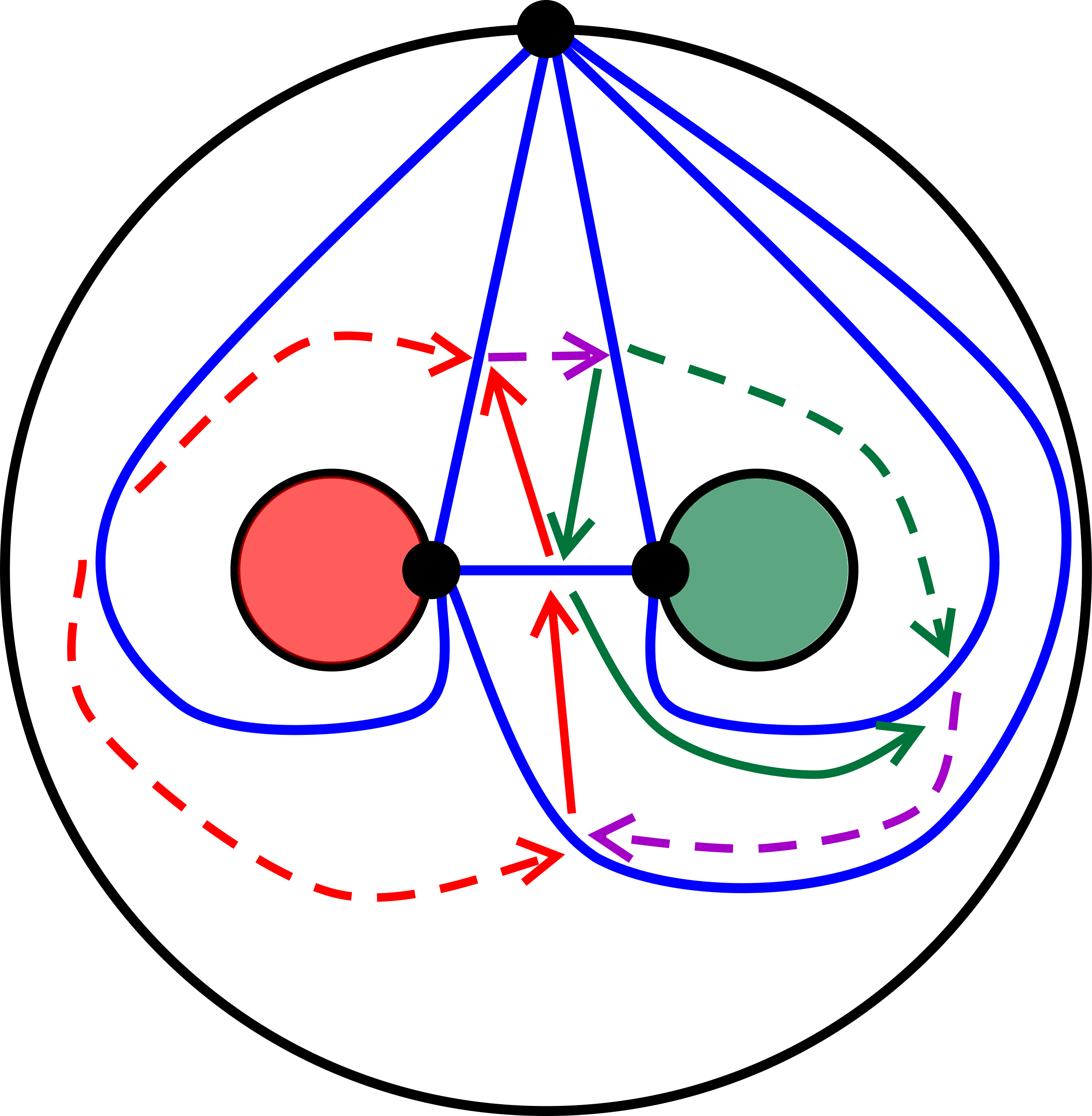}
        \end{figure}
The quiver $Q(\tau)$ has been drawn on the surface. For each boundary component, the bypass that runs around it has been indicated (for the ``red'' boundary component, the band is indicated in red, for the ``green'' boundary component, the band is indicated in green, and for the ``outer'' boundary component, the band is indicated by dashed arrows).
\end{ex}

\subsection{Linear independence of the set of generically $\tau^-$-reduced CC-functions}
\label{sec:lin-ind-for-surfaces}

We will need Fomin-Shapiro-Thurston's characterization of the corank of the signed-adjacency matrix of a triangulation of $\surf$, namely:

\begin{thm}\cite[Theorem 14.3]{FST}\label{thm:corank-of-Btau} For any surface with marked points $\surf$ and any triangulation $\sigma$ of $\surf$, the rank of the skew-symmetric matrix $B(\sigma)$ is equal to $|\sigma|-|\punct|-c_{\operatorname{even}}$, where $c_{\operatorname{even}}$ is the number of boundary components of $\Sigma$ that contain an even number of marked~points.
\end{thm}

The key result that will allow us to show the linear independence of the set of generically $\tau^-$-reduced CC-functions for any choice of geometric coefficients is the following.

\begin{thm}\label{thm:existence-triang-with-nice-matrix} Every surface $\surf$ such that $\partial\Sigma\neq\varnothing$ admits a triangulation $\sigma$ whose associated skew-symmetric matrix $B(\sigma)$ satisfies $\Ker(B(\sigma))\cap\ZZ^n_{\geq 0}$.
\end{thm}

\begin{proof} It suffices to show that there exists a triangulation $\sigma$ of $\surf$ such that $B(\sigma)$ satisfies condition \eqref{item:B-satisfies-columns-condition} of the first item of Remark \ref{rem:the-5-conditions-on-B}. We shall prove this constructively in two steps: first for unpunctured surfaces, then for punctured ones. 

Consider an unpunctured surface $(\Sigma,\marked_0)$ with exactly one marked point on each boundary component of $\Sigma$. Let $\tau_0$ be any triangulation of $(\Sigma,\marked_0)$ and $\kappa$ a boundary component of $\Sigma$. Since there is exactly one point from $\marked_0$, say $m_\kappa$, lying on $\kappa$, the entire $\kappa$ is a side of a unique triangle $\triangle_\kappa$ of $\tau_0$. The other two sides of $\triangle_\kappa$ are then arcs incident to $m_\kappa$, see Figure  \ref{Fig:gent_triang_one_pt_per_comp}.
        \begin{figure}[!h]
                \caption{}\label{Fig:gent_triang_one_pt_per_comp}
                \centering
                \includegraphics[scale=.25]{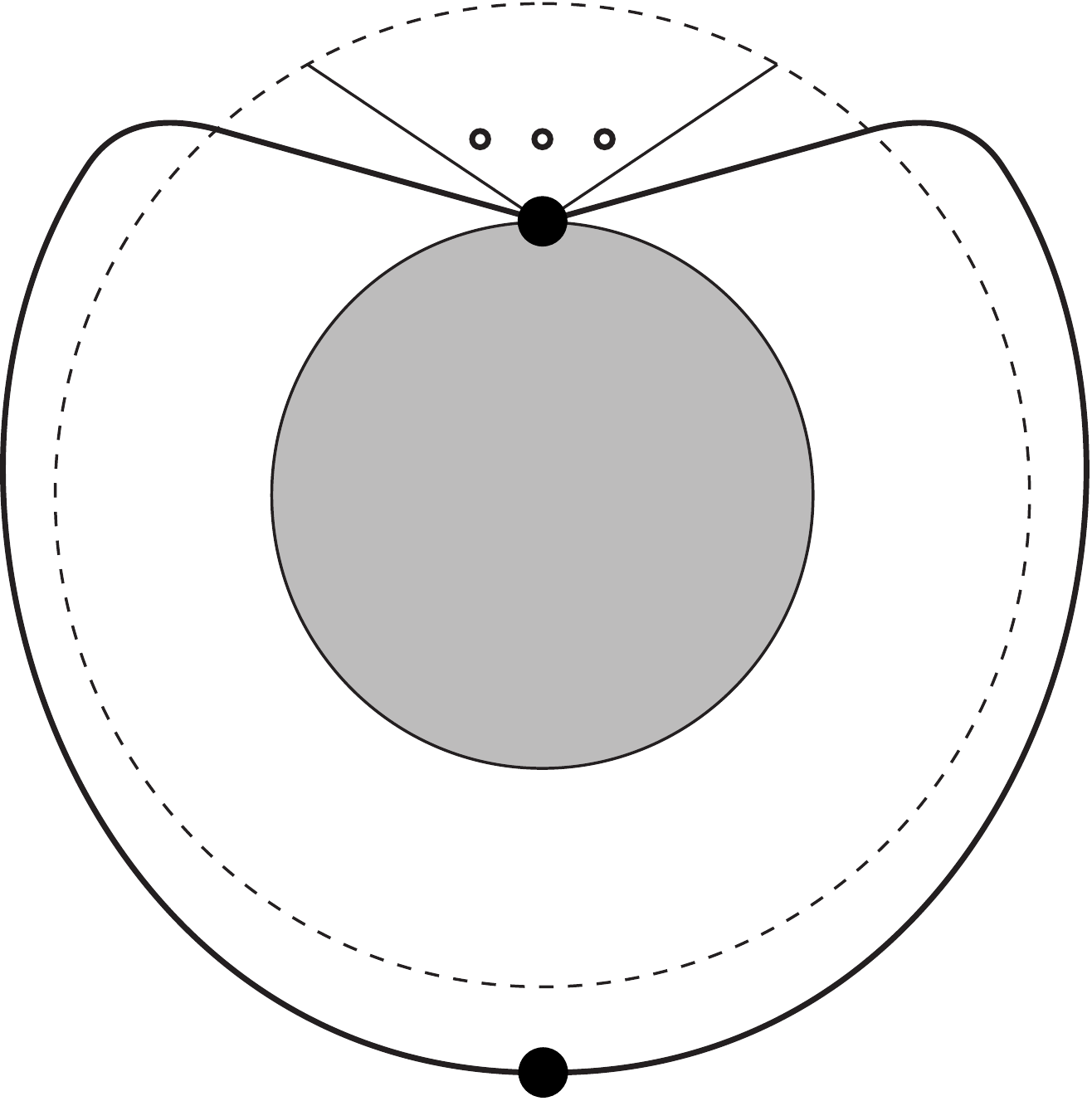}
        \end{figure}

For any positive integer $c_\kappa$ pick $c_\kappa-1$ distinct points $m_{\kappa,2},\ldots,m_{\kappa,c_\kappa}$ on $\kappa\setminus\{m_\kappa\}$, ordered along $\kappa$ as shown in Figure \ref{Fig:picking_points_on_component} on the left.
        \begin{figure}[!h]
                \caption{}\label{Fig:picking_points_on_component}
                \centering
                \includegraphics[scale=.15]{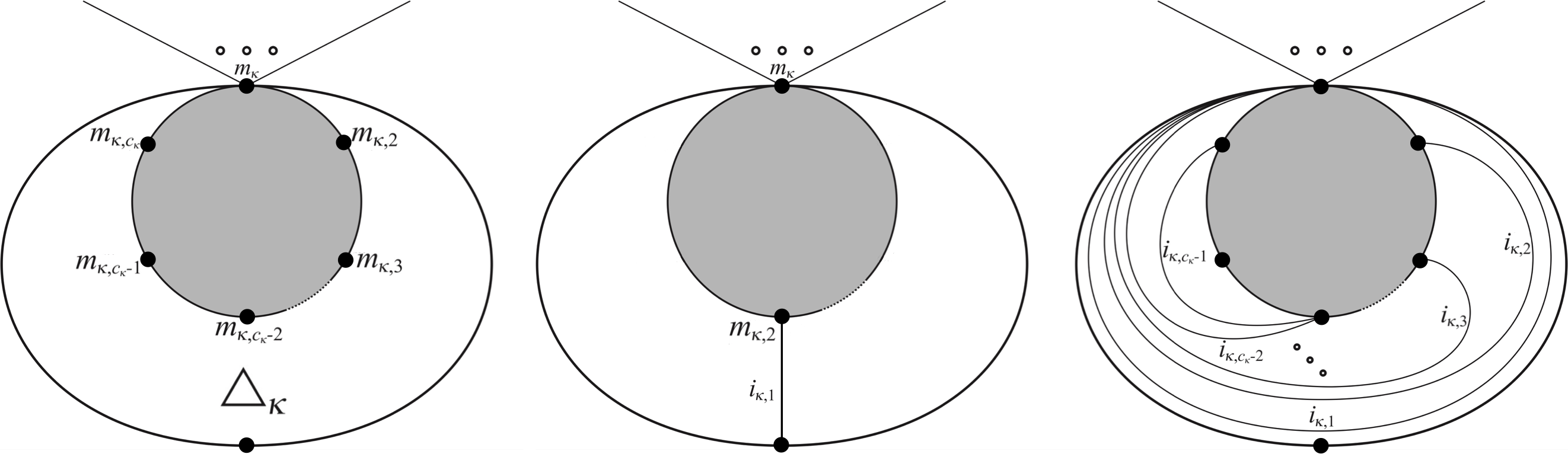}
        \end{figure}
Then draw the $c_\kappa-1$ arcs $i_{\kappa,1},i_{\kappa,2},\ldots,i_{\kappa,c_\kappa-1}$ contained in $\triangle_\kappa$ that are shown in that same Figure on the center (if $c_\kappa=2$) and on the right (if $c_\kappa>2$). See also Example \ref{ex:good-triangs-for-annuli} below; note that if $c_\kappa\geq 4$, then all but one of these arcs emanate from $m_\kappa$. After doing this for each boundary component $\kappa$ of $\Sigma$, we obtain a triangulation $\tau_1$ of an unpunctured surface $(\Sigma,\marked_1)$, where $\marked_1$ is a set containing exactly $c_\kappa$ points from each boundary component $\kappa$ of $\Sigma$.

Notice that $\tau_0\subseteq\tau_1$ 
and that for every boundary component $\kappa$ such that $c_\kappa\geq 4$, the bypass $\pi_\kappa$ on $Q(\tau_0)$ that runs around $\kappa$ is also a bypass on $Q(\tau_1)$. When $c_\kappa=2$, $\pi_\kappa$ is an almost bypass on $Q(\tau_1)$.

Suppose $c_\kappa$ is an even integer. A straightforward computation shows that
\begin{equation}\label{eq:positive-combination}
\mathbf{b}(\tau_1)_{i_{\kappa,c_{\kappa}-1}}=\sum_{\ell=1}^{\frac{c_{\kappa}-2}{2}}\mathbf{b}(\tau_1)_{i_{\kappa,2\ell-1}}
+\mathbf{e}_{t(\pi_{\kappa})}-\mathbf{e}_{s(\pi_\kappa)},
\end{equation}
where $\mathbf{e}_{t(\pi_{\kappa})}$ and $\mathbf{e}_{s(\pi_\kappa)}$ are standard basis vectors corresponding to the arcs $t(\pi_{\kappa})$ and $s(\pi_\kappa)$ of $\tau_1$. Since $\pi_\kappa$ is either a bypass or an almost bypass, we can apply Proposition \ref{prop:sum-of-columns} and deduce that $\mathbf{e}_{t(\pi_{\kappa})}-\mathbf{e}_{s(\pi_\kappa)}$ is a $\mathbb{Q}_{\geq 0}$-linear combination of the columns of $B(\tau_1)$ that correspond to arcs in $\tau_0\subseteq \tau_1\setminus\{i_{\kappa',c_{\kappa'}-1}\suchthat c_{\kappa'}$ is even$\}$. Since that cardinality of $\tau_1\setminus\{i_{\kappa',c_{\kappa'}-1}\suchthat c_{\kappa'}$ is even$\}$ is precisely $\rank(B(\tau_1))$ by Theorem~\ref{thm:corank-of-Btau}, this proves Theorem \ref{thm:existence-triang-with-nice-matrix} for unpunctured surfaces.

To move to the punctured case, we take a boundary component $\kappa$ as follows. If every boundary component has an odd number of marked points from $\marked_1$, we take $\kappa$ arbitrarily. Otherwise, if at least one boundary component has an even number of marked points from $\marked_1$, then we take it to satisfy $c_\kappa\in 2\mathbb{Z}$.

If $c_\kappa\geq 2$, then we set $\tau_2=f_{i_{\kappa,c_{\kappa}-1}}(\tau_1)=(\tau_1\setminus\{i_{\kappa,c_{\kappa}-1}\})\cup\{i_{\kappa,c_{\kappa}-1}'\}$, which is another triangulation of the unpunctured surface $(\Sigma,\marked_1)$. But if $c_\kappa=1$, we set $\tau_2=\tau_1$. Notice that, in any case, $\tau_0\subseteq\tau_2$.

Let $q$ be any positive integer. If $q=1$ we pick a puncture $p_1$ and draw three arcs $j_{1,1},j_{1,2},j_{1,3}$ on $(\Sigma,\marked_1\cup\{p_1\})$ as shown in Figure \ref{Fig:picking_punctures_one}, so that the three arcs drawn, together with the arcs in $\tau_2$, form a triangulation $\sigma$ that has a self-folded triangle based at $m_{\kappa,c_{\kappa}}$ and enclosing $p_1$.
 \begin{figure}[!h]
                \caption{Triangulation $\sigma$ of the once-punctured surface $(\Sigma,\marked_1\cup\{p_1\})$.}\label{Fig:picking_punctures_one}
                \centering
                \includegraphics[scale=.13]{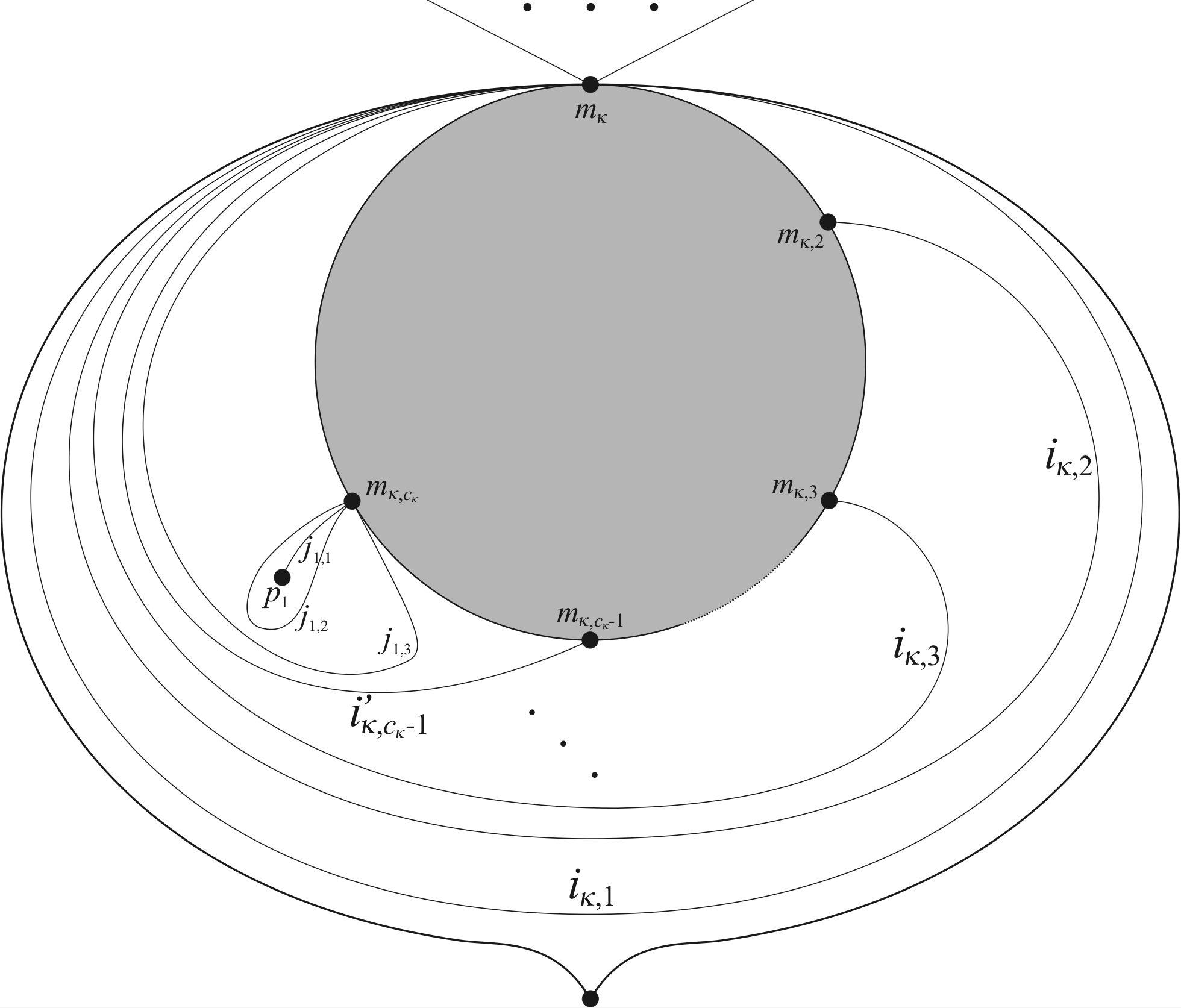}
        \end{figure}
Otherwise, if $q>1$ we pick $q$ distinct punctures $p_1,\ldots,p_q$, and draw $3q$ arcs $j_{1,1},j_{1,2},j_{1,3},\ldots,j_{q,1},j_{q,2},j_{q,3}$ on $(\Sigma,\marked_1\cup\{p_1,\ldots,p_q\})$ as shown in Figure \ref{Fig:picking_punctures_many}, so that the $3q$ arcs drawn, together with the arcs in $\tau_2$, form a triangulation $\sigma$ that has $q-1$ self-folded triangles based at $m_{\kappa}$ and respectively enclosing $p_1,\ldots,p_{q-1}$, and a self-folded triangle based at $m_{\kappa,c_{\kappa}}$ and enclosing $p_q$.
\begin{figure}[!h]
                \caption{Triangulation $\sigma$ of the $q$-times-punctured surface $(\Sigma,\marked_1\cup\{p_1,\ldots,p_q\})$ for $q>1$.}\label{Fig:picking_punctures_many}
                \centering
                \includegraphics[scale=.17]{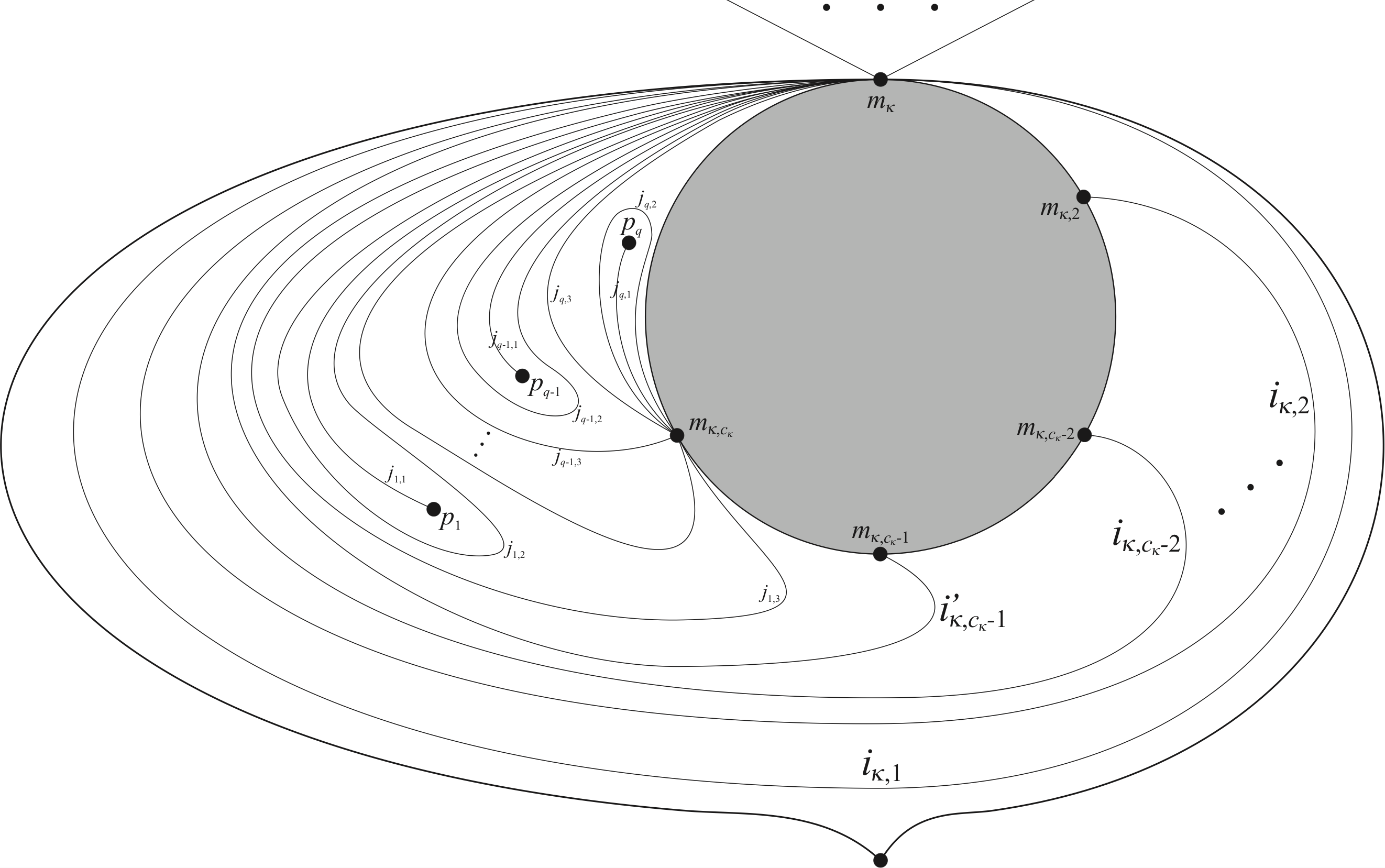}
        \end{figure}

In any case, regardless of whether $q=1$ or not, the triangulation $\tau_0$ of $(\Sigma,\marked_0)$ is contained in the triangulation $\sigma$ of $(\Sigma,\marked_1\cup\{p_1,\ldots,p_q\})$, and $Q(\tau_0)$ is a full subquiver of $Q(\sigma)$. In particular, for every boundary component $\kappa'$, all arcs appearing in the bypass $\pi_{\kappa'}$ of $Q(\tau_0)$ that runs around $\kappa'$ belong to $\sigma$. For $\kappa'\neq\kappa$, if $c_{\kappa'}$ is even, direct computation shows that
\begin{equation}\label{eq:positive-combination-sigma-kappa'}
\mathbf{b}(\sigma)_{i_{\kappa',c_{\kappa'}-1}}=\sum_{\ell=1}^{\frac{c_{\kappa'}-2}{2}}\mathbf{b}(\sigma)_{i_{\kappa',2\ell-1}}
+\mathbf{e}_{t(\pi_{\kappa'})}-\mathbf{e}_{s(\pi_{\kappa'})},
\end{equation}
whereas if $c_{\kappa}$ is even, then
\begin{equation}\label{eq:positive-combination-sigma-kappa}
\mathbf{b}(\sigma)_{j_{q,1}}=\left(\sum_{t=1}^{q-1}\mathbf{b}(\sigma)_{j_{t,1}}\right)+
\left(\sum_{\ell=1}^{\frac{c_{\kappa}-2}{2}}\mathbf{b}(\sigma)_{i_{\kappa,2\ell-1}}\right)+
\mathbf{b}(\sigma)_{i'_{\kappa,c_\kappa-1}}+
\mathbf{e}_{t(\pi_\kappa)}-\mathbf{e}_{s(\pi_\kappa)}.
\end{equation}
Note that, regardless of whether $\kappa'=\kappa$ or not, the adjacencies between arcs of $\sigma$ and arcs in the bypass $\pi_{\kappa'}$ are precisely the adjacencies between arcs of $\tau_1$ and arcs in the bypass $\pi_{\kappa'}$. This means that in \eqref{eq:positive-combination-sigma-kappa'} and \eqref{eq:positive-combination-sigma-kappa}, $\mathbf{e}_{t(\pi_{\kappa'})}-\mathbf{e}_{s(\pi_{\kappa'})}$ and $\mathbf{e}_{t(\pi_\kappa)}-\mathbf{e}_{s(\pi_\kappa)}$ are still expressible as $\mathbb{Q}_{\geq0}$-linear combinations of the columns of $B(\sigma)$ that correspond to the arcs appearing in $\pi_{\kappa'}$ and $\pi_\kappa$.

Notice also that for all $l=1,\ldots,q$, the columns $\mathbf{b}(\sigma)_{j_{l,1}}$ and $\mathbf{b}(\sigma)_{j_{l,2}}$ of $B(\sigma)$ coincide.

Summarizing, if $c_{\kappa}$ is even, then the columns of $B(\sigma)$ that correspond to arcs in the set $S=\{j_{l,2}\suchthat l=1,\ldots,q-1\}\cup\{j_{q,1},j_{q,2}\}\cup\{i_{\kappa',c_{\kappa'}-1}\suchthat \kappa'\neq\kappa \ \text{and} \ c_{\kappa'}\in2\mathbb{Z}\}$ can be expressed as $\mathbb{Q}_{\geq 0}$-linear combinations of the columns of $B(\sigma)$ in the complement of $S$. And if $c_{\kappa}$ is odd, then the columns of $B(\sigma)$ that correspond to arcs in the set $S=\{j_{l,2}\suchthat l=1,\ldots,q\}\cup\{i_{\kappa',c_{\kappa'}-1}\suchthat \kappa'\neq\kappa \ \text{and} \ c_{\kappa'}\in2\mathbb{Z}\}$ can be expressed as $\mathbb{Q}_{\geq 0}$-linear combinations of the columns of $B(\sigma)$ in the complement $S$. Since the cardinality of $S$ is precisely the corank of $B(\sigma)$, Theorem \ref{thm:existence-triang-with-nice-matrix} is proved.
\end{proof}

\begin{coro}\label{coro:lin-ind-for-surfs-arbitrary-coeffs} If $Q=Q(\tau)$ for some tagged triangulation $\tau$ of a surface $\surf$ with non-empty boundary, and $(\widetilde{Q},\widetilde{S})$ is non-degenerate, then $\calB_{(\widetilde{Q},\widetilde{S})}(Q,S)$ is linearly independent over $R$.
\end{coro}

\begin{proof} This is a direct combination of Corollary \ref{coro:lin-ind-of-generic-basis-in-general} and \ref{thm:existence-triang-with-nice-matrix} since any two tagged triangulations of $\surf$ can be joined by a sequence of flips and flip is compatible with Fomin-Zelevinsky's matrix mutation, see \cite{FST}.
\end{proof}

\begin{ex}\label{ex:good-triangs-for-annuli} In Figure \ref{Fig:good_triangs_annuli} we can see triangulations of four different annuli.
        \begin{figure}[!h]
                \caption{}\label{Fig:good_triangs_annuli}
                \includegraphics[scale=.125]{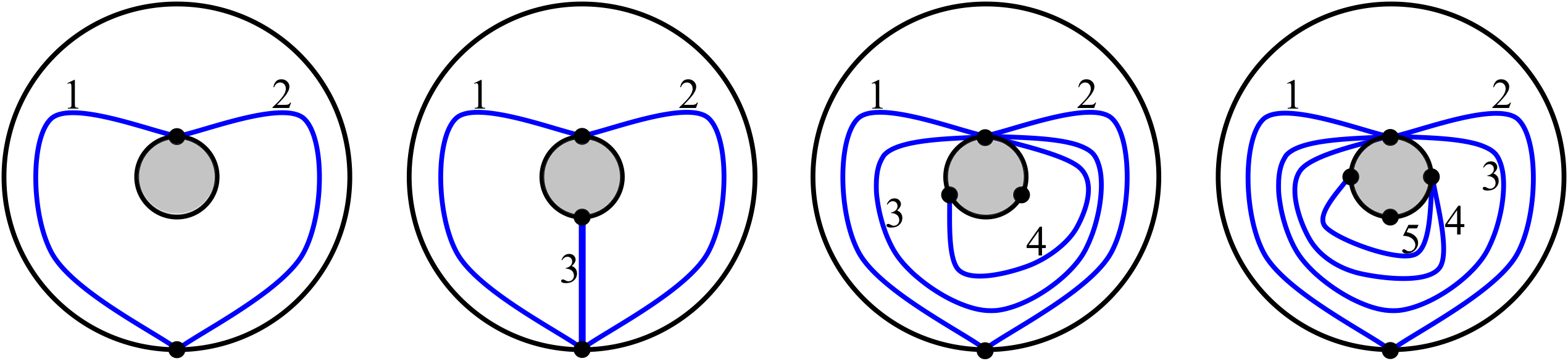}
        \end{figure}
The matrices of these triangulations are
$$
\left[\begin{array}{rr}0 & 2\\ -2 & 0\end{array}\right],
\left[\begin{array}{rrr}0 & 1 & 1 \\ -1 & 0 & -1\\ -1 & 1 & 0\end{array}\right],
\left[\begin{array}{rrrr}0 & 2 & -1 & 0\\ -2 & 0 & 1 & 0\\ 1 & -1 & 0 & 1\\ 0 & 0 & -1 & 0\end{array}\right],
\left[\begin{array}{rrrrr}0 & 2 & -1 & 0 & 0\\ -2 & 0 & 1 & 0 & 0\\ 1 & -1 & 0 &-1 & 0\\ 0 & 0 & 1 & 0 & 1\\ 0 & 0 & 0 & -1 & 0\end{array}\right].
$$
The first and the third are of course full-rank matrices. In the second matrix, equation \eqref{eq:positive-combination} becomes
$$
\mathbf{b}(\tau_1)_3 = \left[\begin{array}{r}1\\ -1\\ 0\end{array}\right]=\left[\begin{array}{r}1\\ 0\\ 0\end{array}\right]-\left[\begin{array}{r}0\\ 1\\ 0\end{array}\right]=\mathbf{e}_1-\mathbf{e}_2,
$$
while Proposition \ref{prop:sum-of-columns} takes the form $\mathbf{e}_1-\mathbf{e}_2=\mathbf{b}(\tau_1)_1+\mathbf{b}(\tau_1)_2$. Thus, $\mathbf{b}(\tau_1)_3=\mathbf{b}(\tau_1)_1+\mathbf{b}(\tau_1)_2$, as the reader could have checked directly in the $3\times 3$ matrix. In the fourth matrix, equation \eqref{eq:positive-combination} becomes
$$
\mathbf{b}(\tau_1)_5 = \left[\begin{array}{r}0\\ 0\\ 0 \\ 1 \\ 0\end{array}\right] = \left[\begin{array}{r}-1\\ 1\\ 0 \\ 1 \\ 0\end{array}\right] + \left[\begin{array}{r}1\\ 0\\ 0 \\ 0 \\ 0\end{array}\right]-\left[\begin{array}{r}0\\ 1\\ 0 \\ 0 \\ 0\end{array}\right]=\mathbf{b}(\tau_1)_3+\mathbf{e}_1-\mathbf{e}_2,
$$
while Proposition \ref{prop:sum-of-columns} takes the form $\mathbf{e}_1-\mathbf{e}_2=\frac{1}{2}\mathbf{b}(\tau_1)_1+\frac{1}{2}\mathbf{b}(\tau_1)_2$. Thus, $\mathbf{b}(\tau_1)_5=\mathbf{b}(\tau_1)_3+\frac{1}{2}\mathbf{b}(\tau_1)_1+\frac{1}{2}\mathbf{b}(\tau_1)_2$.
\end{ex}

\subsection{Spanning the Caldero-Chapoton algebra with coefficients}\label{sec:spanning-the-CC-algebra}

The following result has been proved recently by Fan Qin:

\begin{thm}\label{thm:Qin's-theorem}\cite{Qin} If the rectangular matrix $\widetilde{B}\in\ZZ^{(n+m)\times n}$ corresponding to $\widetilde{Q}$ has full rank, and the mutable part $Q$ of $\widetilde{Q}$ admits a reddening mutation sequence, then for any non-degenerate potential $\widetilde{S}\in\RA{\widetilde{Q}}$, $\calB_{(\widetilde{Q},\widetilde{S})}(Q,S)$ is a basis of the upper cluster algebra $\upper(\widetilde{B})$ over the ground ring $R=\Q[x_{n+1}^{\pm1},\ldots,x_{n+m}^{\pm1}]$.
\end{thm}

We will need the following slightly technical result.

\begin{lemma}\label{lemma:simultaneously-nondeg-potential} If $\widetilde{Q}$ is 2-acyclic, then there exists a non-degenerate potential on $\widetilde{Q}$ whose restriction to $Q$ turns out to be a non-degenerate potential for $Q_{\prin}$ as well.
\end{lemma}

\begin{proof} Let $\widetilde{Q}'$ be the quiver obtained from $\widetilde{Q}$ by adding $n$ vertices $1',2',\ldots,n'$ and by adding exactly one arrow $j\rightarrow j'$ for each $j\in[n]$. By Derksen-Weyman-Zelevinsky's celebrated theorem on the existence of non-degenerate potentials, \cite[Corollary 7.4]{DWZ1}, $\widetilde{Q}'$ admits a non-degenerate potential. Since restriction of QPs preserves non-degeneracy by \cite[Corollary 22]{LF1}, this potential does the job.
\end{proof}

Mills \cite{Mills1} has shown that all quivers associated to triangulations of surfaces with non-empty boundary admit maximal green mutation sequences, see also \cite[Section 4.9]{Keller-green-seqs}. This clearly implies the following:

\begin{thm}\label{thm:existence-of-reddening-sequence}\cite{Mills1} Let $\surf$ be a possibly-punctured surface with non-empty boundary. For any triangulation $\tau$ of $\surf$, the quiver $Q(\tau)$ admits a reddening mutation sequence.
\end{thm}

\begin{coro}\label{coro:gen-basis-is-basis-of-CC-alg-for-surfaces}  If $Q=Q(\tau)$ for some tagged triangulation $\tau$ of a surface $\surf$ with non-empty boundary, then there exists a non-degenerate potential $\widetilde{S}\in\RA{\widetilde{Q}}$ such that $\calB_{(\widetilde{Q},\widetilde{S})}(Q,S)$ is a basis of the Caldero-Chapoton algebra $\calCC_{(\widetilde{Q},\widetilde{S})}(Q,S)$ over the ground ring $R$.
\end{coro}

\begin{proof} By Lemma \ref{lemma:simultaneously-nondeg-potential}, there exists a potential $\widetilde{S}\in\RA{\widetilde{Q}}$ such that, setting $(Q,S):=(\widetilde{Q}|_{[n]},\widetilde{S}|_{[n]})$, both $(\widetilde{Q},\widetilde{S})$ and $(Q_{\operatorname{prin}},S)$ are non-degenerate

By Corollary \ref{coro:lin-ind-for-surfs-arbitrary-coeffs}, $\calB_{(\widetilde{Q},\widetilde{S})}(Q,S)$ is linearly independent over $R$.

By Theorem \ref{thm:existence-of-reddening-sequence}, the quiver $Q$ admits a reddening sequence. Theorem \ref{thm:Qin's-theorem} then implies that $\calB_{(Q_{\prin},S)}(Q,S)$ spans the upper cluster algebra $\upper(B_{\prin})$ over the ground ring $R_{\prin}$. Since the $R_{\prin}$-linear combinations of elements of $\calB_{(Q_{\prin},S)}(Q,S)$ belong to the Caldero-Chapoton algebra $\calCC_{(Q_{\prin},S)}(Q,S)$, this implies that $\calB_{(Q_{\prin},S)}(Q,S)$ spans $\calCC_{(Q_{\prin},S)}(Q,S)$ over $R_{\prin}$. Applying Theorem \ref{thm:spanning-set-in-princ-coeffs=>spanning-set-in-arbitrary-coeffs}, we deduce that $\calB_{(\widetilde{Q},\widetilde{S})}(Q,S)$ spans $\calCC_{(\widetilde{Q},\widetilde{S})}(Q,S)$ over the ground ring $R$.
\end{proof}

\subsection{$\calB_{(\widetilde{Q},\widetilde{S})}(Q,S)$ is a basis of the upper cluster algebra with arbitrary geometric coefficients}\label{sec:B-is-a-basis-of-upper-cluster-algebra}

\begin{thm}\label{thm:A=U-for-surfs-with-non-empty-boundary}\cite{Muller1,Muller2} If $B=B(\tau)$ for some tagged triangulation  $\tau$ of a surface with non-empty boundary $\surf$ with at least two marked points on the boundary, then $\mathcal{A}(\widetilde{B})=\upper(\widetilde{B})$ (see Definition \ref{def:precise-def-of-cluster-alg-and-upper-cluster-alg}).
\end{thm}

\begin{remark}\begin{enumerate}
\item In fact, Muller shows in~\cite[Theorem~10.6]{Muller1} that for $\tau$ as in Theorem \ref{thm:A=U-for-surfs-with-non-empty-boundary}, the coefficient-free cluster algebra
$\mathcal{A}(B(\tau))$ is locally acyclic with the help of his Banff algorithm
in the form of~\cite[Lemma~5.1]{Muller1}, as stated in~\cite[Lemma 10.3]{Muller1}.
In~\cite[Proposition~5.8]{Muller1} it is pointed out that in this situation $\mathcal{A}(\widetilde{B})$
is also locally acyclic. Finally, the main result of~\cite{Muller2} states that
a locally acyclic cluster algebra $\mathcal{A}(\widetilde{B})$ coincides with its
upper cluster algebra $\mathcal{U}(\widetilde{B})$.

\item In~\cite[Proposition~12]{CaLeSc} it is claimed, that in Theorem \ref{thm:A=U-for-surfs-with-non-empty-boundary} the
hypothesis that $\mathbb{M}$  contain at least two points in the boundary of
$\Sigma$ is not necessary as long as $\partial\Sigma\neq\varnothing$. Unfortunately,
we are not able to follow the (short) argument of the authors.
\end{enumerate}
\end{remark}

Combining Corollary \ref{coro:gen-basis-is-basis-of-CC-alg-for-surfaces} with Theorem \ref{thm:A=U-for-surfs-with-non-empty-boundary}, we obtain our fifth main result:

\begin{thm} For any rectangular matrix $\widetilde{B}\in\ZZ^{(n+m)\times n}$, whose top $n\times n$ submatrix is of the form $B(\tau)$ for some triangulation $\tau$ of a possibly punctured surface with at least two marked points on the boundary, $\calB_{(\widetilde{Q},\widetilde{S})}(Q,S)$ is a basis of the (upper) cluster algebra with coefficients $\mathcal{A}(\widetilde{B})=\upper(\widetilde{B})$ over the ground ring $R$.
\end{thm}

\begin{defi}
The set $\calB_{(\widetilde{Q},\widetilde{S})}$ is called \emph{generic basis} provided it is a basis of the corresponding (upper) cluster algebra with coefficients.
\end{defi}

We close the paper with a recent result of the authors. It follows easily from \cite{GLFS} that for any triangulation $\tau$ of an unpunctured surface $\surf$ different from the torus with exactly one marked point, the quiver $Q(\tau)_{\prin}$ admits exactly one non-degenerate potential up to right-equivalence. Denoting this potential by $S(\tau)$, it is obvious that $S(\tau)$ belongs to $\RA{Q(\tau)}$. It follows, again by \cite{GLFS}, that $(Q(\tau),S(\tau))$ is right-equivalent to the gentle quiver with potential associated in \cite{ABCP,LF1} to the triangulation $\tau$ of the unpunctured surface $\surf$.

\begin{thm}\cite[Theorem 11.9]{GLFS-schemes} Suppose that $\tau$ is a triangulation of an unpunctured surface $\surf$ with at least two marked points. In the coefficient-free surface cluster algebra $\mathcal{A}(B(\tau))$, and in the cluster algebra $\mathcal{A}(B(\tau)_{\prin})$ that has principal coefficients at $\tau$, the generic bases $\calB_{(Q(\tau),S(\tau))}(Q(\tau),S(\tau))$ and $\calB_{(Q(\tau)_{\prin},S(\tau))}(Q(\tau),S(\tau))$ coincide with Musiker-Schiffler-Williams' bangle bases \cite{MSWbases} (coefficient-free and with principal coefficients, respectively).
\end{thm}

\section*{Acknowledgements}

We owe thanks to Giovanni Cerulli Irelli, Bernhard Keller, Greg Muller, Gregg Musiker, Pierre-Guy Plamondon, Fan~Qin and Jon Wilson for helpful conversations. We also thank Tomoyuki Arakawa, Osamu Iyama, Atsuo Kuniba, Hiraku Nakajima, Tomoki Nakanishi and Masato Okado, organizers of the 3-week program \emph{Cluster algebras 2019}, (RIMS, Kyoto University, Japan), for the invitation to present the results in this paper during the second week of the program.

CG and DLF thank the Mathematical Institute of the University of Bonn for two weeks of hospitality in January 2020. DLF thanks Giovanni Cerulli Irelli and Sapienza-Universit\`{a} di Roma's Dipartimento SBAI for their hospitality during July 2019. JS thanks the Instituto de Matem\'aticas  of the UNAM (Mexico City) for two weeks of hospitality in September 2019. Parts of this work
were done during these~visits.

CG was supported by CONACyT's grant 239255. DLF gratefully acknowledges the support he received from  a \emph{C\'{a}tedra Marcos Moshinsky} and the grants CONACyT-238754 and PAPIIT-IN112519.
JS was partially funded by the Deutsche Forschungsgemeinschaft
(DFG, German Research Foundation) under Germany's Excellence Strategy - GZ 2047/1, Projekt-ID 390685813.


\end{document}